\documentclass[]{amsart} \usepackage{amsfonts}

\usepackage{bbding}
 \usepackage[dvips]{graphicx}
\usepackage{amsfonts,amssymb,amsmath,amsthm,amscd}

\usepackage{mathrsfs}
\usepackage{xcolor}
\usepackage{empheq}
\usepackage{adforn}
\usepackage{cancel}
\usepackage{xr}
\usepackage{mdframed}
\usepackage{tikz}
\usepackage{amsfonts}
\usepackage{mathdots}

\usepackage{enumerate}
\usepackage{lmodern}
\usepackage{mdframed}
\usepackage{stmaryrd}
\usepackage{blindtext}
\usepackage[all, knot]{xy}
\usepackage{leftidx}

\usepackage{accents}

\definecolor{slightblue}{rgb}{.8, .8, 1}
\definecolor{tif}{RGB}{10, 186, 181}
\definecolor{hair}{RGB}{100,225,190}
\definecolor{ruby}{RGB}{220,50,120}
\definecolor{grass}{RGB}{150,220,110}

\usepackage[colorlinks=true, pdfstartview=FitP, linkcolor=tif!85!black, citecolor=ruby, urlcolor=blue!80!black]{hyperref}

\usepackage{tgpagella}
\usepackage[T1]{fontenc}

\newtheorem{theorem}{Theorem}[section] 
\newtheorem*{theorem*}{Theorem}
\newtheorem{proposition}[theorem]{Proposition}
\newtheorem{lemma}[theorem]{Lemma}
 \newtheorem{corollary}[theorem]{Corollary}

\theoremstyle{definition} 
\newtheorem{definition}[theorem]{Definition}

\theoremstyle{remark} \newtheorem{remark}[theorem]{Remark} \numberwithin{equation}{section}
\numberwithin{figure}{section}

\newcommand{\N}{\mathcal{N}}

\newcommand{\Lap}[1]{\raisebox{-1pt}{\scalebox{1.15}{$\Delta$}}^{\!{#1}}}

\newcommand{\id}{\mathrm{id}}

\newcommand{\A}{\mathcal{A}}

\newcommand{\Diffeo}{\mathsf{Diffeo}}

\newcommand{\dev}{\mathsf{dev}}
\newcommand{\hol}{\mathsf{hol}}

\newcommand{\Aut}{\mathrm{Aut}}

\newcommand{\pt}{$\bullet$ }

\newcommand{\ve}[1]{{\boldsymbol{#1}}}
\newcommand{\T}{\mathsf{T}}
\newcommand{\dif}{\mathsf{d}}

\newcommand{\ac}[1]{{\boldsymbol{#1}}} %complex structure
\newcommand{\pa}{\partial}
\newcommand{\bpa}{{\bar{\partial}}}

\newcommand{\dz}{{\dif z}}
\newcommand{\dbz}{{\dif\bar z}}
\newcommand{\dx}{{\dif x}}

\newcommand{\paz}{\partial_z}
\newcommand{\pabz}{\partial_{\bar z}}
\newcommand{\pazbz}{\partial_{z\bar z}}

\newcommand{\ima}{\boldsymbol{i}} %imaginary unit

\newcommand{\GL}{\mathrm{GL}}
\newcommand{\SL}{\mathrm{SL}}

\newcommand{\eg}{\textit{e.g. }}
\newcommand{\cf}{\textit{c.f. }}
\newcommand{\ie}{\textit{i.e. }}
\newcommand{\etc}{\textit{etc.}}

\DeclareMathOperator{\ch}{ch}

\DeclareMathOperator{\Ad}{Ad}

\DeclareMathOperator{\im}{\mathsf{Im}}
\DeclareMathOperator{\re}{\mathsf{Re}}

\DeclareMathOperator{\U}{U}

\newcommand{\Fa}[1]{\mathsf{Farey}({#1})}

\newcommand{\Se}{\mathscr{S}}
\renewcommand{\H}{\mathbf{H}}
\newcommand{\CH}{\overline{\H}}
\newcommand{\fun}{f}
\renewcommand{\P}{\mathcal{P}}
\newcommand{\lap}{\Delta}
\newcommand{\tinf}{{t\rightarrow+\infty}}

\renewcommand{\k}{\ve{k}}
\newcommand{\dr}{\mathsf{r}}
\newcommand{\V}{\mathcal{W}}
\renewcommand{\dev}{\mathit{dev}}
\newcommand{\Dev}{\mathit{Dev}}
\renewcommand{\hol}{\mathit{hol}}

\newcommand{\C}{\mathscr{C}}
\renewcommand{\U}{\mathscr{U}}
\newcommand{\Q}{B}
\newcommand{\dzeta}{\dif\zeta}
\newcommand{\dbzeta}{\dif\bar\zeta}
\newcommand{\hyp}{\mathsf{hyp}}
\newcommand{\para}{\ve{T}}
\renewcommand{\sl}{\mathfrak{sl}}
\renewcommand{\a}{\ve{T}}

\newcommand{\Ac}{\mathsf{Accum}}

\newcommand{\dist}{\mathsf{dist}}
\newcommand{\devbd}[1]{\pa_{#1}\Omega}
\newcommand{\devlim}{\mathsf{Lim}}

\newcommand{\devlima}{\mathsf{Lim}_0}
\newcommand{\D}{\mathsf{D}}

\newcommand{\PGL}{\mathrm{PGL}}
 \newcommand{\gl}{\mathfrak{gl}}

\begin{document}

\title{Meromorphic cubic differentials and convex projective structures}
\author[Xin Nie]{Xin Nie} 
\address{School of Mathematics, KIAS
85 Hoegiro, Dongdaemun-gu, 
Seoul 02455,
Republic of Korea}
 \email{xinnie@mail.kias.re.kr}
% \thanks{The research leading to these results has received funding from the European Research Council under the {\em European Community}'s seventh Framework Programme (FP7/2007-2013)/ERC {\em grant agreement} ${\rm n^o}$ FP7-246918}

\maketitle

%\textbf{\red{This paper is in preparation, please do not cite or distribute.}}

\begin{abstract}
Extending the Labourie-Loftin correspondence, we establish, on any punctured oriented surface of finite type, a one-to-one correspondence between  convex projective structures with specific types of ends and punctured Riemann surface structures endowed with meromorphic cubic differentials whose poles are at the punctures. This generalises previous results of Loftin, Benoist-Hulin and Dumas-Wolf. 
\end{abstract}

\tableofcontents

\section{Introduction}

\subsection{Backgrounds} 
Let $\Sigma$ be an oriented surface other than the $2$-sphere.  A \emph{convex projective structure} on $\Sigma$ is given by a \emph{developing pair} $(\dev,\hol)$, where the holonomy $\hol$  is a representation of $\pi_1(\Sigma)$ in the group $\PGL(3,\mathbb{R})\cong\SL(3,\mathbb{R})$ of projective transformations  of $\mathbb{RP}^2$, whereas the \emph{developing map} $\dev$ is a $\hol$-equivariant diffeomorphism from the universal cover $\widetilde{\Sigma}$ to a bounded convex open subset $\Omega$ of an affine chart $\mathbb{R}^2\subset\mathbb{RP}^2$.

%Convex projective structures on manifolds of arbitrary dimensions are extensively studied, see \eg the surveys \cite{benoist_divisible, marquis}. However,  a good understanding of the moduli space of convex projective structures is available only for surfaces, \ie the case of dimension two, as we review now. 

Let $\mathcal{C}(\Sigma)$ denote the space of pairs $(\ac{J},\ve{b})$, where $\ac{J}$ is a complex structure on $\Sigma$ compatible with the orientation and $\ve{b}$ a holomorphic cubic differential on the Riemann surface $(\Sigma, \ac{J})$.
Results of Cheng-Yau \cite{cheng-yau_1, cheng-yau_2} and Wang \cite{wang} in affine differential geometry provide a natural map
\begin{equation}\label{eqn_corresp}
\P(\Sigma)\rightarrow \mathcal{C}(\Sigma).
\end{equation}
If $\Sigma$ is further assumed to be closed, Labourie  \cite{labourie_cubic} and Loftin \cite{loftin_amer} showed independently that the map (\ref{eqn_corresp}) is bijective, establishing a bijection between the quotients $\P(\Sigma)/\Diffeo^0(\Sigma)$ and $\mathcal{C}(\Sigma)/\Diffeo^0(\Sigma)$, where $\Diffeo^0(\Sigma)$ is the identity component of the diffeomorphism group of $\Sigma$. 
%Since $\mathcal{C}(\Sigma)/\Diffeo^0(\Sigma)$ is a holomorphic vector bundle of rank $5(g-1)$ over the Teichm\"uller space $\mathcal{T}(\Sigma)$, 
%When $\Sigma$ has genus $g\geq 2$, 
This recovers a previous result of Choi and Goldman \cite{choi-goldman} that $\P(\Sigma)/\Diffeo^0(\Sigma)$ is homeomorphic to $\mathbb{R}^{16(g-1)}$.

The purpose of the present paper is to extend the Labourie-Loftin bijection to open surfaces of finite type, \ie the case where $\Sigma$ is obtained from a closed oriented surface $\overline{\Sigma}$ by removing finitely many punctures. 

%In this case, it is perceived that an inverse to the map (\ref{eqn_corresp}) can exist only when we restrict to some particular classes of $(\ac{J}, \ve{b})\in\mathcal{C}(\Sigma)$, those such that the asymptotic behavior of $\ve{b}$ at punctures can be controlled.  
%More specifically, 
Let $\mathcal{C}_0(\Sigma)$ denote the space of $(\ac{J},\ve{b})\in\mathcal{C}(\Sigma)$ satisfying the following \mbox{requirements}. 

\begin{itemize}
\item The complex structure $\ac{J}$ is cuspidal at each puncture $p$, \ie a punctured neighborhood of $p$ is biholomorphic to $\{z\in\mathbb{C}\mid 0<|z|<1\}$ in such a way that points near $p$ correspond to points near $0$.
\item Each puncture is a pole or a removable singularity of $\ve{b}$. If $\Sigma$ has non-negative Euler characteristic (\ie homeomorphic to either $\mathbb{C}$, $\mathbb{C}^*$ or a torus), we further assume that $\ve{b}$ does not vanish identically.
\end{itemize}

Our main theorem says that the map (\ref{eqn_corresp}) restricts to a bijection between $\mathcal{C}_0(\Sigma)$ and a subset $\P_0(\Sigma)$ of $\P(\Sigma)$ consisting of convex projective structures 
with specific types of ends. We proceed to describe these ends in detail.

\subsection{The main theorem}
Fix a developing pair $(\dev,\hol)$ of a convex projective structure on $\Sigma$ and put $\Omega=\dev(\widetilde{\Sigma})$. Let $\devbd{p}$ denote the subset of $\pa\Omega$ ``corresponding to $p$'' (we will define this notion rigorously in \S \ref{sec_farey}). $\devbd{p}$  has infinitely many connected components, which are indexed by a transitive $\pi_1(\Sigma)$-set that we call \emph{Farey set} following \cite{fock-goncharov} and denote by $\Fa{\Sigma,p}$ (see \S \ref{sec_farey} for details), thus $\devbd{p}=\bigsqcup_{\tilde{p}\in\Fa{\Sigma,p}}\devbd{\tilde{p}}$. The generalise of $\tilde{p}\in\Fa{\Sigma,p}$ in $\pi_1(\Sigma)$ is generated by a loop $\gamma_{\tilde{p}}$ which goes around $p$ once and the holonomy $\hol_{\tilde{p}}:=\hol(\gamma_{\tilde{p}})\in\SL(3,\mathbb{R})$ preserves $\devbd{\tilde{p}}$.

%The following more precise description will be needed in the sequel.
%
%If $\Sigma$ has negative Euler characteristic,   the \emph{Farey set} of $p$ and denote it by $\Fa{p}$. Given $\tilde{p}\in\Fa{p}$,  take a decreasing sequence of horodisks $D_1\supset D_2\supset\cdots$ based at $\tilde{p}$ and converging to $\tilde{p}$.  The developed boundary $\devbd{\tilde{p}}$ is defined as the set of accumulation points of the sequence of subsets $(\dev(D_i))$ of $\Omega$. Also let $\hol_{\tilde{p}}\in\SL(3,\mathbb{R})$ denote the holonomy of the projective structure around $p$ which  preserves $\devbd{\tilde{p}}$. If $\Sigma$ has non-negative Euler characteristic, \ie homeomorphic to $\mathbb{C}$ or $\mathbb{C}^*$, the definition of developed boundaries is more straightforward since no choice of $\tilde{p}$ is involved.

We define $\P_0(\Sigma)$ as the set of convex projective structures such that each $\devbd{\tilde{p}}$ belongs to one of the following types. Here we employ the standard classification of projective automorphisms of properly convex sets into hyperbolic, quasi-hyperbolic, planar and parabolic ones (see \S \ref{sec_auto} below).
\begin{itemize}
\item
\textbf{A point}. In this case $\hol_{\tilde{p}}$ is parabolic and we call $p$ a \emph{cusp} of the convex projective surface.
Marquis \cite{marquis} proved that we are in this case if and only if a punctured neighborhood of $p$ has finite volume with respect to the Hilbert metric.
\item
\textbf{A segment}. In this case $\hol_{\tilde{p}}$ is either hyperbolic, quasi-hyperbolic or planar (the last one occurs only when $\Sigma$ is an annulus) and we call $p$ a \emph{geodesic end}.
%because, if we let $\Sigma'_p$ be the partial compactification of $\Sigma$ obtained by attaching a boundary circle to $p$, then the projective structure extends to $\Sigma'_p$ and has straight boundary.
\item
\textbf{A letter ``V''}, \ie two non-collinear segments sharing an endpoint. In this case $\hol_{\tilde{p}}$ is hyperbolic and the three endpoints of the two segments are the fixed points of $\hol_{\tilde{p}}$. We call such $p$ a \emph{V-end}. Choi \cite{choi_ends} studied ends of convex projective manifolds of arbitrary dimensions,  the present case corresponding to \emph{properly convex radial ends} in his classification.
\item
\textbf{A twisted $n$-gon}, \ie $\devbd{\tilde{p}}$ consists of consecutive segments $(Y_k)_{k\in\mathbb{Z}}$ such that the adjacent segments $Y_k$ and $Y_{k+1}$ are not collinear and $\hol_{\tilde{p}}$ maps $Y_k$ to $Y_{k+n}$ for every $k$. In this case we call $p$ a \emph{broken geodesic end} with $n$-pieces. 

\end{itemize}

%structure of Fr\'echet manifold...

Our main result is the following theorem. Here we recall that the residue $R$ of a cubic differential $\ve{b}$ at a pole $p$ of order at most $3$ is the coefficient of $z^{-3}\dz^3$ in the Laurent expansion of $\ve{b}$ with respect to a conformal local coordinate $z$ centered at $p$.  $R$ does not depend on the choice of the coordinate (which is  false when the order exceeds $3$)
\begin{theorem}\label{intro_main}
Let $\Sigma$ be a punctured oriented surface of finite type. Then
\begin{enumerate}
\item\label{item_intro1}
The natural map (\ref{eqn_corresp}) restricts to a bijection from $\P_0(\Sigma)$ to $\mathcal{C}_0(\Sigma)$.

\item\label{item_intro2}
Given a convex projective structure in $\P_0(\Sigma)$, at each puncture, the type of end of the convex projective structure is determined by  the type of singularity of the corresponding cubic differential as in the following table. Here $R$ denotes the residue of the cubic differential
\vspace{4pt}

\begin{tabular}{|c|c|}
\hline
\emph{type of end of the}& \emph{type of singularity of}\\
\emph{projective structure}&\emph{the cubic differential}\\ \hline
cusp&removable singularity or\\ 
&pole of order at most $2$\\ \hline
geodesic end with planar&third order pole\\
 or quasi-hyperbolic holonomy&with $R\in\ima\mathbb{R}^*$\\
\hline
geodesic end with & third order pole\\
hyperbolic holonomy&with $\re(R)>0$ \\
 \hline
{V-end}&third order pole\\ 
&with $\re(R)<0$\\
\hline
{broken geodesic end with $n$-pieces}&{pole of order $n+3$}\\[3pt]
\hline
\end{tabular}
\end{enumerate}
\end{theorem}
\begin{remark}
In the second case above, planar holonomy occurs only when the convex projective surface is an annulus obtained as the quotient of a triangle in $\mathbb{RP}^2$ by a planar projective transformation. In this case the corresponding $(\ac{J},\ve{b})\in \mathcal{C}_0(\Sigma)$ is given by
$((\Sigma,\ac{J}),\ve{b})\cong(\mathbb{C}^*, Rz^{-3}\dz^3$). See Lemma \ref{lemma_planar} below. 
\end{remark}

The above theorem is a generalisation of many previous works. Namely,
\begin{itemize}
\item The restriction of Theorem \ref{intro_main} to the subspace of $\mathcal{C}_0(\Sigma)$ consisting of those $(\ac{J},\ve{b})$ where $\ve{b}$ only has removable singularities or poles of order at most $2$ (resp. $3$) is proved by Benoist and Hulin \cite{benoist-hulin} (resp. Lofin \cite{loftin_compactification, loftin_neck}).

\item Dumas and Wolf \cite{dumas-wolf} proved Theorem \ref{intro_main} for $\Sigma=\mathbb{R}^2$. In this case, an element in $\mathcal{C}_0(\Sigma)$ is a complex structure biholomorphic to $\mathbb{C}$ together with a polynomial cubic differential, whereas an element in $\P(\Sigma)$ is a diffeomorphism from $\mathbb{R}^2$ to a properly convex polygon $\Omega\subset\mathbb{RP}^2$. 
%In this case, Dumas and Wolf showed that (\ref{eqn_corresp}) restricts to a bijection between $\mathcal{C}_0(\mathbb{R}^2)$ and convex projective structures in which $\Omega$ is a polygon.
\end{itemize}

Loftin \cite{loftin_compactification} also showed that the residue of $\ve{b}$ at a third order pole determines the eigenvalues of the holonomy of the convex projective structure around that pole. We can state his result as follows.
\begin{theorem}[Loftin]\label{intro_thm2}
Let $(\ac{J},\ve{b})\in\mathcal{C}_0(\Sigma)$. Equip $\Sigma$ with the convex projective structure corresponding to $(\ac{J},\ve{b})$ given by Theorem \ref{intro_main}. Let $p$ be a third order pole of $\ve{b}$ with residue $R$.  Then the eigenvalues of the holonomy of the projective structure around $p$ are $\exp({-4\pi \mu_i})$ ($i=1,2,3$) where $\mu_1, \mu_2,\mu_3\in\mathbb{R}$ are the imaginary parts of the three cubic roots of $R/2$, respectively.
\end{theorem}

Theorem \ref{intro_thm2} and many other results in \cite{loftin_compactification} are re-obtained as byproducts of our proof of Theorem \ref{intro_main}.

Around poles of order $\geq 4$, it does not seem possible to capture the holonomy from local information of the cubic differential. Indeed, when the order $r$ is not divisible by $3$, there always exists a conformal local coordinate $z$ in which $\ve{b}=z^{-r}\dz^3$, nevertheless the holonomy can be quite arbitrary.

\subsection{How twisted polygons arises}\label{intro_deve}
Before discussing the proof Theorem \ref{intro_main}, it might be enlightening to briefly describe here how a twisted polygon arises from a higher order pole of a cubic differential. 

Take a punctured Riemann surface $\Sigma=\overline\Sigma\setminus P$, where $\overline\Sigma$ is a closed Riemann surface and $P\subset\overline{\Sigma}$ a finite set. Fix a meromorphic cubic differential $\ve{b}$ on $\overline\Sigma$ with poles in $P$ and let $\dev:\widetilde{\Sigma}\rightarrow\mathbb{RP}^2$ be a developing map of the convex projective structure on $\Sigma$ corresponding to $\ve{b}$ provided by Theorem \ref{intro_main}. 

Let $p\in P$ be a pole of order $n+3$ with $n\geq1$ and fix $\tilde{p}\in\Fa{\Sigma,p}$. Before giving a precise statement in Theorem \ref{intro_how} below which explains why $\devbd{\tilde{p}}$ is a twisted $n$-gon, let us introduce some more notations. Details of the notions and claims below will be presented in \S \ref{sec_farey}.

\pt Let $\mathscr{P}_{m,p}$ denote the set of all parametrized paths $\beta:[0,+\infty)\rightarrow\Sigma$ such that $\beta(0)=m$ and $\beta(t)\rightarrow p$ as $t\rightarrow+\infty$.   The Farey set $\Fa{\Sigma,p}$ is naturally identified with the set of homotopy classes of elements in $\mathscr{P}_{m,p}$.

\pt $\beta\in\mathscr{P}_{m,p}$ develops into a path $\beta^\dev:[0,+\infty)\rightarrow\Omega$. The limit points of $\beta^\dev(t)$ as $t\rightarrow+\infty$ are in $\devbd{p}\subset\pa\Omega$. 

\pt If  $\beta\in\mathscr{P}_{m,p}$ is in the homotopy class $\tilde{p}\in\Fa{\Sigma,p}$, the limit points of $\beta^\dev$ are correspondingly in the connected component $\devbd{\tilde{p}}$ of $\devbd{p}$.  When there is a single limit point, \ie  $\lim_{t\rightarrow+\infty}\beta^\dev(t)$ exits, we denote it by $\devlim(\beta)$.

\vspace{6pt}

Of particular importance are those $\beta\in\mathscr{P}_{m,p}$ in the homotopy class $\tilde{p}$ such that $\ve{b}(\dot{\beta}(t))=-1$ when $t$ is large enough. 
Let $\mathscr{N}\subset\mathscr{P}_{m,p}$ denote the set of all such $\beta$. More generally, we call the (oriented) trajectory of a path $t\mapsto\beta(t)\in\Sigma$ satisfying $\ve{b}(\dot{\beta}(t))\in\mathbb{R}_-$  a \emph{negative trajectory}.

In order to better understand negative trajectories tending to $p$, we identify the tangent space $\T_p\overline{\Sigma}$ with $\mathbb{C}$ by means of a local coordinate of $\overline\Sigma$ centered at $p$ under which the leading coefficient in the Laurent expansion of $\ve{b}$ is positive. Then a negative trajectory tending to $p$ must be asymptotic to a ray in $\T_p\overline{\Sigma}$ of the form
$$
N_k:=e^{2\pi \ima k/n}\mathbb{R}_{\geq 0}\subset\mathbb{C}\cong\T_p\overline\Sigma,
$$
where $k\in\mathbb{Z}/n\mathbb{Z}$. One easily checks this fact for the example $(\Sigma, \ve{b})=(\mathbb{C},\dz^3)$ at the puncture $p=\infty$, where a negative trajectory is just an oriented line towards the direction $e^{\pi\ima/3}$, $-1$ or $e^{5\pi\ima/3}$.

We call $\beta_1,\beta_2\in\mathscr{N}$ \emph{equivalent} if there is $t_0\in\mathbb{R}$ such that $\beta_1(t+t_0)=\beta_2(t)$ when $t$ is large enough. This is an equivalence relation which does not affects $\devlim(\beta)$. Let $\mathcal{N}$ denote the set of equivalence classes. In the above $(\mathbb{C},\dz^3)$ example, $\N$ is the union of three one-parameter families, corresponding to the three directions. In general, when $\Sigma$ is not simply connected, $\N$ consists of countably many one-parameter families $\N_k$ ($k\in\mathbb{Z}$) such that elements in $\N_k$ are asymptotic to the direction $N_k$, whereas elements in $\N_k$ and in $\N_{k+n}$ are distinguished by how many times they wrap around $p$, see \S \ref{sec_geq} for details.
The following theorem is a part of Theorem \ref{thm_devhigher} and essentially implies that $\devbd{\tilde{p}}$ is a twisted $n$-gon.
\begin{theorem}\label{intro_how}
The limit $\devlim(\beta)$ exists for any $\beta\in\mathcal{N}$. For each $k\in\mathbb{Z}$, the set of limits $Y^\circ_{k}:=\{\devlim(\beta)\}_{\beta\in\N_{k}}$ is the interior of  a segment $Y_k\subset\mathbb{RP}^2$. 
\end{theorem}

The convex projective structure corresponding to $(\Sigma,\ve{b})=(\mathbb{C}, 2\,\dz^3)$ provides an explicit example for the above theorem. A developing map of this projective structure is 
$$
\dev_0(z)=[e^{2\re(z)}:e^{2\re(\omega^2 z)}: e^{2\re(\omega z)}] \quad (\omega:=e^{2\pi\ima/3}).
$$
The image  $\dev_0(\mathbb{C})$ is the triangle $\Delta=\{[x_1:x_2:x_3]\mid x_i>0\}$. 
%(see \S \ref{sec_titeica} for more discussions on this example). 
The following picture shows the images of some negative trajectories. Theorem \ref{intro_how} can be read off from the picture.
\begin{figure}[h]
\centering
\includegraphics[width=3.6in]{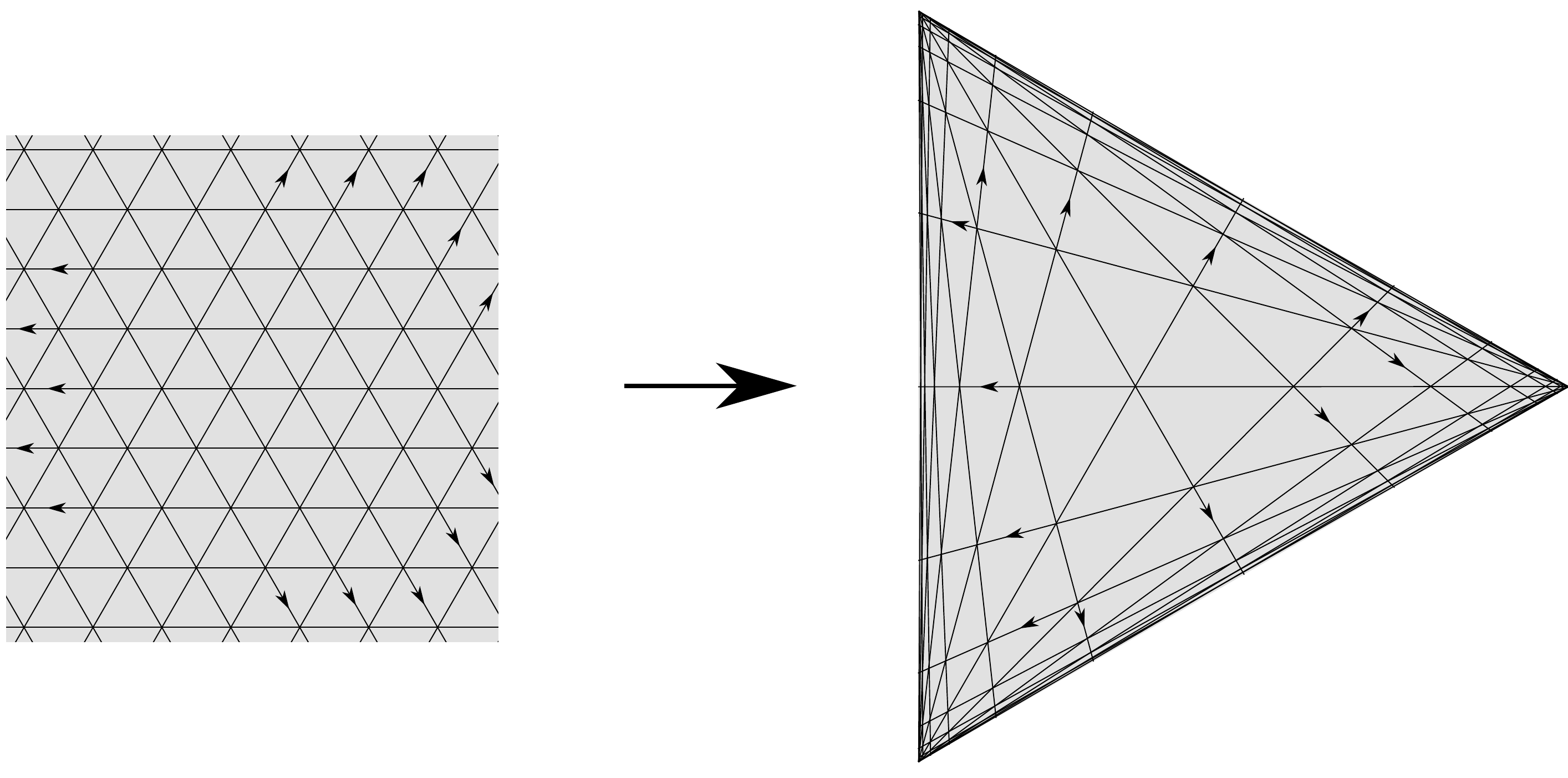}
\caption{Images of negative trajectories by $\dev_0$.}
\label{figure_dev01}
\end{figure}

On the other hand, the next picture illustrates the fact that if the asymptotic direction of $\beta\in\mathscr{P}_{m,p}$ at $p$ ``lies between'' $\N_k$ and $\N_{k+1}$, then $\devlim(\beta)$ is the common vertex of the edges $Y_k$ and $Y_{k+1}$.
This will be made precise in Theorem \ref{thm_devhigher}.
\begin{figure}[h]
\centering
\includegraphics[width=3.75in]{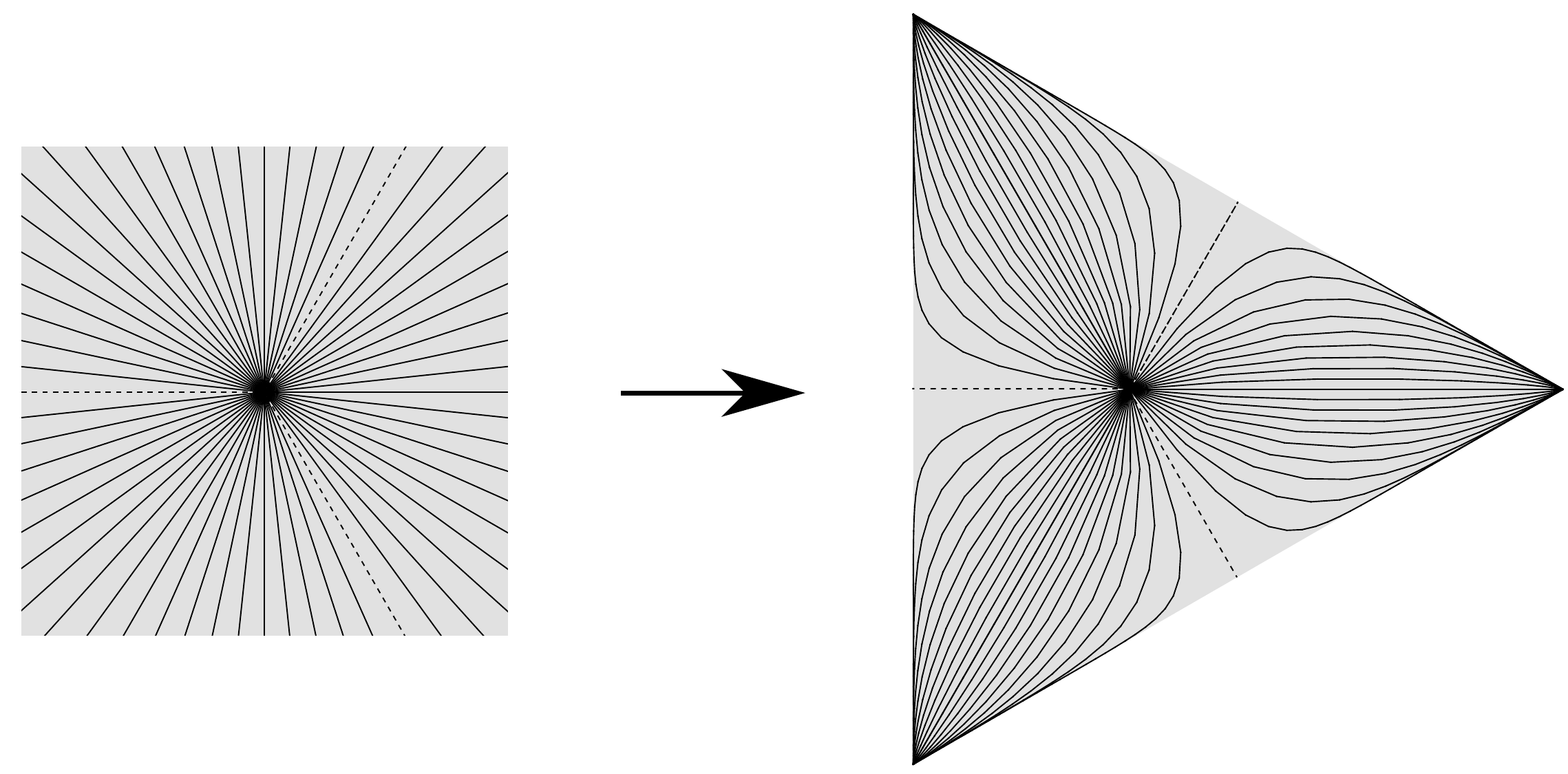}
\caption{Images of rays issuing from $0$ by $\dev_0$.}
\label{figure_dev02}
\end{figure}

When $p$ is a third order pole, similar descriptions of the $\devlim(\beta)$'s are given in Theorem \ref{thm_pole3}.

\subsection{Sketch of proof}\label{sec_sketch}
Theorem \ref{intro_main} is proved by  assembling  and adapting various techniques developed by Loftin \cite{loftin_compactification}, Benoist-Hulin \cite{benoist-hulin} and Dumas-Wolf \cite{dumas-wolf}. In order to give a sketch, we first recall the construction of the natural map (\ref{eqn_corresp}).  

Let $\V(\Sigma)$ denote the space of pairs $(g,\ve{b})$ where $g$ is a complete Riemannian metric on $\Sigma$ and $\ve{b}$ a holomorphic cubic differential (with respect to the conformal structure of $g$) satisfying \emph{Wang's equation}
\begin{equation}\label{intro_vortex}
\kappa_g=-1+2\|\ve{b}\|^2_g.
\end{equation}
Here $\kappa_g$ is the curvature of $g$ and $\|\ve{b}\|_g$ is the pointwise norm of $\ve{b}$ with respect to $g$.

The map (\ref{eqn_corresp}) is defined as a composition
$$
\P(\Sigma)\overset\sim\rightarrow\V(\Sigma)\rightarrow\mathcal{C}(\Sigma)
$$
where 

\pt 
$\V(\Sigma)\rightarrow\mathcal{C}(\Sigma)$ is the forgetful map,  sending $(g,\ve{b})$ to $(\ac{J},\ve{b})$ where $\ac{J}$ is the complex structure given by $g$ and the orientation of $\Sigma$.

\pt The bijection $\P(\Sigma)\overset\sim\rightarrow\V(\Sigma)$ is given by assigning to each convex projective structure an \emph{hyperbolic affine sphere structure} then taking its \emph{Blaschke metric}  $g$ and \emph{Pick differential} $\ve{b}$.

\pt The inverse $\V(\Sigma)\overset\sim\rightarrow \P(\Sigma)$ to the above bijection 
is provided by integrating a flat connection on the vector bundle $\T\Sigma\oplus\underline{\mathbb{R}}$ (where $\underline{\mathbb{R}}$ denotes the trivial line bundle) associated to each $(g,\ve{b})\in \V(\Sigma)$. We refer to the so-obtained developing map \emph{Wang's developing pair} after \cite{loftin_amer}. 
The flat vector bundle is actually the real form of a $\SL(3,\mathbb{R})$-Higgs bundle \cite{hitchin} (see \cite{labourie_cubic}).

\vspace{5pt}

%Given a cubic differential $\ve{b}=b(z)\dz^3$, we let $|\ve{b}|^\frac{2}{3}:=|b(z)|^\frac{2}{3}|\dz|^2$ denote the associated flat metric with conic singularities at the zeros.  
Let $\V_0(\Sigma)$ denote the set of $(g,\ve{b})\in\V(\Sigma)$ satisfying the following conditions:
\begin{itemize}
\item[-] $(g,\ve{b})$ is sent into $\mathcal{C}_0(\Sigma)$ by the forgetful map $\V(\Sigma)\rightarrow\mathcal{C}(\Sigma)$;
\item[-]  $\kappa_g(x)$ tends to $-1$ as $x$ tends to a removable singularity or a pole of order at most $2$.
\item[-] $\kappa_g(x)$ tends to $0$ as $x$ tends to a pole of order $\geq 3$.
\end{itemize}

%\begin{itemize}
%\item[-] $(g,\ve{b})$ is sent into $\mathcal{C}_0(\Sigma)$ by the forgetful map $\V(\Sigma)\rightarrow\mathcal{C}(\Sigma)$;
%\item[-]  $g$ and  $|\ve{b}|^\frac{2}{3}$ are \emph{conformally quasi-isometric} around poles of order $\geq 3$, \ie their conformal ratio has positive upper and lower bounds around these poles.
%\item[-] $\kappa_g$  has negative upper and lower bounds around each removable singularitie or pole of order $\leq 2$.
%\end{itemize}

We obtain Theorem \ref{intro_main} as a combination of the following independent results.
 
 \vspace{5pt}
 
\textbf{(I) $\P(\Sigma)\overset\sim\rightarrow\mathcal{W}(\Sigma)$ maps $\P_0(\Sigma)$ into $\V_0(\Sigma)$.} More specifically, if $p$ is either a cusp, a geodesic end, a V-end or a twisted polygonal end of a convex projective structure, then the conformal structure underlying the Blaschke metric $g$ is cuspidal at $p$ whereas the Pick differential $\ve{b}$ is meromorphic at $p$, and furthermore, $g$ satisfies the above defining conditions of $\V_0(\Sigma)$ at $p$.

For cusps, this is proved in \cite{benoist-hulin}. We will give a proof for other types of ends by generalizing a result from \cite{dumas-wolf}. 
The main tool  is the Hausdorff continuity property of Blascke metrics and Pick differentials discovered in \cite{benoist-hulin}.

\vspace{5pt}

\textbf{(II) $\V_0(\Sigma)\rightarrow\mathcal{C}_0(\Sigma)$ is bijective}. In other words, for any $(\ac{J},\ve{b})\in\mathcal{C}_0(\Sigma)$, there exists a unique  metric $g$ in the conformal class of $\ac{J}$ satisfying Wang's equation (\ref{intro_vortex}) and the above defining conditions of $\V_0(\Sigma)$. 

This is the Koebe-Poincar\'e uniformization theorem if $\ve{b}=0$ and $\Sigma$ has negative Euler characteristic, so we can assume that $\ve{b}$ does not vanish identically.  

We will give a proof of the above statement following the standard approach as in \cite{loftin_compactification, benoist-hulin, dumas-wolf}, using the method of barriers to establish existence and the Yau-Omori maximum principle to prove uniqueness.

\vspace{5pt}

\textbf{(III)} \textbf{The correspondence between the types of ends and the types of singularities in Theorem \ref{intro_main} holds.} At poles of other $\geq 4$, this essentially follows from the results outlined in \S \ref{intro_deve} above.  The proof consists in comparing the developing map of the convex projective structure corresponding to $(g,\ve{b})$ with an ``auxiliary developing map'' $\dev_0: \widetilde{\Sigma}\rightarrow\mathbb{RP}^2$, although in general the metric $2^\frac{1}{3}|\ve{b}|^\frac{2}{3}$ is singular and $\dev_0$ does not define a projective structure. The map $\dev_0$ can be fully understood by virtue of its explicit expressions. The twisted polygon is found by comparing the actual developing map $\dev$ with $\dev_0$ using ODE techniques from \cite{dumas-wolf}.

We then apply the same method to tackle the case of third order poles. Many results that we obtained in this case have already appeared more or less inexplicitly in \cite{loftin_compactification}. We also refine the ODE technique to obtain informations on ``unstable directions'' as singled out in \cite{dumas-wolf}, \ie directions in $\T_p\overline{\Sigma}$ along which $\dev_0$ and $\dev$ are not comparable.

For poles of order $\leq 2$, the required result is simpler and is already proved in \cite{loftin_compactification}, using ODE techniques as well.  There is also a new proof in \cite{benoist-hulin}. We will review these proves for completeness of the paper.

\vspace{5pt}

Once statements (I), (II) and (III) are proved, Theorem \ref{intro_main} follows immediately. Indeed, we obtain two bijections $\P_0(\Sigma)\overset\sim\rightarrow \V_0(\Sigma)\overset\sim\rightarrow \mathcal{C}_0(\Sigma)$, the first one provided by (I) and (III) and the second by (II).

\subsection{Organization of the paper}
In Section \ref{sec_pre} we review in detail the construction of the natural map (\ref{eqn_corresp}) mentioned above. In Section \ref{sec_devpunc}, ends of convex projective structures are discussed, where we carefully define the developed boundary $\devbd{\tilde{p}}$ and related notions, and establish some fundamental properties. Section \ref{sec_model} describes a local model for a cubic differential around a pole of order $\geq 4$. In Section \ref{sec_ptoc}, \ref{sec_cv} and \ref{sec_ctop} we prove assertions (I), (II) and (III) in the above sketch, respectively.  Section \ref{sec_ctop} also contains a proof of Theorem \ref{intro_thm2}.

\subsection{Acknowledgements}
We would like to thank David Dumas and Michael Wolf for helpful conversations and correspondences, and thank \mbox{Alexandre} Eremenko for pointing out the reference \cite{huber}.

\section{Review of the identification $\P(\Sigma)\cong\V(\Sigma)$}\label{sec_pre}

Let $\Sigma$ be a surface obtained from a connected closed oriented $C^\infty$ surface $\overline{\Sigma}$ by removing a finite (possibly empty)  set of punctures. Furthermore, $\Sigma$ is assumed not to be the $2$-sphere. 

In this section, we first give a concise review of the natural bijection between the space $\P(\Sigma)$ of convex projective structures and the space $\V(\Sigma)$ of pairs $(g,\ve{b})$ satisfying Wang's equation. As we have seen in the introduction, this bijection lies at the foundation of the proof of Theorem \ref{intro_main}. We then investigate the particular pair $(g,\ve{b})=(2|\dz|^2, 2\,\dz^3)$ for later use.

Besides the seminal work \cite{loftin_amer, labourie_cubic}, the theory relating convex projective structures, affine spheres and cubic differentials is surveyed in many subsequent papers such as \cite{loftin_compactification, benoist-hulin, dumas-wolf}.
Here we merely aim at covering the notions and results used later on as quickly as possible. The reader is referred to the above cited papers whenever a detailed explanation or proof is in need.

\subsection{From $\P(\Sigma)$ to $\V(\Sigma)$: Blaschke metrics and Pick differentials}\label{sec_ptov}

\subsubsection{Support functions}
Given a bounded convex open set $\Omega\subset \mathbb{R}^2$, a \emph{support function} for $\Omega$ is by definition a strictly convex function $u\in C^0(\overline{\Omega})\cap C^\infty(\Omega)$ which solves the following boundary value problem of (Monge-Amp\`ere type) non-linear PDE
$$
\det(\mathsf{Hess}(u))=u^{-4}, \quad u|_{\pa\Omega}=0.
$$
Cheng and Yau \cite{cheng-yau_1} ensured existence and uniqueness of the solution.

The significance of support functions lies in the fact that the above PDE is the condition for the hypersurface
$$
S=\left\{-\frac{1}{u(x)}
(x, 1)
\right\}_{ x\in \Omega}\subset \mathbb{R}^3
$$ 
to be an \emph{hyperbolic affine sphere}, whereas the vanishing boundary condition means that $S$ is asymptotic to the convex cone $\mathbb{R}_+\cdot\{(x,1)\}_{x\in \Omega}$. Thus the Cheng-Yau theorem establishes unique-existence of hyperbolic affine sphere asymptotic to a convex cone.

The following Hausdorff continuity result on $u_\Omega$ is due to Benoist and Hulin (see \cite{benoist-hulin} Corollary 3.3 and \cite{dumas-wolf} Theorem 4.4).
\begin{theorem}\label{thm_bh}
Let $\Omega\subset\mathbb{R}^2$ be a bounded convex open set,  $K\subset \Omega$ be a compact subset and let $\varepsilon>0$. Then for any  $k\in\mathbb{N}$ and any convex open set $\Omega'\subset\mathbb{R}^2$ sufficiently close to $\Omega$ in the Hausdorff topology, we have $K\subset \Omega'$ and 
$$
\|u_\Omega-u_{\Omega'}\|_{K,k}<\epsilon.
$$
Here $\|\cdot\|_{K,k}$ denote the $C^k$-norm on $K$.
\end{theorem}

\subsubsection{Blaschke metrics and Pick differentials}\label{sec_canometric}
Any oriented properly convex open set $\Omega\subset\mathbb{RP}^2$ carries a canonical complete Riemannian metric $g_\Omega$, called the \emph{Blaschke metric},  and a holomorphic \mbox{cubic} differential $\ve{b}_\Omega$ (with respect to the conformal structure underlying $g_\Omega$ and compatible with the orientation), called the \emph{Pick differential}. 
The Blaschke metric is actually independent of the orientation, while Pick differentials resulting from opposite orientations are complex-conjugates of each other. We let $\nu_\Omega$ denote the volume form of $g_\Omega$.

The precise definitions of $g_\Omega$ and $\ve{b}_\Omega$ are not so important for our purpose. We only mention here that they are affine invariants of the hyperbolic affine sphere asymptotic to a cone in $\mathbb{R}^3$ generated by $\Omega$. More important to us are the following properties.
\begin{itemize}
\item  $g_\Omega$ and $\ve{b}_\Omega$ satisfy Wang's equation $\kappa_{g_\Omega}=-1+2\|\ve{b}\|^2_{g_\Omega}$. Here $\kappa_{g_\Omega}$ is the curvature of $g_\Omega$ and $\|\ve{b}\|_{g_\Omega}$ is the pointwise norm of $\ve{b}_\Omega$ with respect to $g_\Omega$.

\item
If we view $\Omega$ as a bounded subset of an affine chart $\mathbb{R}^2\subset\mathbb{RP}^2$ and let $u=u_\Omega$ be the support function, then the coefficients of $g_\Omega$ and $\ve{b}_\Omega$ are linear combinations of $u$ and its derivatives of order up to $3$. Indeed, we have
$$
g_\Omega=-\frac{1}{u}\sum_{i,j=1,2}\pa_i\pa_ju\,\dx^i\dx^j,\quad %\nu_\Omega=-u_\Omega^{-3}\dx^1\wedge\dx^2
$$
while $\ve{b}_\Omega$ is basically the difference between the Levi-Civita connection of $g_\Omega$ and the \emph{Blaschke connection} $\nabla$ associated to the affine sphere, given by
$$
\nabla=\dif-\frac{1}{u}
\begin{pmatrix}
\dif u+\pa_1u\,\dif x^1&\pa_2u\,\dif x^1\\
\pa_1u\,\dif x^2&\dif u+\pa_2u\,\dif x^2
\end{pmatrix}.
$$
\item $g_\Omega$ and $\ve{b}_\Omega$ are invariant by projective transformations in the sense that for any $a\in\SL(3,\mathbb{R})$ we have $
g_{a(\Omega)}=a_*g_\Omega$ and  $\ve{b}_{a(\Omega)}=a_*\ve{b}_\Omega\,.
$
Here the orientation on $a(\Omega)$ is induced from the one on $\Omega$ through $a$.
\item $\Omega$ is a triangle if and only if 
$$
(\Omega, g_\Omega,\ve{b}_\Omega)\cong (\mathbb{C}, 2|\dz|^2,\,2\,\dz^3),
$$ 
whereas $\Omega$ is an ellipse if and only if 
$$
(\Omega, g_\Omega,\ve{b}_\Omega)\cong (\mathbb{D}, \,g_\hyp,\, 0).
$$ Here $(\mathbb{D},\,g_\hyp)$ is the Poincar\'e disk with the hyperbolic metric.

\end{itemize}

\subsubsection{Hilbert metrics}

A properly convex set $\Omega\subset\mathbb{RP}^2$ also carries a natural Finsler metric $g^\mathsf{H}_\Omega$, called the \emph{Hilbert metric}. By definition, in an affine chart $\mathbb{R}^2$ containing $\Omega$, we have
$$
g^\mathsf{H}_\Omega(v)=\left(\frac{1}{|x-a|}+\frac{1}{|x-b|}\right)|v|,\quad \forall x\in\Omega,\,v\in\T_{x}\Omega\cong\mathbb{R}^2,
$$
where $a,b\in\pa\Omega$ are the intersections of $\pa\Omega$ with the line passing through $x$ along the direction $v$ and $|\cdot|$ is a Euclidean norm on $\mathbb{R}^2$. Let $\nu^\mathsf{H}_\Omega$ denote the volume form of $g^\mathsf{H}_\Omega$.

The Hilbert metric is invariant by projective transformations in the same sense as how $g_\Omega$ and $\ve{b}_\Omega$ are.

\subsubsection{Continuity}

Let $\mathfrak{C}$ denote the space of all properly convex open sets in $\mathbb{RP}^2$, equipped with the Hausdorff topology. Let $\mathfrak{C}_*$ denote the topological subspace of $\mathfrak{C}\times \mathbb{RP}^2$ consisting of pairs $(\Omega, x)$ such that $x\in\Omega$.

The Benz\'ecri Compactness Theorem (see \cite{benoist-hulin} Theorem 2.7) says that the quotient of $\mathfrak{C}_*$ by the natural action of $\SL(3,\mathbb{R})$ is compact. Combining this with Theorem \ref{thm_bh}, we get the following  consequences. 
%They are our main tools in Section \ref{sec_ptoc} below.
\begin{corollary}[\textbf{Hausdorff continuity}]\label{coro_bh}
${}$
\begin{enumerate}
\item\label{item_bh1}
The proportion between the Hilbert volume form and the Blaschke volume form, viewed as a map
$$
\mathfrak{C}_*\rightarrow\mathbb{R}_+,\quad (\Omega,x)\mapsto \frac{\nu_\Omega^\mathsf{H}}{\nu_\Omega}(x),
$$
is continuous and is bounded from above and below by positive constants.
\item\label{item_bh2}
On any properly convex open set $\Omega$ we define the function 
$$
f_\Omega:=2\|\ve{b}_\Omega\|_{g_\Omega}^2=\kappa_{g_\Omega}+1: \Omega\rightarrow\mathbb{R}_{\geq 0}.
$$ 
Then the map
$$
\mathfrak{C}_*\rightarrow\mathbb{R}_{\geq 0},\quad (\Omega, x)\mapsto f_{\Omega}(x)
$$
is continuous and is bounded from above.
\end{enumerate}
\end{corollary}

\begin{proof}[Outline of proof]
We only need to prove that these maps are continuous, then the bounds follow from Benz\'ecri Compactness theorem. But $(\Omega,x)\mapsto\nu_\Omega^\mathsf{H}(x)$ is continuous by definition, while $(\Omega,x)\mapsto\nu_\Omega^\mathsf{H}$,
 $(\Omega, x)\mapsto\ve{b}_\Omega(x)$ and $(\Omega, x)\mapsto g_\Omega(x)$
are continuous by Theorem \ref{thm_bh} and the fact that $\nu_\Omega$, $g_\Omega$ and $\ve{b}_\Omega$ can be expressed in terms of $u_\Omega$ and its derivatives. Continuity of the required maps follows.
\end{proof}

\subsubsection{Convex projective structures}\label{sec_pscs}
%\begin{definition}[\textbf{Projective structures}]\label{def_clas}
A \emph{projective structure} on $\Sigma$ is an equivalence class of \emph{developing pairs} $(\dev, \hol)$, where 
\begin{itemize}
\item The \emph{holonomy} $\hol$ is a homomorphism from $\pi_1(\Sigma)$ to $\PGL(3,\mathbb{R})\cong\SL(3,\mathbb{R})$.

\item The \emph{developing map} $\dev: \widetilde{\Sigma}\rightarrow\mathbb{RP}^2$  is a $\hol$-equivariant local diffeomorphism.

\item Two pairs $(\dev,\hol)$ and $(\dev',\hol')$ are equivalent if there is $g\in\SL(3,\mathbb{R})$ such that 
$$
\dev'=g\circ\dev,\quad \hol'(\alpha)=g\, \hol(\alpha)\, g^{-1},\  \forall \alpha\in\pi_1(\Sigma).
$$
\end{itemize}
A projective structure is said to be \emph{convex} if $\dev$ sends $\widetilde{\Sigma}$ bijectively to a properly convex open set $\Omega\subset\mathbb{RP}^2$. The space of all convex projective structure on $\Sigma$ is denoted by $\P(\Sigma)$.
%\end{definition}

Given a convex projective structure on $\Sigma$ with developing map $\dev$ and developing image $\Omega=\dev(\widetilde{\Sigma})$, we endow $\Omega$ with an orientation by means of the orientation on $\Sigma$. The previously defined $g_\Omega$, $\ve{b}_\Omega$ and $g_\Omega^\mathsf{H}$ pull back to  $\Sigma$ through $\dev$.  These pullbacks are simply called the \emph{Blaschke metric}, the \emph{Pick differential} and the \emph{Hilbert metric} of the convex projective structure, respectively. 
%They are well defined by virtue of the projective invariance of $g_\Omega$, $\ve{b}_\Omega$ and $g_\Omega^\mathsf{H}$.

Recall from the introduction that  we let $\V(\Sigma)$ denote the space of pairs $(g,\ve{b})$ where
\begin{itemize}
\item
$g$ is a complete Riemannian metric on $\Sigma$,
\item $\ve{b}$ is a holomorphic cubic differential on $\Sigma$ with respect to the conformal structure underlying $g$ and the orientation of $\Sigma$. 
\item $g$ and $\ve{b}$ satisfy Wang's equation
\begin{equation}\label{eqn_vortex}
\kappa_g=-1+2\|\ve{b}\|^2_g.
\end{equation}
\end{itemize}
For later use, we record here the coordinate expression of Eq.(\ref{eqn_vortex}) in a conformal local coordinate $z$: suppose $g=h|\dz|^2$ and $\ve{b}=b\,\dz^3$, then (\ref{eqn_vortex}) reads
$$
-\frac{2}{h}\paz\pabz \log h=-1+\frac{2|b|^2}{h^3}.
$$

Assigning to each convex projective structure its Blaschke metric and Pick differential gives rise to a canonical map
$\P(\Sigma)\rightarrow\V(\Sigma)$ which is known to be bijective. 

Let us give here some more details on the bijectivity.  Let $\A(\Sigma)$ be the space of  \emph{hyperbolic affine sphere structures} on $\Sigma$, \ie equivalence classes of ``developing pairs'' $(\Dev,\hol)$, where $\Dev: \widetilde{\Sigma}\rightarrow\mathbb{R}^3$ is now a $\hol$-equivariant embedding whose image is a hyperbolic affine sphere. Then the map $\P(\Sigma)\rightarrow\V(\Sigma)$ is essentially defined as a composition 
\begin{equation}\label{eqn_paw}
\P(\Sigma)\rightarrow\A(\Sigma)\rightarrow\V(\Sigma).
\end{equation}
The first map is bijective by virtue of the Cheng-Yau theorem, while bijectivity of the second map follows from the fundamental theorem of affine differential geometry (an analogue of the fundamental theorem of surface theory which determines a hyperspace in a Euclidean space from the first and second fundamental forms). 

The inverse $\A(\Sigma)\rightarrow\P(\Sigma)$ to the first map in (\ref{eqn_paw})  trivially assigns to each pair $(\Dev, \hol)$ the underlying pair $(\dev=\mathbb{P}\circ\dev, \hol)$, where $\mathbb{P}: \mathbb{R}^2\setminus\{0\}\rightarrow\mathbb{RP}^2$ is the projectivization.
In the next section we give a concrete construction of the inverse $\V(\Sigma)\rightarrow\A(\Sigma)$ to the second map in (\ref{eqn_paw}). This would yield immediately a description of the inverse to the canonical map $\P(\Sigma)\rightarrow\V(\Sigma)$.

\subsection{From $\V(\Sigma)$ to $\P(\Sigma)$: Wang's developing pair}\label{sec_wangdev}
\subsubsection{Parallel transports}\label{sec_pt}
Let $(E,\D,\mu)$ be a flat $\SL(3,\mathbb{R})$-vector bundle over $\Sigma$ (\ie $E$ is a rank $3$ vector bundle, $\D$ is a flat connection on $E$ and $\mu$ is a fiberwise volume form on $E$ preserved by $\D$).
The \emph{parallel transport} of the connection $\D$ along a $C^1$ path $\gamma: [0,1]\rightarrow \Sigma$ is a linear isomorphism between fibers
$$
\para(\gamma): E_{\gamma(0)}\rightarrow E_{\gamma(1)}
$$
preserving $\mu$.  A precise description goes as follows. Let $\ve{e}(t)=(e_1(t),e_2(t),e_3(t))$ be a unimodular frame of the pullback $\SL(3,\mathbb{R})$-vector bundle $\gamma^*E\rightarrow[0,1]$. The pullback connection $\gamma^*\D$ is expressed under the frame as $\dif+A(t)$, where $A(t)$ is a traceless $3\times 3$-matrix of functions on $[0,1]$. The initial value problem of linear ODE
$$
\tfrac{\dif}{\dif t}T(t) +A(t)T(t) =0, \quad T(0)=\id
$$
admits a unique solution $T: [0,1]\rightarrow\SL(3,\mathbb{R})$. We then define $\para(\gamma)$ as the linear map  $E_{\gamma(0)}\rightarrow E_{\gamma(1)}$ which is represented by the matrix $T(1)$ under the bases $\ve{e}(\gamma(0))$ and $\ve{e}(\gamma(1))$. In particular, given a base point $m\in\Sigma$, $\para$ associates to each loop $\gamma\in\pi_1(\Sigma, m)$ an element $\para(\gamma)$ in $\SL(E_m,\mu_m)$. The resulting group homomorphism $\pi_1(\Sigma, m)\rightarrow\SL(E_m,\mu_m)$ is called the \emph{holonomy} of $\D$ at $m$.

%See \eg \cite{kobayashi-nomizu} Chapter 2 Section 3 for more discussions.
The parallel transport $\para(\gamma)$ remains unchanged when $\gamma$ undergoes a (endpoints-fixing) homotopy. It also satisfies the following properties
 $$
\para(\alpha^{-1})=\para(\alpha)^{-1},\quad \para(\beta)\para(\alpha)=\para(\alpha\cdot\beta).
 $$
Here $\alpha$ and $\beta$ are oriented paths such that $\alpha$ terminates at the point where $\beta$ starts. $\alpha\cdot\beta$ denotes the concatenation of $\alpha$ and $\beta$, running from the starting point of $\alpha$ to the ending point of $\beta$.

Now let $\ve{e}=(e_1,e_2,e_3)$ be a unimodular frame field of $E$ over a contractible open set $U\subset \Sigma$. Suppose $\D$ is expressed under $\ve{e}$ as $\D=\dif+A$. Since paths in $U$ are determined up to homotopy by their endpoints, parallel transports are expressed by a \emph{two-pointed parallel transport map}
\begin{equation}\label{eqn_pttwo}
\para: U\times U\rightarrow \SL(3,\mathbb{R}), \quad (y,x)\mapsto \para(y,x).
\end{equation}
For any $x,y\in U$, let $\gamma$ be a path in $U$ going from $x$ to $y$, then $\para(y,x)$ is by definition the matrix of $\para(\gamma): E_x\rightarrow E_y$ with respect to the bases $\ve{e}(x)$ and $\ve{e}(y)$.

The map (\ref{eqn_pttwo})  is characterized by the following properties:
\begin{enumerate}
\item\label{item_pttwo1}
$\para(x,x)=\id$.
\item\label{item_pttwo2}
$\para(x,y)=\para(y,x)^{-1}$.
\item\label{item_pttwo3}
Let $\pa_u^{(2)}\para(y,x)$ (resp. $\pa_v^{(1)}\para(y,x)$) denote the derivative of $\para(y,x)$ with respect to the variable $x$ (resp. $y$) along the tangent vector $u\in\T_xU$ (resp. $v\in\T_yU$). Then
$$
\pa^{(2)}_u\para(y,x)=\para(y,x)A(u),\quad\pa^{(1)}_v\para(y,x)=-A(v)\para(y,x).
$$
\end{enumerate}

In the rest of the paper, $E$ will be the rank $3$ vector bundle $\T\Sigma\oplus\underline{\mathbb{R}}$ over $\Sigma$. Let $\underline{1}$ denote the canonical section of the line bundle $\underline{\mathbb{R}}\subset E$ and let $\underline{1}^*$ denote the section of $E^*$ such that $\langle\underline{1}^*,\underline{1}\rangle=1$ and $\underline{1}^*(v)=0$ for any $v\in\T\Sigma\subset E$.

\subsubsection{Wang's developing pair}\label{sec_wang}

We now describe the map $\V(\Sigma)\rightarrow\A(\Sigma)$ mentioned at the end of \S \ref{sec_pscs}. Namely, we assign to each $(g,\ve{b})\in\V(\Sigma)$ a developing pair $(\Dev, \hol)$ of some affine sphere structure which represents the pre-image of $(g,\ve{b})$ under the second map in (\ref{eqn_paw}).
The construction depends on a choice of base point $m\in\Sigma$ and the resulting map $\Dev$ takes values in $E_m$. A different choice gives rise to a equivalent developing pair.

Given $m\in\Sigma$, let $\widetilde{\Sigma}_m$ denote the universal cover of $\Sigma$ viewed as the space of (endpoints-fixing) homotopy classes of oriented paths on $\Sigma$ starting from $m$. The covering map $\pi:\widetilde{\Sigma}_m\rightarrow\Sigma$ assigns to each homotopy class its final point. The fundamental group $\pi_1(\Sigma, m)$ acts on $\widetilde{\Sigma}_m$ by pre-concatenation.

We first associate to $(g,\ve{b})\in\V(\Sigma)$ a flat $\SL(3,\mathbb{R})$-vector bundle $(E, \D, \mu)$ over $\Sigma$. Suppose
$g=h|\dz|^2$ and $\ve{b}=b\,\dz^3$ in a conformal local coordinate $z$. We define 
\begin{itemize}
\item 
$E:=\T\Sigma\oplus\underline{\mathbb{R}}$ , $\mu:=\nu_g\wedge\underline{1}^*$, where $\nu_g$ is the volume form of $g$;
\item
Let $\D$ be the connection on $E$ whose extension (by complex linearity) to $E\otimes\mathbb{C}=\T_\mathbb{C}\Sigma\oplus\underline{\mathbb{C}}$ is expressed under the local frame $(\paz,\pabz,\underline{1})$ as
\begin{equation}\label{eqn_wangcon}
\D=\dif+
\begin{pmatrix}
\pa\log h&\frac{\bar b}{h}\dbz&\dz\\[6pt]
\frac{b}{h}\dz&\bpa\log h&\dbz\\[6pt]
\frac{h}{2}\dbz&\frac{h}{2}\dz&0
\end{pmatrix}.
\end{equation}
Straightforward computations show $\D$ is coordinate-independent, flat and preserves $\mu$.
\end{itemize}

\begin{theorem}[\textbf{Wang's developing pair}]\label{thm_wang}
Let 
$$
\hol: \pi_1(\Sigma, m)\rightarrow \SL(E_m)\cong\SL(3,\mathbb{R})
$$
be the holonomy of the connection $\D$ at $m$ and define the map
$$
\Dev: \widetilde{\Sigma}_m\rightarrow E_m\cong\mathbb{R}^3, \quad \Dev(\gamma):=\para(\gamma^{-1})\underline{1}_{\pi(\gamma)},
$$
where $\a(\gamma^{-1}):E_{\pi(\gamma)}\rightarrow E_m$ is the parallel transport of $\D$ along $\gamma^{-1}$.
Then $(\Dev,\hol)$ represents the pre-image of $(g,\ve{b})$ under the second map in (\ref{eqn_paw}). 

As a consequence, $(\dev:=\mathbb{P}\circ\Dev,\hol)$ is a developing pair of the convex projective structure with Blaschke metric $g$ and Pick differential $\ve{b}$. 
\end{theorem}
Here, $\SL(E_m)$ denotes the group of $\mu_m$-preserving linear automorphisms of the fiber $E_m$.

We call the developing pair $(\dev, \hol)$ produced by Theorem \ref{thm_wang}  \emph{Wang's developing pair associated to $(g,\ve{b})\in \V(\Sigma)$} relative to the base point $m$. 

%\begin{remark}
%A different choice of base point $m'$ gives rise to another developing pair $(\Dev',\hol')$ which is equivalent to $(\Dev,\hol)$ in the sense of Definition 
%\ref{def_clas} and  \ref{def_centro}. Indeed, for any path $\gamma$ going from $m$ to $m'$, we have
%$$
%\hol'(\alpha)=\para(\gamma)\hol(\alpha)\para(\gamma^{-1}),\quad \Dev'=\para(\gamma)\circ\Dev.
%$$
%\end{remark}

\subsubsection{Equilateral frames}\label{sec_equilateral}
It is conceptually clearer to express the connection (\ref{eqn_wangcon}) under a frame of $E$, rather than a frame of the complexification $E\otimes\mathbb{C}$. An obvious choice of such a frame is $(\pa_x, \pa_y, \underline{1})$, where $z=x+\ima y$. Another choice, which is more convenient for our purpose, is
\begin{align}
(e_1,e_2,e_3)&:=\left(\pa_x+\underline{1}\,,\ -\tfrac{1}{2}\pa_x+\tfrac{\sqrt{3}}{2}\pa_y+\underline{1}\,,\ -\tfrac{1}{2}\pa_x-\tfrac{\sqrt{3}}{2}\pa_y+\underline{1}\right)\label{eqn_frame}\\
&=(\paz,\pabz,\underline{1})\Q,\nonumber
\end{align}
where the base change matrix $\Q$ and its inverse are given by
\begin{equation}\label{eqn_db}
\Q=\begin{pmatrix}
1&\omega&\omega^2\\
1&\omega^2&\omega\\
1&1&1
\end{pmatrix},
\quad\Q^{-1}=\frac{1}{3}
\begin{pmatrix}
1&1&1\\
\omega^2&\omega&1\\
\omega&\omega^2&1
\end{pmatrix}, \quad \omega=e^{2\pi\ima/3}.
\end{equation}
We refer to the frame (\ref{eqn_frame}) as the \emph{equilateral frame} of $E$ associated to the coordinate $z$, because its projection to $\T\Sigma$ gives vertices of an equilateral triangle in each $\T_m\Sigma$. Note that the canonical section $\underline{1}$ of $E$ is $e_1+e_2+e_3$.

Such a frame is useful when $\ve{b}$ is a constant multiple of $\dz^3$. Throughout this paper, we always use the letter $\zeta$ to denote a local coordinate of $\Sigma$ away from zeros of $\ve{b}$ such that $\ve{b}=2\,\dzeta^3$, and let $u(\zeta)$ be the function  such that the Blaschke metric is expressed as
$$
g=2e^{u(\zeta)}|\dzeta|^2.
$$
In other words, $u$ is defined in a coordinate-free way as 
$$
u=\log\left(\frac{g}{2^\frac{1}{3}|\ve{b}|^\frac{2}{3}}\right)=-\frac{1}{3}\log(2\|\ve{b}\|_g^2). 
$$

The connection (\ref{eqn_wangcon}) is expressed under the equilateral frame associated to $\zeta$ as
\begin{equation}\label{eqn_d}
\D=\dif+\Q^{-1}\begin{pmatrix}
\pa u&e^{-u}\dbzeta&\dzeta\\[5pt]
e^{-u}\dzeta&\bpa u&\dbzeta\\[5pt]
e^u\dbzeta&e^u\dzeta&0\\
\end{pmatrix}\Q.
\end{equation}
We shall study next the resulting parallel transport and developing map in detail when $g=2|\dzeta|^2$, \ie when $u\equiv0$.

\subsection{Wang's developing map for $g=2|\dzeta|^2$ and $\ve{b}=2\,\dzeta^3$}
\subsubsection{The \c{T}i\c{t}eica affine sphere}\label{sec_titeica}
The simplest example of $(g,\ve{b})\in\V(\Sigma)$ is given by 
\begin{equation}\label{eqn_g0}
\Sigma=\mathbb{C},\quad g_0=2|\dzeta|^2,\quad \ve{b}_0=2\,\dzeta^3.
\end{equation}

Let $\D_0$ denote the connection (\ref{eqn_wangcon}) in this case. The matrix expression of $\D_0$ under the equilateral frame associated to $\zeta$ turns out to be diagonal:
\begin{equation}\label{eqn_wangcon0}
\D_0=\dif+B^{-1}
\begin{pmatrix}
0&\dbzeta&\dzeta\\[3pt]
\dzeta&0&\dbzeta\\[3pt]
\dbzeta&\dzeta&0
\end{pmatrix}\Q
=\begin{pmatrix}
2\re(\dzeta)&&\\
&\!\!\!2\re(\omega^2\,\dzeta)&\\
&&\!\!\!2\re(\omega\,\dzeta)
\end{pmatrix},
\end{equation}
where $\omega=e^{2\pi\ima/3}$.

Taking $m=0\in\mathbb{C}$ to be the base point, we apply Theorem \ref{thm_wang} and get an affine spherical embedding $\Dev_0: \mathbb{C}\rightarrow E_0$ (here $E_0$ is the fiber of $E=\T\mathbb{C}\oplus\underline{\mathbb{R}}$ at $0$), known as the \emph{\c{T}i\c{t}eica affine sphere}.

In order to get an explicit expression, we let $\a_0: \mathbb{C}\times\mathbb{C}\rightarrow \SL(3,\mathbb{R})$ denote the two-pointed parallel transport map of $\D_0$ under the equilateral frame $(e_1,e_2,e_3)$ associated to the globally defined coordinate $\zeta$.
 %To get an explicit expression for $\a_0$, we remark that $\D$ is invariant under translations $z\mapsto z+a$, hence $\a_0$ is invariant as well, so we have
%$$
%\a_0(z_2,z_1)=\a_0(0,z_1-z_2),
%$$
%$$
%\pa^{(2)}_u\a_0(0,z)=\a_0(0,z)\begin{pmatrix}
%2\re(u)&&\\
%&2\re(\omega^2u)&\\
%&&2\re(\omega u)
%\end{pmatrix}
%$$
Since $\D_0$ is invariant under translations, we have
$$
\a_0(\zeta_2,\zeta_1)=\a_0(0, \zeta_1-\zeta_2),
$$
whereas the expression of $\D_0$ and the properties of two-pointed parallel transport maps given at the end of \S \ref{sec_pt} yield
\begin{equation}\label{eqn_tit}
\a_0(0,\zeta)
=
\begin{pmatrix}
e^{2\re(\zeta)}&&\\
&\!\!\!e^{2\re(\omega^2\zeta)}&\\
&&\!\!\!e^{2\re(\omega \zeta)}
\end{pmatrix}
.
\end{equation}

Since $\underline{1}=e_1+e_2+e_3$, we get
$$
\Dev_0(\zeta)=\a_0(0,\zeta)\underline{1}_\zeta
%\begin{pmatrix}
%1\\
%1\\
%1
%\end{pmatrix}=
%\begin{pmatrix}
%e^{2\re(z)}\\[3pt]
%e^{2\re(\omega^2z)}\\[3pt]
%e^{2\re(\omega z)}
%\end{pmatrix}.
=e^{2\re(\zeta)}e_1(0)+e^{2\re(\omega^2\zeta)}e_2(0)+e^{2\re(\omega \zeta)}e_3(0).
$$
%More precisely, this expression gives the coordinates of $\Dev(z)\in \T_0\mathbb{C}\oplus\mathbb{R}\cong\mathbb{R}^3$ under the equilateral frame.  

The image of $\Dev_0$ is the hypersurface 
$$
\{x\,e_1(0)+y\,e_2(0)+z\,e_3(0)\in E_0\mid xyz=1\}
$$
asymptotic to the first octant in $E_0\cong\mathbb{R}^3$.  The developing map of the associated convex projective structure is
$$
\dev_0=\mathbb{P}\circ\Dev_0: \widetilde{\Sigma}_m\rightarrow \mathbb{P}(E_0).
$$ 
Its image is a triangle $\Delta\subset\mathbb{P}(E_0)$, in accordance with the last property in \S \ref{sec_canometric}.  

For each  $v\in E_0$, let $[v]\in\mathbb{P}(E_0)$ denote its projectivization. We denote the vertices and edges of the triangle $\Delta$ by
$$
X_i:=[e_i(0)], \  X_{ij}:=\left\{\big[(1-s)\,e_i(0)+s\,e_j(0)\big]\right\}_{s\in[0,1]}
$$
and let $X_{ij}^\circ$ denote the interior of the closed segment $X_{ij}$.
%$$
%x_{100}:=[e_1(0)],\quad x_{010}:=[e_2(0)],\quad x_{001}:=[e_3(0)].
%$$
%Also denote the interiors of edges by
%\begin{align*}
%l_{110}&:=\{[s e_1(0)+s^{-1}e_2(0)]\}_{s>0},\\
%l_{101}&:=\{[s e_1(0)+s^{-1}e_3(0)]\}_{s>0},\\
%l_{011}&:=\{[s e_2(0)+s^{-1}e_3(0)]\}_{s>0}.
%\end{align*}

Although the norm of $\Dev_0(\zeta)$ tends to infinity when $\zeta$ goes to $\infty$ in $\mathbb{C}$, the projectivization $\dev_0(\zeta)$ has limit points lying on $\pa\Delta$, as described by the next proposition.  The proof amounts to elementary limit calculations, which we omit.  Note that in the first part of the proposition, a parametrized curve $t\mapsto x(t)$ in $\mathbb{RP}^2$ converging to a point $x_0$ is said to be \emph{asymptotic} to a ray or a segment $l$  issuing from $x_0$ if, for any circular sector $S$ of small angle base at $x_0$ such that $S$ contains a portion of $l$, we have $x(t)\in S$ when $t$ is big enough.  

\begin{proposition}\label{prop_tit}
Assume that a path $\alpha: [0,+\infty)\rightarrow \mathbb{C}$ satisfies  
$$|\alpha(t)|\rightarrow+\infty,\quad\arg(\alpha(t))\rightarrow\theta(\alpha)\in[0,2\pi)$$ 
as $t\rightarrow+\infty$.
\begin{enumerate}
\item\label{item_tit1}
%If $\theta(\alpha)\neq\pm\tfrac{\pi}{3}$ or $\pi$, then $\dev_0(\alpha(t))$ converges. More specifically, 
%$$
%\lim_\tinf \dev_0(\alpha(t))=
%\begin{cases}
%X_1&\mbox{ if } \theta\in (-\tfrac{\pi}{3}, \tfrac{\pi}{3}),\\
%X_2&\mbox{ if } \theta\in (\tfrac{\pi}{3}, \pi),\\
%X_3&\mbox{ if }\theta\in (\pi, \frac{4\pi}{3}).
%\end{cases}
%$$
Define the intervals
$$
I_1=(-\tfrac{\pi}{3}, \tfrac{\pi}{3}),\quad I_2=(\tfrac{\pi}{3},\pi),\quad I_3=(\pi,\tfrac{5\pi}{3}).
$$
Then
$\lim_\tinf\dev_0(\alpha(t))=X_i$ if $\theta(\alpha)\in I_i$.
Furthermore, let $I_i^-$ (resp. $I_i^+$) denote the first (resp. second) open half of $I_i$, \ie $I_1^-=(-\frac{\pi}{3},0)$, $I_1^+=(0,\frac{\pi}{3})$, \etc, then the curve $t\mapsto\dev_0(\alpha(t))$ is asymptotic to $X_{i,i\pm1}$ (here the indices are counted modulo $3$) if $\theta(\alpha)\in I_i^\pm$.

 \item\label{item_tit2}
For any $\theta\in\{\pm\tfrac{\pi}{3}, \pi\}$ and $s\in\mathbb{R}$, put 
$$
\alpha_{\theta,s}(t):=e^{\theta\ima}(t+s\ima).
$$ 
Then
\begin{align*}
\lim_\tinf \dev_0(\alpha_{\tfrac{\pi}{3},s}(t))&=\big[e^{-\sqrt{3}\, s}e_1(0)+e^{\sqrt{3}\, s}e_2(0)\big]\in X_{12}^\circ,\\	
 \lim_\tinf \dev_0(\alpha_{-\tfrac{\pi}{3},s}(t))&=\big[e^{-\sqrt{3}\, s}e_3(0)+e^{\sqrt{3}\, s}e_1(0)\big]\in X_{31}^\circ,\\
  \lim_\tinf \dev_0(\alpha_{\pi,s}(t))&=\big[e^{-\sqrt{3}\, s}e_2(0)+e^{\sqrt{3}\, s}e_3(0)\big]\in X_{23}^\circ.
\end{align*}
As a consequence, each point of  $X_{12}^\circ$ (resp. $X_{23}^\circ$, $X_{31}^\circ$) is the limit of $\dev_0(\alpha(t))$ for some $\alpha$ satisfying $\theta(\alpha)=\tfrac{\pi}{3}$ (resp. $\pi$, $-\tfrac{\pi}{3}$).
\end{enumerate}
\end{proposition}

See Figure \ref{figure_dev02}  and Figure \ref{figure_dev01} in the introduction for illustrations of part (\ref{item_tit1}) and part (\ref{item_tit2}), respectively, of the above proposition.

\subsubsection{Eigenvalues of $\Ad_{\a_0(0,\zeta)}$}\label{sec_eigen}
For our purpose in Section \ref{sec_ctop},  we need some details on  eigenvalues of the Lie algebra inner automorphism 
$$
\Ad_{\a_0(0,\zeta)}: \sl_3\mathbb{R}\rightarrow \sl_3\mathbb{R}
$$ 
relative to $\zeta\in\mathbb{C}$.  We record here some calculations used later on.

Let $E_{ij}$ denote the matrix whose $(i,j)$-entry is $1$ and the other entries vanish. It follows from the expression (\ref{eqn_tit}) that $\Ad_{\a_0(0,\zeta)}$ acts trivially on the diagonal subalgebra and acts on $E_{ij}$ ($i\neq j$) with eigenvalue 
$$
\exp\big(2\re(\omega^{1-i}\zeta-\omega^{1-j}\zeta)\big)=\exp\big(\varpi_{ij}(\arg(\zeta))|\zeta|\big),
$$
where
$$
\varpi_{ij}(\theta):=2\re\big(\omega^{1-i}e^{\theta\ima}-\omega^{1-j}e^{\theta\ima}\big).
$$

It is convenient to visualize $\varpi_{ij}(\theta)$ as follows. Consider a planar tripod formed by three equally angled rods of length $2$, rotating on the complex plane $\mathbb{C}$ with center fixed at $0$. Label the rods clockwise by $1$, $2$ and $3$. Then $\varpi_{ij}(\theta)$ is the signed horizontal distance from the end of rod $i$ to the end of rod $j$ when the angle between rod $1$ and the ray $\mathbb{R}_{\geq 0}$ is $\theta$. Using this description, one easily checks the following facts.
\begin{enumerate}
\item\label{item_eigen1}
The function $$\varpi(\theta):=\max_{i,j}\varpi_{ij}(\theta)$$ is continuous and reaches its maximum $2\sqrt{3}$ if and only if $\theta$ is an odd multiple of $\tfrac{\pi}{6}$.  Note that the spectral radius of $\Ad_{\a_0(0,\zeta)}$ is
$$
\rho(\Ad_{\a_0(0,\zeta)})=\exp\big(\varpi\big(\arg(\zeta)\big)|\zeta|\big).
$$

\item\label{item_eigen2}
For $\theta$ at each odd multiple of $\frac{\pi}{6}$, there is exactly one pair $(i,j)$ such that $\varpi_{ij}(\theta)=\varpi(\theta)=2\sqrt{3}$,  given by the following table.

\vspace{5pt}
\begin{tabular}{|c|c|c|c|c|c|c|}
\hline
$\theta$&$\frac{\pi}{6}$&$\frac{\pi}{2}$&$\frac{5\pi}{6}$&$\frac{7\pi}{6}$&$\frac{3\pi}{2}$&$-\frac{\pi}{6}$\\[3pt]\hline
$(i,j)$&$(1,3)$&$(2,3)$&$(2,1)$&$(3,1)$&$(3,2)$&$(1,2)$\\
\hline
\end{tabular}
\vspace{5pt}

%then $(i,j)$-entry of the following matrix is the eigenvalue of $\Ad_{\a_0(0,\zeta)}$ on $E_{ij}$ (since the eigenvalue on $E_{ij}$ is reciprocal to the one on $E_{ji}$, it is sufficient to specify the upper-triangular part of the matrix).  \begin{align*}
%\begin{bmatrix}
%1\!&e^{2t(\cos(\theta)-\cos(\theta-\frac{2\pi}{3}))}&\!\!\!\!\!\!\!e^{2t(\cos(\theta)-\cos(\theta+\frac{2\pi}{3}))}\\[5pt]
%*&1&\!\!\!\!\!\!\!e^{2t(\cos(\theta-\frac{2\pi}{3}))-\cos(\theta+\frac{2\pi}{3}))}\\[5pt]
%*&*&1
%\end{bmatrix}
%%\end{align*}
%%\begin{align*}
%%\quad\quad
%=
%\begin{bmatrix}
%1\!&e^{2\sqrt{3}\,t\cos(\theta+\tfrac{\pi}{6})}&\!\!e^{2\sqrt{3}\,t\cos(\theta-\tfrac{\pi}{6})}\\[5pt]
%*&1&\!\!e^{2\sqrt{3}\,t\sin(\theta)}\\[5pt]
%*&*&1
%\end{bmatrix}\label{eqn_eigens}
%\end{align*}
\end{enumerate}

\subsubsection{Convex projective structures whose developing images are triangles}
The following proposition enumerates all convex projective structures whose developing image is a triangle in $\mathbb{RP}^2$.

\begin{proposition}\label{prop_triangle}
Let $(\dev,\hol)$ be a developing pair of a convex projective structure $\Sigma$ with Blaschke metric $g$ and Pick differential $\ve{b}$. Then $\Omega=\dev(\widetilde{\Sigma})$ is a triangle if and only if the triple $(\Sigma, g,\ve{b})$ is equivalent  to one of the following three.
\begin{enumerate}[(a)]
\item \label{item_tri1}
$(\mathbb{C}, 2|\dzeta|^2, 2\dzeta^3)$;\vspace{4pt}
\item\label{item_tri2}
$(\mathbb{C}^*, 2^\frac{1}{3}|R|^\frac{2}{3}|z|^{-2}|\dz|^2, Rz^{-3}\dz^3)$, where $R\in\mathbb{C}^*$;\vspace{4pt}
\item\label{item_tri3}
$(\mathbb{C}/\Lambda, 2|\dzeta|^2, 2\,\dzeta^3)$, where $\Lambda\cong\mathbb{Z}^2$ is a lattice in $\mathbb{C}$.
\end{enumerate}
Furthermore,
\begin{enumerate}
\item\label{item_proptri1}
In case (\ref{item_tri2}),  the holonomy of the convex projective structure along a  loop going around $0$ counter-clockwise is conjugate to
$$
\begin{pmatrix}
e^{-4\pi\mu_1}&&\\
&e^{-4\pi\mu_2}&\\
&&e^{-4\pi\mu_3}
\end{pmatrix},
$$
where $\mu_1,\mu_2,\mu_3\in\mathbb{R}$ are the imaginary parts of the three cubic roots of $R/2$, respectively.
\item\label{item_proptri2}
If $\Sigma$ is a torus then $(\Sigma, g, \ve{b})$ is always equivalent to (\ref{item_tri3}). As a consequence, the developing image of any convex projective structure on a torus is a triangle.
\end{enumerate}
\end{proposition}
Note that part (\ref{item_proptri1}) is the simplest example of Theorem \ref{intro_thm2} from the introduction.

\begin{proof}
%Let $\tilde{g}$ and $\tilde{\ve{b}}$ be the lifts of $g$ and $\ve{b}$ to $\widetilde{\Sigma}$.
%The discussions in \S \ref{sec_titeica} imply that the triangle $\Delta=\Omega$, we have
%$$
%(\Delta, g_\Delta, \ve{b}_\Delta)\cong (\mathbb{C}, 2|\dz|^2, 2\dz^3).
%$$

The developing map $\dev$ induces an isomorphism between $(\Sigma, g, \ve{b})$ and the quotient of $(\mathbb{C}, g_0, \ve{b}_0)$ by a free and properly discontinuous action of $\pi_1(\Sigma)$ preserving $\ve{b}_0$ (and hence preserving $g_0$). But the only non-trivial such actions are the actions of either $\mathbb{Z}$ or $\mathbb{Z}^2$ by translations.

In the case of $\mathbb{Z}$,  the action is generated by $\zeta\mapsto \zeta+2\pi\ima\mu$ for some $\mu\in\mathbb{C}$, $\mu\neq0$. The map
\begin{equation}\label{eqn_proptri}
\mathbb{C}\rightarrow \mathbb{C}^*,\quad \zeta\mapsto z=e^{\zeta/\mu},
\end{equation}
identifies the quotient $(\mathbb{C},g_0,\ve{b}_0)/\mathbb{Z}$ with 
$\left(\mathbb{C}^*,\, \tfrac{2|\mu|^2}{|\zeta|^2}|\dzeta|^2, \, \tfrac{2\mu^3}{\zeta^{3}}\dzeta^3\right)$
%$$
%\left(\mathbb{C}^*,\, \tfrac{2|\lambda|^2}{|\zeta|^2}|\dzeta|^2, \, \tfrac{2\lambda^3}{\zeta^{3}}\dzeta^3\right)=\left(\mathbb{C}^*,\,\frac{2^\frac{1}{3}|R|^\frac{2}{3}}{|\zeta|^{2}}|\dzeta|^2, \frac{R}{\zeta^{3}}\dzeta^3\right)
%$$
in such a way that a path in $\mathbb{C}$ going from  $2\pi\ima\mu$ to $0$ corresponds to a loop in $\mathbb{C}^*$ going around $0$ counter-clockwise.
Eq. (\ref{eqn_tit}) gives the holonomy along such a loop as
$$
\a_0(0,2\pi\lambda\ima)
=
\begin{pmatrix}
e^{-4\pi\im(\lambda)}&&\\
&e^{-4\pi\im(\omega^2\lambda)}&\\
&&e^{-4\pi\im(\omega\lambda)}
\end{pmatrix},\  \omega=e^{2\pi\ima/3}.
$$
Statement (\ref{item_proptri1}) is proved by setting $R=2\mu^3$.

We now assume that $\Sigma$ is a torus and prove statement (\ref{item_proptri2}). 
Fix $(g,\ve{b})\in\V(\Sigma)$. $(\Sigma,g)$ is conformally equivalent to $\mathbb{C}/\Lambda$ for some lattice $\Lambda\subset\mathbb{C}$ and we have $\ve{b}=b(z)\,\dz^3$ for a $\Lambda$-periodic entire function $b(z)$, which must be a constant. We can assume $b=2$ by scaling and assume $g=h|\dz|^2$ for some smooth function $h: \mathbb{C}/\Lambda\rightarrow\mathbb{R}_+$. Wang's equation (\ref{eqn_vortex}) then becomes
$$
-\frac{2}{h}\pazbz\log h+1-\frac{8}{h^3}=0.
$$
Applying the maximum principle to this equation shows that the maximal and minimal values of $h$ are both $2$, hence the proof is complete.
\end{proof}

\subsubsection{Wang's developing pair associated to $(2^\frac{1}{3}|\ve{b}|^\frac{2}{3}, \ve{b})$}\label{sec_wangass} 
%\red{.....}
 It is important for our purpose later on to remark that the construction of Wang's developing pair $(\dev, \hol)$ in Theorem \ref{thm_wang} makes sense not only for $(g,\ve{b})\in\V(\Sigma)$, but also for some more general $(g,\ve{b})$.

Indeed, one can carry out the construction of $\Dev$ and $\hol$ in Theorem \ref{thm_wang} whenever the connection $\D$ defined by (\ref{eqn_wangcon}) is flat, whereas the flatness is equivalent to Wang's equation (\ref{eqn_vortex}) and does not particularly require $g$ to be complete as in the definition of $\V(\Sigma)$. We can even allow $g$ to have zeros as long as the ratio $b/h$ in the expression of $\D$ makes sense at the zeros.

Therefore, we can put the singular flat metric $2^\frac{1}{3}|\ve{b}|^\frac{2}{3}$ in place of $g$ and carry out the construction, \ie integrate the flat connection (\ref{eqn_wangcon}), denote by $\D_0$ in this case. We get a pair $(\dev_0, \hol_0)$ with
$$
\dev_0: \widetilde{\Sigma}_m\rightarrow\mathbb{P}(E_m),\quad \hol_0:\pi_1(\Sigma, m)\rightarrow \SL(E_m, \mu_m)
$$
such that $\dev_0$ is $\hol_0$-equivariant.  We called it \emph{Wang's developing pair associated to $(2^\frac{1}{3}|\ve{b}|^\frac{2}{3}, \ve{b})$}, although in general it is not really a developing pair of any projective structure.

%Note that if $\ve{b}$ is non-vanishing and has a pole of order $\geq 3$ at each puncturen, then $(2^\frac{1}{3}|\ve{b}|^\frac{1}{3}, \ve{b})$ actually belongs to $\V(\Sigma)$. This occurs only in the three examples in Proposition \ref{prop_triangle}.

As in \S \ref{sec_equilateral}, let $\zeta$ be a conformal local coordinate defined on a contractible open set $U$ such that $\ve{b}=2\,\dzeta^3$. Let $\a_0: U\times U\rightarrow \SL(3,\mathbb{R})$ be the two-pointed parallel transport map of $\D_0$ with respect to the equilateral frame associated to $\zeta$.
Then the local expression of $(2^\frac{1}{3}|\ve{b}|^\frac{2}{3}, \ve{b})$ coincide with  (\ref{eqn_g0}), hence $\D_0$ and $\a_0$ also have the same expressions (\ref{eqn_wangcon0}) and (\ref{eqn_tit}) as the \c{T}i\c{t}eica affine sphere. 
%This explains our choice of notations.
Indeed, in this paper, the main use of the \c{T}i\c{t}eica example is to provide a local model for Wang's developing map associated to $(2^\frac{1}{3}|\ve{b}|^\frac{2}{3}, \ve{b})$.

\section{Ends of convex projective surfaces}\label{sec_devpunc}

In this section, we fix the following data:
\begin{itemize}
\item a surface $\Sigma$ satisfying the assumptions at the beginning of Section \ref{sec_pre}.
\item a base point $m\in\Sigma$, so as to view points in the universal cover $ \widetilde{\Sigma}=\widetilde{\Sigma}_m$ as paths on $\Sigma$ as in \S \ref{sec_wang}; let $\pi:\widetilde{\Sigma}\rightarrow\Sigma$ denote the covering map;
\item a developing pair $(\dev, \hol)$ of a convex projective structure on $\Sigma$; put $\Omega:=\dev(\widetilde{\Sigma})\subset\mathbb{RP}^2$.
\end{itemize}

\subsection{Developed boundaries}\label{sec_devbd}
\subsubsection{The Farey set}\label{sec_farey}
If $S$ is a complete hyperbolic surface and $p$ is a cusp, identifying the universal cover $\widetilde{S}$ with the Poincar\'e disk $\mathbb{D}$, following \cite{fock-goncharov}, we call the subset of $\pa\mathbb{D}$ corresponding to $p$ the \emph{Farey set} of $S$ with respect to $p$.  This is a countable set with a transitive $\pi_1(S)$-action such that the stabilizer of each point is an infinite cyclic group generated by a loop going around $p$ once. This name comes from the fact that when $S$ is a punctured torus, an ideal triangulation of $S$ by two triangles lifts to the classic Farey triangulation of $\mathbb{D}$, whose vertex set is the Farey set.

We shall now give an equivalent definition of Farey sets without referring to hyperbolic metrics, which is more suitable for our purposes later on. 

\begin{definition}[\textbf{Homotopy of paths tending to a puncture}]\label{def_homotopy}
Given a puncture $p$ of $\Sigma$, let $\mathscr{P}_{m,p}$ denote the set of continuous paths $\beta:[0,+\infty)\rightarrow \Sigma$ such that $\beta(0)=m$ and $\lim_\tinf\beta(t)=p$. Two paths $\beta_0,\beta_1\in\mathscr{P}_{m,p}$ are call \emph{homotopic} if there is $\beta_s\in\mathscr{P}_{m,p}$ for $s\in(0,1)$ such that the map 
$$
[0,1]\times[0,+\infty)\rightarrow\Sigma,\quad(s,t)\mapsto \beta_s(t)
$$ 
is continuous and that the convergence $\lim_\tinf\beta_s(t)=p$ is uniform in $s\in[0,1]$. 
\end{definition}
%\begin{remark}
%Adding the puncture $p$ to $\beta\in\mathscr{P}_{m,p}$, we get a path $\bar{\beta}$ on $\overline{\Sigma}$ going from $m$ to $p$. Then $\beta_0,\beta_1\in\mathscr{P}_{m,p}$ are homotopic in the above sense if and only if $\bar{\beta}_0$ and $\bar{\beta}_1$ are homotopic as paths on $\overline{\Sigma}$ in the usual sense through a family of paths from $m$ to $p$ which does not pass through any puncture. 
%\end{remark}

\begin{definition}[\textbf{Farey set}]
The Farey set $\Fa{\Sigma,p}$ is the set of homotopy classes of elements in $\mathscr{P}_{m,p}$, equipped with the natural $\pi_1(\Sigma,m)$-action given by the pre-concatenation of $\beta\in\mathscr{P}_{m,p}$ with loops based at $m$ which represent elements in $\pi_1(\Sigma,m)$.
\end{definition}
Note that in the non-negative Euler characteristic cases $\Sigma\cong\mathbb{C}$ and $\Sigma\cong\mathbb{C}^*$, the Farey set $\Fa{\Sigma,p}$ consists of a single element.

For the sake of completeness, we outline a prove of the equivalence between the above definition and the one introduced earlier through hyperbolic metrics.
\begin{proposition}\label{prop_farey}
Fixing a complete hyperbolic metric on $\Sigma$ such that $p$ is a cusp, we identify $\widetilde{\Sigma}$ with the Poincar\'e disk $\mathbb{D}$ and view $\pi_1(\Sigma)$ as a subgroup of the isometry group $\mathsf{Isom}(\mathbb{D})$. Then there is a natural $\pi_1(\Sigma)$-action-preserving bijection between $\Fa{\Sigma,p}$ and the subset of $\pa\mathbb{D}$ corresponding to $p$.
\end{proposition}
\begin{proof}[Outline of proof]
Let $F\subset\pa\mathbb{D}$ be the subset corresponding to $p$.  Assume that $0\in\mathbb{D}$ projects to $m\in\Sigma$. Each $\beta\in\mathscr{P}_{m,p}$ lifts to a path $\widetilde{\beta}:[0,+\infty)\rightarrow\mathbb{D}$ with $\widetilde{\beta}(0)=0$. $\widetilde{\beta}(t)$ converges to a point in $F$, so 
we can define the map 
\begin{equation}\label{eqn_pmpf}
\devlim: \mathscr{P}_{m,p}\rightarrow F, \quad \devlim(\beta)=\lim_\tinf\widetilde{\beta}(t).
\end{equation}
It is surjective because for any $x\in F$, the ray in $\mathbb{D}$ from $0$ to $x$ projects to an element of $\beta\in\mathscr{P}_{m,p}$ whose image by $\devlim$ is $x$. To prove the proposition, it is sufficient to show that $\beta_0$ and $\beta_1$ are homotopic if and only if $\devlim(\beta_0)=\devlim(\beta_1)$. The ``if'' part is proved by taking the homotopy $(\beta_s)_{s\in[0,1]}$ such that $\beta_s$ is the projection of $\widetilde{\beta}_s(t):=(1-s)\widetilde{\beta}_0+s\,\widetilde{\beta}_1$. To prove the ``only if'' part, let $U$ be a punctured neighborhood of $p$ such that $\pa U$ is a horocycle, so that $\pi^{-1}(U)$ is a disjoint union of horodisks.  By assumption there is $t_0>0$ such that $\beta_s(t)\in U$ whenever $t>t_0$ and $s\in[0,1]$, hence the image of $[0,1]\times[t_0,+\infty)$ by the map
$$
[0,1]\times[0,+\infty)\mapsto \mathbb{D},\quad (s,t)\mapsto \widetilde{\beta}_s(t)
$$
is contained in $\pi^{-1}(U)$. By connectedness, the image is contained in the horodisk at some $\tilde{p}\in F$, therefore $\devlim(\beta_0)=\devlim(\beta_1)=\tilde{p}$.
\end{proof}
\begin{remark}
The following example about necessity of the uniformity condition in Definition \ref{def_homotopy} might help understanding Proposition \ref{prop_farey}: let $\widetilde{\beta}_0$ and $\widetilde{\beta}_1$ be rays in $\mathbb{D}$ from $0$ to different points $x_0,x_1\in\pa\mathbb{D}$, then they extend to a continuous map $(s,t)\mapsto\widetilde{\beta}_s(t)$ from $[0,1]\times[0,+\infty)$ to $\mathbb{D}$ such that $\lim_\tinf\widetilde{\beta}_0(t)=x_0$ and $\lim_\tinf\widetilde{\beta}_s(t)=x_1$ for $s\in(0,1]$. Such an  extension can be obtained, say, from the following map into the triangle $\Delta=\{[x_1:x_2:x_3]\mid x_i\geq 0\}\subset\mathbb{RP}^2$,
$$
f: [0,1]\times[0,+\infty)\rightarrow \Delta,\ (s,t)\mapsto [(1-s)t:st^2:1],
$$ 
which satisfies $\lim_\tinf f(0,t)=[1:0:0]$ and $f(s,t)=[0:1:0]$ for $s\in(0,1]$.  Projecting $\widetilde{\beta}_s$ to the hyperbolic surface, we can get paths $\beta_0,\beta_1\in\mathscr{P}_{m,p}$ satisfying Definition \ref{def_homotopy} except the uniformity condition, and $\beta_0, \beta_1$ are mapped by (\ref{eqn_pmpf}) to different points of $F$.
\end{remark}

With the above interpretation of $\Fa{\Sigma,p}$ as a subset of $\pa\mathbb{D}$, given a punctured neighborhood $U$ of $p$ such that $\pa U$ is a horocycle, the pre-image of $U$ by the covering map $\mathbb{D}\cong\widetilde{\Sigma}\rightarrow\Sigma$ consists of infinitely many horodisks, and we call the one centered at $\tilde{p}\in\Fa{\Sigma,p}\subset\pa\mathbb{D}$ the \emph{$\tilde{p}$-lift} of $U$. Here is an alternative definition without referring to hyperbolic metrics.

\begin{definition}[\textbf{$\tilde{p}$-lift}]
Let $U$ be a connected punctured neighborhood of $p$ and let $\tilde{p}\in\Fa{\Sigma,p}$. Suppose that $\beta\in\mathscr{P}_{m,p}$ is in the homotopy class $\tilde{p}$ and that $\beta([t_0,+\infty))\subset U$ for some $t_0>0$. The subset $\widetilde{U}\subset\widetilde{\Sigma}$ consisting of points represented by paths of the form $\alpha\cdot\beta|_{[0,t_0]}$, where $\alpha$ is a path in $U$ starting from $\beta(t_0)$, is called the \emph{$\tilde{p}$-lift} of $U$.
\end{definition}

\subsubsection{Developed boundaries}\label{subsubsec_devbd}
%Fix a connected punctured neighborhood $U$ of $p$. The pre-image $\pi^{-1}(U)$ has infinitely many connected components unless $\Sigma$ is homeomorphic to $\mathbb{C}$ or $\mathbb{C}^*$. We choose from now on a connected component $\widetilde{U}$ of $\pi^{-1}(U)$. Many definitions below depends on the choice, but statements are true for any choice.
%
%
%Let $\pi_1^p(\Sigma,m)$ denote the subgroup of $\pi_1(\Sigma,m)$ consisting of elements represented by loops of the form
%$$
%\gamma\circ\alpha\circ \gamma^{-1}, 
%$$
%where the path $\gamma$ represents a point in $\widetilde{U}$ and $\alpha$ is a loop in $U$ based at the starting point of $\gamma$.
%
%The orientation of $\Sigma$ enables us to pick a distinguished generator $\beta$ of $\pi_1^p(\Sigma,m)$. We define holonomy of the projective structure \emph{around $p$} as 
%$$
%\hol_{\tilde{p}}:=\hol(\beta)\in\SL(3,\mathbb{R}).
%$$

%As mentioned in the introduction, given a convex projective structure on $\Sigma$, one can assign to each puncture $p$ some connected closed sets on the boundary of $\Omega=\dev(\widetilde{\Sigma})$. 
%
%In order to give a detailed definition, 

We recall the following point-set-topological notion: given a sequence of subsets $S_1, S_2,\cdots$ of a topological space $X$, the set of \emph{accumulation points} of the $S_i$'s, as $i$ tends to infinity, is defined as
$$
\underset{i\rightarrow+\infty}\Ac(S_i):=\{x\in X\mid \exists\, x_i\in S_i\mbox{ such that } \lim_{i\rightarrow+\infty}x_i=x\}.
$$
Note that if $(S_i)$ is a decreasing sequence, \ie $S_1\supset S_2\supset\cdots$, then
$$
\underset{i\rightarrow+\infty}\Ac(S_i)=
\bigcap_{i=1}^{+\infty
}\overline{S_i}.
$$

We return to the punctured surface $\Sigma$. A decreasing sequence of punctured neighborhoods $U_1\supset U_2\supset\cdots$ of $p$ is said to be \emph{exhausting} if any punctured neighborhood of $p$ contains some $U_i$.

\begin{definition}[\textbf{Developed Boundary and end holonomy}]\label{def_db}
Suppose that $U_1\supset U_2\supset\cdots$ is an exhausting sequence of connected punctured neighborhoods of $p$. Let $\tilde{p}\in\Fa{\Sigma,p}$ and let $\widetilde{U}_i\subset\widetilde{\Sigma}$ be the $\tilde{p}$-lift of $U_i$. Define 
$$
\devbd{\tilde{p}}:=\underset{i\rightarrow+\infty}\Ac(\dev(\widetilde{U}_i))\subset\mathbb{RP}^2,\quad \hol_{\tilde{p}}:=\hol(\gamma_{\tilde{p}})\in\SL(3,\mathbb{R}).
$$
We call $\devbd{\tilde{p}}$ and $\hol_{\gamma_{\tilde{p}}}$ the \emph{developed boundary} and the \emph{end holonomy} at $\tilde{p}$, respectively.

When $\Sigma\cong\mathbb{C}$ or $\mathbb{C}^*$, $\Fa{\Sigma,p}$ has only one element and we simply denote its corresponding developed boundary and end holonomy, respectively, by $\devbd{p}$ and $\hol_p$.
\end{definition}
Since $(U_i)$ is exhausting, the definition of $\devbd{\tilde{p}}$ does not depend on the choice of $(U_i)$.  Some fundamental properties of $\devbd{\tilde{p}}$ are given by the next proposition.

\begin{proposition}\label{prop_devbdproperty}
\begin{enumerate}
\item\label{prop_devbdproperty1}
$\devbd{\tilde{p}}$ is a nonempty connected closed subset of $\pa\Omega$.
\item\label{prop_devbdproperty2}
$\devbd{\gamma.\tilde{p}}=\hol(\gamma).\devbd{\tilde{p}}$ for any $\gamma\in\pi_1(\Sigma)$. In particular, $\devbd{\tilde{p}}$ is preserved by $\hol_{\tilde{p}}$.
\item\label{prop_devbdproperty3}
If $\Sigma$ has negative Euler characteristic and $\tilde{p}\neq \tilde{p}'$, then $\devbd{\tilde{p}}$ and $\devbd{\tilde{p}'}$ are disjoint.
\end{enumerate}
\end{proposition}
\begin{proof}[Proof of part (\ref{prop_devbdproperty1}) and part (\ref{prop_devbdproperty2})]
(\ref{prop_devbdproperty1}) Since $\overline{\Omega}$ is compact, $\devbd{\tilde{p}}=\bigcap_i\overline{\dev(\widetilde{U}_i)}$ is nonempty and is contained in $\overline{\Omega}$. We proceed to prove that $\devbd{\tilde{p}}$ is contained in $\pa\Omega$. Otherwise, $\dev(\widetilde{U}_1), \dev(\widetilde{U}_2),\cdots$ accumulates at some $q\in \Omega$. The projection of $\dev^{-1}(q)\in\widetilde{\Sigma}$ to $\Sigma$ is an accumulation point of $(U_i)$, contradicting the exhausting hypothesis. 

Since each $\dev(\widetilde{U}_i)$ is connected, connectedness of $\devbd{\tilde{p}}$ follows from the simple fact that in a compact normal topological space $X$, if $(S_i)$ is a sequence of connected subsets, then $\Ac_{i\rightarrow+\infty}(S_i)$ is connected as well.  To prove this fact, we assume by contradiction that $\Ac(S_i)$ is a disjoint union of nonempty closed sets $A, B\subset X$. Then $A$ and $B$ have disjoint neighborhoods $A'$ and $B'$. Since $S_i$ is connected, each $S_i$ contains a point $x_i$ outside $A'\cup B'$. Then accumulation points of the sequence $(x_i)$, which are contained in $\Ac(S_i)$ by definition, are not contained in $A'\cup B'$, a contradiction.

(\ref{prop_devbdproperty2}) Let $\widetilde{U}_i$ and $\widetilde{U}_i'$ denote the $\tilde{p}$-lift and $\gamma.\tilde{p}$-lift of $U_i$, respectively. Then we have $\widetilde{U}_i'=\gamma. \widetilde{U}_i$, hence $\dev(\widetilde{U}_i')=\hol(\gamma).\dev(\widetilde{U}_i)$ by equivariance of $\dev$. Taking accumulation points as $i$ goes to $+\infty$, we get $\devbd{\gamma.\tilde{p}}=\hol(\gamma).\devbd{\tilde{p}}$.

%(\ref{prop_devbdproperty3})...

\end{proof}
Our proof of part (\ref{prop_devbdproperty3}) involves classifications of automorphisms of properly convex sets and will be presented in \S \ref{sec_auto} below.

\subsubsection{Developing paths}\label{sec_devlim}
We now give the details on how a path $\beta:[0,+\infty)\rightarrow\Sigma$ belonging to $\mathscr{P}_{m,p}$ develops into a path $\beta^\dev:[0,+\infty)\rightarrow\Omega$ tending to $\devbd{\tilde{p}}$, as mentioned in the introduction.

For each $t>0$, the truncated path $\beta|_{[0,t]}$ represents an element in $\widetilde{\Sigma}$ whose projection on $\Sigma$ is $\beta(t)$. We denote this point simply by $\beta_{[0,t]}$. So $t\mapsto \beta_{[0,t]}$ is the  lift of $\beta$ to $\widetilde{\Sigma}$. We then put 
$$
\beta^\dev(t):=\dev(\beta_{[0,t]}).
$$

The point $\beta_{[0,t]}\in \widetilde{\Sigma}$ eventually enters any open set $\widetilde{U}_i$ from Definition \ref{def_db} as $\tinf$. It follows from the definition that any limit point of $\beta^\dev(t)$ as $t\rightarrow+\infty$ is contained in $\devbd{\tilde{p}}$. We will only need to consider the case where there is a single limit point.

\begin{definition}\label{def_devlim}
%Given $\beta$ as above, we let $\pdev(\beta)$ denote the path which $\beta$ develops into
%\footnote{The notation ``$\dev(\beta)$'' is not to be confused with the more commonly used  notation  ``$\dev(\,\cdot\,)$'' where ``$\,\cdot\,$'' is a point in $\widetilde{\Sigma}$.}.
%Precisely, $\pdev(\beta)$ is the path in $\Omega$ which can be parametrized by
%$$
%[0,+\infty)\rightarrow\Omega,\quad t\mapsto \dev(\beta_{[0,t]}^{-1}).
%$$
If the limit
$
\lim_\tinf \beta^\dev(t) \in \devbd{\tilde{p}}
$
exists, we denoted it by $\devlim(\beta)$. Otherwise, we just say ``$\devlim(\beta)$ does not exists''.
\end{definition}
As an example, for the convex projective structure on $\mathbb{C}$ given by the \c{T}i\c{t}eica affine sphere embedding (\cf \S \ref{sec_titeica}), Proposition \ref{prop_tit} gives a description of $\devlim(\beta)$ for various $\beta$'s.

\subsection{Simple and non-simple ends}
\subsubsection{Automorphisms of properly convex sets}\label{sec_auto}
%The holonomy $\hol$ of a convex projective structure on $\Sigma$ sends a non-trivial element $\gamma$ in $\pi_1(\Sigma)$ to a projective transformation $\hol(\gamma)$ preserving the properly convex set $\Omega$. Moreover, $\hol(\gamma)$ does not fix any point in $\Omega$ because the deck action of $\gamma$ on $\widetilde{\Sigma}$ does not fix any point. 

Let $a\in\SL(3,\mathbb{R})$  be a projective transformation preserving a properly convex open set $\Omega\subset\mathbb{RP}^2$ such that $a$ does not have any fixed point in $\Omega$. It is well known (see \eg \cite{goldman_convex}) that  $a$ belongs to one of the following four classes.

\begin{enumerate}
\item\label{item_clas1} $a$ is said to be \textbf{hyperbolic} if it is conjugate to
$$
\begin{pmatrix}
\lambda_+&&\\
&\lambda_0&\\
&&\lambda_-
\end{pmatrix}
$$
with $\lambda_+>\lambda_0>\lambda_->0$ and $\lambda_+\lambda_0\lambda_-=1$. The fixed point of $a$ in $\mathbb{RP}^2$ corresponding to the eigenvalue $\lambda_+$ (resp. $\lambda_0$, $\lambda_+$)  is called the \empty{attracting} (resp.  \emph{saddle}, \emph{repelling}) fixed point and is denoted by $x_+$ (resp. $x_0$, $x_-$).

Any properly convex open set $\Omega$ preserved by such $a$ has the following properties. Both $x_+$ and $x_-$ are on the boundary $\pa\Omega$. The segment $I\subset\overline{\Omega}$ joining $x_+$ and $x_-$ is called the \emph{principal segment} of $a$. We can suppose $\pa\Omega=C_1\cup C_2$, each $C_i$ being a curve joining $x_-$ and $x_+$. Then $C_i$ is either 
\begin{itemize}
\item $I$, or 
\item the union of two segments connecting $x_-$ and $x_+$ to $x_0$, or
\item an $a$-invariant convex curve such that $C_i\setminus\{x_+,x_-\}$ is contained in an open triangle $\Delta\subset\mathbb{RP}^2$ invariant by $a$.
\end{itemize}

% For a generic point $x\in\mathbb{RP}^2$, the iterating sequence $(a^k(x))_{k\in\mathbb{Z}}$ converges to $x_+$ (resp. $x_-$) as $k\rightarrow+\infty$ (resp. $k\rightarrow-\infty$), hence $x_+$ and $x_-$ are both on the boundary $\pa\Omega$. Convexity implies that a segment joining $x_+$ and $x_-$ is contained in $\overline{\Omega}$. This segment is called the \emph{principal segment} of $a$.
\vspace{5pt} 

\item\label{item_clas2} $a$ is said to be \textbf{quasi-hyperbolic} if it is conjugate to
$$
\begin{pmatrix}
\lambda&1&\\
&\lambda&\\
&&\mu
\end{pmatrix} 
$$
with $\lambda>0$, $\lambda\neq 1$ and $\lambda^2\mu=1$. We call the fixed points corresponding to the eigenvalues $\lambda$ and $\mu$ the \emph{double fixed point} and the \emph{simple fixed point}, respectively.

For any properly convex open set $\Omega$ preserved by $a$, the boundary
 $\pa\Omega$ contains a segment $I$ joining the two fixed points, still called the \emph{principal segment} of $a$.

\vspace{5pt}

\item\label{item_clas3} $a$ is said to be \textbf{planar}  if it is conjugate to
$$
\begin{pmatrix}
\lambda&&\\
&\lambda&\\
&&\mu
\end{pmatrix}
$$
with $\lambda>0$, $\lambda\neq 1$ and $\lambda^2\mu=1$. $a$ has an isolated fixed point $x_0$ and a pointwise fixed line $l$, corresponding to the eigenvalues $\mu$ and $\lambda$, respectively.  

A properly convex open set $\Omega$ preserved by $a$ must be a triangle whose boundary contains $x_0$ as a vertex and contains a segment of $l$ as an edge. In \mbox{analogy} with the quasi-hyperbolic case, we call both edges of the triangle issuing from $x_0$ the \emph{principal segments} of $a$, and call the two vertices other than $x_0$ the \emph{double fixed points}.

\vspace{5pt}

\item\label{item_clas4}
$a$ is said to be \textbf{parabolic} if it is conjugate to
$$
\begin{pmatrix}
1&1&\\
&1&1\\
&&1
\end{pmatrix}.
$$
In this case, $a$ fixes a single point $x_0\in\mathbb{RP}^2$ and preserves a line $l$ passing through $x_0$. For any properly convex open set $\Omega$ preserved by $a$, the boundary $\pa\Omega$ contains $x_0$ and is tangent to $l$ at $x_0$.
\end{enumerate}

Dynamical properties of each type of automorphisms, such as asymptotic behavior of an iteration sequence $(a^n(x))$, are well known and are used in proving the above statements about the shape of $\Omega$ relative to fixed points, see \eg \cite{marquis} Section 2. We will make use of these dynamical properties a few times below without detailed explanation.

%...\subsubsection{Developed boundaries of convex projective annuli}
%....
%\begin{proposition}\label{prop_annuli}
%Let ....
%\end{proposition}

The following lemma singles out all possible convex projective surfaces with planar end holonomies.
\begin{lemma}[\textbf{Ends with planar holonomy}]
\label{lemma_planar}
A convex projective structure on $\Sigma$ has planar end holonomy $\hol_{\tilde{p}}$ if and only if $\Sigma$ is an annulus and the Blaschke metric $g$ and Pick differential $\ve{b}$ are in case (\ref{item_tri2}) of Proposition \ref{prop_triangle} with $R\in\ima \mathbb{R}^*$.
In this case, the developed boundaries of each puncture is a principal segment of $\hol_{\tilde{p}}$.
\end{lemma}
\begin{proof}
If a convex projective structure has a planar holonomy, then the developing image is a triangle (see \S \ref{sec_auto}), hence Proposition \ref{prop_triangle} applies. 
But among the three cases in Proposition \ref{prop_triangle}, the only one with planar end holonomy is (\ref{item_tri2}) with $R\in\ima\mathbb{R}^*$. This proves the first statement.

Let $x_1$ and $x_2$ denote the vertices of $\Delta$ other than $x_0$. Let $(y_i)_{i\geq 1}$ be a sequence of points contained in the interior of the edge $\overline{x_1x_2}$ and converging to $x_1$. Let $\Delta_i$ be the sub-triangle of $\Delta$ with vertices $x_0$, $x_1$ and $y_i$, which is invariant by the holonomy $h$.

The projective structure identifies the annulus $\Sigma$ with the quotient of $\Delta$ by  $h$. Hence the quotients the $\Delta_i$'s form an exhausting sequence of  punctured neighborhoods of a punctured $p$. By definition, the developed boundary at $p$ is $\bigcap_i\overline{\Delta_i}=\overline{x_0x_1}$.
Similarly, the developed boundary of the other puncture is $\overline{x_0x_2}$. 
\end{proof}

\subsubsection{Simple ends}
In addition to the surface $\Sigma$ and the convex projective structure  fixed at the beginning of this section, we pick a punctured $p$ of $\Sigma$ and pick $\tilde{p}\in\Fa{\Sigma,p}$ throughout the rest of this section. We shall study $\devbd{\tilde{p}}$ and $\hol_{\tilde{p}}$ (see \S \ref{sec_devbd} for the definitions).

\begin{definition}[\textbf{Simple ends}]\label{def_simpleends}
We call $p$ a \emph{simple end} of the convex projective structure
if the developed boundary $\devbd{\tilde{p}}$ is either
\begin{itemize}
\item
 a point, or
\item
a segment, or
\item 
a letter ``V'', \ie union of two non collinear segments sharing an endpoint.
%the holonomy $\hol_{\tilde{p}}$ is a hyperbolic and $\devbd{\tilde{p}}$ is the union of two segments: one joining the unstable and attracting fixed points of $\hol_{\tilde{p}}$, the other joining the unstable and repelling fixed points.
%If it is the last case, we call $p$ a $\Lambda$-end.
\end{itemize}
In the three cases, we call $p$ a \emph{cusp}, a \emph{geodesic end} and a \emph{V-end}, respectively.  
%If $\devbd{\tilde{p}}$ is not in these cases, we call $p$ a \emph{non-simple end}.
\end{definition}
Note that if $\Sigma\cong\mathbb{C}$ then $\devbd{\tilde{p}}$ is just the whole $\pa\Omega$ and the puncture $\infty$ is always a non-simple end. 

The holonomy of a general simple end is described by the following proposition.
\begin{proposition}[\textbf{Holonomy of simple ends}]\label{prop_classification}
${}$
\begin{enumerate}
\item\label{item_endclas1}
$p$ is a cusp if and only if $\hol_{\tilde{p}}$ is parabolic and $\devbd{\tilde{p}}$ is the fixed point of $\hol_{\tilde{p}}$.
\item\label{item_endclas3}
$p$ is a V-end if and only if $\hol_{\tilde{p}}$ is a hyperbolic projective transformation with saddle fixed point contained in $\devbd{\tilde{p}}$.  In this case, $\devbd{\tilde{p}}$ consists of a segment joining the saddle and attracting fixed points and a segment joining the saddle and repelling fixed points.
\item\label{item_endclas2}
$p$ is a geodesic end if and only if  $\hol_{\tilde{p}}$ is either hyperbolic, quasi-hyperbolic or planar and $\devbd{\tilde{p}}$ is a principal segment of $\hol_{\tilde{p}}$.

% If it is hyperbolic or quasi-hyperbolic then $\devbd{\tilde{p}}$ is the principal segment of $\hol_{\tilde{p}}$. 

%Furthermore, $\devbd{\tilde{p}}$ is a maximal segment on $\pa\Omega$ in the sense that there is not segment  on $\pa\Omega$ strictly containing $\devbd{\tilde{p}}$.

\end{enumerate}
\end{proposition}
\begin{proof}
(\ref{item_endclas1}) The ``if'' part is trivial. To prove the ``only if'' part, suppose by contradiction that $a:=\hol_{\tilde{p}}$ is not parabolic, so that $a$ is either hyperbolic, quasi-hyperbolic or planar and $x_0:=\devbd{\tilde{p}}$ is one of the fixed points of $a$.  For a generic point $x$ in a neighborhood of $x_0$, either $a^n(x)$ or $a^{-n}(x)$ converges to another fixed point $x_1$ of $a$ as $n\rightarrow+\infty$. It follows that the closure of each $a$-invariant set $\dev(\widetilde{U}_i)$ in Definition \ref{def_db} contains $x_1$, hence so does $\devbd{\tilde{p}}$, a contradiction. 

(\ref{item_endclas3}) 
Suppose $p$ is a V-end. By the classification in \S \ref{sec_auto}, the \mbox{projective} transformation $a$ restricts to an orientation-preserving homeomorphism of $\pa \Omega$, thus the three endpoints of the two segments are all fixed by $a$, hence $a$ is either hyperbolic or planar. The planar case is excluded by Lemma \ref{lemma_planar}. This proves the ``only if'' part.

Conversely, suppose that $a$ is hyperbolic with saddle fixed point contained in $\devbd{\tilde{p}}$. The description of $\Omega$ in \S \ref{sec_auto} for hyperbolic $a$ implies that $\Omega$ contains some open triangle $\Delta$ preserved by $a$, so that the edge $\overline{x_\pm x_0}$ of $\Delta$ joining $x_\pm$ and $x_0$ is a part of $\pa\Omega$.

Let $(U_i)$ be an exhausting sequence of punctured neighborhood of $p$ whose boundaries are embedded closed curves. Let $\widetilde{U}_i\subset\widetilde{\Sigma}$ be the lift of $U_i$ attached at $\tilde{p}$. The developing image $C_i$ of the boundary curve of $\widetilde{U}_i$ is an $a$-invariant embedded curve in $\Omega$. 
Then dynamical properties of $a$ imply that $C_i$ is contained in $\Delta$, joins $x_+$ and $x_-$ and is asymptotic to $\overline{x_\pm x_0}$ at $x_\pm$. Since $\dev(\widetilde{U}_i)$ is the connected component of $\Omega\setminus C_i$ whose closure contains $x_0$, the properties of $C_i$ imply that $\overline{\dev(\widetilde{U}_i)}\cap\pa\Omega$ is exactly $\overline{x_-x_0}\cup\overline{x_+x_0}$ and does not depend on $i$. Thus we get $\devbd{\tilde{p}}=\overline{x_-x_0}\cup\overline{x_+x_0}$. This proves the ``if'' part and the second statement.

(\ref{item_endclas2}) Both endpoints of the segment $\devbd{\tilde{p}}$ are fixed by $\hol_{\tilde{p}}$. But only hyperbolic, quasi-hyperbolic and planar projective transformations have at least two fixed points. This proves the  first statement.

The second statement is obvious when $\hol_{\tilde{p}}$ is quasi-hyperbolic. It follows from part (\ref{item_endclas3}) above
when $\hol_{\tilde{p}}$ is hyperbolic and follows from Lemma \ref{lemma_planar} when $\hol_{\tilde{p}}$ is planar. 

%To prove the third assertion, we let $I\subset\pa\Omega$ be the maximal segment containing $\devbd{\tilde{p}}$. Then both endpoints of $I$, as well as both endpoints of $\devbd{\tilde{p}}$, are fixed by $\hol_{\tilde{p}}$. But if $\hol_{\tilde{p}}$ is hyperbolic or quasi-hyperbolic then $\hol_{\tilde{p}}$ can not fixe three aligned points, hence $I=\devbd{\tilde{p}}$. If $\hol_{\tilde{p}}$ is planar, the fact that  $I=\devbd{\tilde{p}}$ follows from the next lemma, which implies that $\Omega$ is a triangle and $\devbd{\tilde{p}}$ is an edge.
\end{proof}

The next lemma gives some information about non-simple ends and will be used in Section \ref{sec_ctop} below.

\begin{lemma}[\textbf{Rough classification of non-simple ends}]\label{lemma_rough}
If $\Sigma$ is not homeomorphic to $\mathbb{C}$ and $p$ is not a simple end, then we are in one of the following cases.
\begin{enumerate}[(a)]
\item\label{item_rough1}
$\hol_{\tilde{p}}$ is hyperbolic, while $\devbd{\tilde{p}}$ is a convex curve joining the attracting fixed point $x_+$ and the repelling fixed point $x_-$ of $\hol_{\tilde{p}}$ such that $\devbd{\tilde{p}}\setminus\{x_+, x_-\}$ is contained in an open triangle $\Delta\subset\mathbb{RP}^2$ preserved by $\hol_{\tilde{p}}$.
\item\label{item_rough2}
$\Sigma$ is an annulus, $\hol_p$ is quasi-hyperbolic and $\devbd{p}=\pa\Omega\setminus I^\circ$. Here $I^\circ$ is the interior of the principal segment $I$ of $\hol_p$. 
\item\label{item_rough3}
$\Sigma$ is an annulus, $\hol_p$ is parabolic and $\devbd{p}=\pa\Omega$. \end{enumerate}
Moreover, in the second and the third case, the other end of $\Sigma$ is a geodesic end and a cusp, respectively.
\end{lemma}
\begin{proof}
We shall consider the cases $\devbd{\tilde{p}}\subsetneqq\pa\Omega$ and $\devbd{\tilde{p}}=\pa\Omega$ respectively.

\textbf{When $\devbd{\tilde{p}}\subsetneqq\pa\Omega$.} Since $\devbd{\tilde{p}}$ is a connected portion of $\pa\Omega$ invariant by $a:=\hol_{\tilde{p}}$, both endpoints are fixed by $a$, thus $a$ is either hyperbolic or quasi-hyperbolic, the planar case being excluded by Lemma \ref{lemma_planar} and the assumption that the end is not simple. 

If $a$ is hyperbolic, it follows from the descriptions of $\Omega$ for hyperbolic $a$ in \S \ref{sec_auto} and the non-simpleness assumption that we are in case (\ref{item_rough1}).  

If $a$ is quasi-hyperbolic, it follows from the descriptions of $\Omega$ for quasi-hyperbolic $a$ in \S \ref{sec_auto} and the non-simpleness assumption that $\devbd{\tilde{p}}=\pa\Omega\setminus I^\circ$. Furthermore, $\Sigma$ must be an annulus, otherwise, take $\tilde{p}'=\gamma.\tilde{p}\neq \tilde{p}$ (where $\gamma\in\pi_1(\Sigma)$), then Proposition \ref{prop_devbdproperty} (\ref{prop_devbdproperty3}) implies that $\devbd{\tilde{p}'}\subset I^\circ$, which is absurd because $\devbd{\tilde{p}'}=\hol(\gamma).\devbd{\tilde{p}}$ is not a segment. Therefore we are in case (\ref{item_rough2}).
%...By Proposition \ref{prop_annuli}, the other end is a geodesic end.

\textbf{When $\devbd{\tilde{p}}=\pa\Omega$.} Again Proposition \ref{prop_devbdproperty} (\ref{prop_devbdproperty3}) implies that $\Sigma$ can only be an annulus. 
%...By Proposition \ref{prop_annuli}, $\hol_p$ is parabolic and the other end is a cusp.

 \end{proof}

\subsubsection{Broken geodesic ends, the space $\P_0(\Sigma)$}
The main object of study of the present paper, the subspace $\P_0(\Sigma)\subset\mathcal{P}(\Sigma)$, consists of convex projective structures whose ends are either simple or ``broken geodesic ends'', as defined below.
\begin{definition}[\textbf{Twisted polygons}]
Given $a\in\SL(3,\mathbb{R})$, a \emph{$n$-gon twisted by $a$} is a piecewise linear curve in $\mathbb{RP}^2$ formed by a sequence of segments $(Y_i)_{i\in\mathbb{Z}}$ such that $Y_i$ and $Y_{i+1}$  are non-collinear segments sharing an endpoint and that $a(Y_i)=Y_{i+n}$.  

Points in $\Ac_{i\rightarrow+\infty}(Y_i)\,\cup\,\Ac_{i\rightarrow-\infty}(Y_i)$  (see \S \ref{subsubsec_devbd} for the notation) are called \emph{accumulation points} of the twisted polygon. 
\end{definition}
In particular,  a $n$-gon twisted by $a=\id$ is a piecewise linear closed curve with $n$ pieces.

\begin{definition}[\textbf{Broken geodesic ends and $\P_0(\Sigma)$}]
A puncture $p$ is said to be a \emph{broken geodesic end with $n$-pieces} if $\devbd{\tilde{p}}$ (where $\tilde{p}\in\Fa{\Sigma,p}$) consists of a $n$-gon twisted by $\hol_{\tilde{p}}$ together with the accumulation points of the twisted polygon.

We let $\P_0(\Sigma)$ denote the subset of $\P(\Sigma)$ consisting of convex projective structures on $\Sigma$ for which each puncture of $\Sigma$ is either a simple end or a broken geodesic end.
\end{definition}

The most common instance of broken geodesic ends occurs when the holonomy $\hol_{\tilde{p}}$ is hyperbolic, so that we are in case (\ref{item_rough1}) of Lemma \ref{lemma_rough} and the accumulation points of the twisted polygon are $x_+$ and $x_-$. It follows from Lemma \ref{lemma_rough} that broken geodesic ends with non-hyperbolic holonomy only occurs in the following two situations:

\pt $\Sigma$ is an annulus. In this case, a convex projective structure can have a broken geodesic end with quasi-hyperbolic or parabolic holonomy, corresponding to case  (\ref{item_rough2}) and case (\ref{item_rough3}) in Lemma \ref{lemma_rough}, respectively.

\pt $\Sigma\cong\mathbb{C}$. In this case, a convex projective structure is given by a single diffeomorphism $\dev:\Sigma\overset\sim\rightarrow\Omega\subset\mathbb{RP}^2$. The puncture $\infty$ being a broken geodesic end with $n$-pieces just means that $\Omega$ is a convex $n$-gon.

\vspace{15pt}
\section{From $\P_0(\Sigma)$ to $\V_0(\Sigma)$}\label{sec_ptoc}
%Recall from the introduction that the subset $\mathcal{C}_0(\Sigma)$ of the space of $\mathcal{C}(\Sigma)$ is defined as follows.
%
%\pt If $\Sigma$ has negative Euler characteristic, $\mathcal{C}_0(\Sigma)$ consists of those $(\ac{J},\ve{b})\in\mathcal{C}(\Sigma)$ (see Section \ref{sec_pre} for the notation) such that $\ac{J}$ is cuspidal and $\ve{b}$ has removable or meromorphic singularity at each puncture $p$ of $\Sigma$.
%
%
%\pt If $\Sigma$ has non-negative Euler characteristic, $\mathcal{C}_0(\Sigma)$ is defined in the same way except we exclude those $(\ac{J},\ve{b})$ with $\ve{b}=0$.
%
%\vspace{5pt}

The goal of this section is to prove the following theorem, which contains the statement (I) in the sketch of proof of Theorem \ref{intro_main} from the introduction. Here, a conformal structure $\ac{J}$ on $\Sigma$ is said to be \emph{cuspidal} at a puncture $p$ if a punctured neighborhood of $p$ is biholomorphic to the punctured disk $\{ 0<|z|<1\}\subset\mathbb{C}$ in such a way that $p$ correspond to $0$. 

\begin{theorem}[\textbf{$g$ and $\ve{b}$ around simple or broken geodesic ends}]\label{thm_ptoc}
Let $g$ and $\ve{b}$ be the Blaschke metric and Pick differential of a convex projective structure, respectively. Let $\ac{J}$ be the conformal structure underlying $g$. 
\begin{enumerate}
\item\label{item_ptoc2}
If $p$ is either a geodesic end, a V-end or a broken geodesic end of the convex projective structure, then 
\begin{itemize}
\item
$\ac{J}$ is cusipdal at $p$;
\item
$p$ is a pole of $\ve{b}$ of order $\geq 3$;
\item
$\lim_{x\rightarrow p}\kappa_g(x)=0$. 
%As a result, $g$ is conformally quasi-isometric to the flat metric $2^\frac{1}{2}|\ve{b}|^\frac{2}{3}$ around $p$.
\end{itemize}
\item\label{item_ptoc1}
If $p$ is a cusp of the projective structure, then
\begin{itemize}
\item 
$\ac{J}$ is cusipdal at $p$;
\item
 $p$ is a removable singularity or a pole of $\ve{b}$ of order $\leq 2$;
 \item
 $\lim_{x\rightarrow p}\kappa_g(x)=-1$;
% \item
% $g$ is conformally quasi-isometric to a 
%cuspidal hyperbolic metric around $p$.
\end{itemize}
\end{enumerate}
\end{theorem}

Part (\ref{item_ptoc2}) of the above theorem will be proved in \S \ref{sec_otherends}. Part (\ref{item_ptoc1}) is proved by Benoist and Hulin in \cite{benoist-hulin} and we will review the proof in \S \ref{sec_cusps}. An important intermediate result in their proof, which we will use again later, is presented in \S \ref{sec_cuspidal}.

\subsection{Geodesic ends, V-ends and broken geodesic ends}\label{sec_otherends}
Recall from \S \ref{sec_canometric} that, given a properly convex open set $\Omega\subset\mathbb{RP}^2$, we define the function $f_\Omega: \Omega\rightarrow\mathbb{R}_{\geq 0}$ as
$$
f_\Omega:=\kappa_\Omega+1=2\|\ve{b}_\Omega\|_{g_\Omega}^2
=\left(\frac{2^\frac{1}{3}|\ve{b}_\Omega|^\frac{2}{3}}{g_\Omega}\right)^3.
$$ 
%It captures the ratio between $2^\frac{1}{3}|\ve{b}_\Omega|^\frac{2}{3}$ and $g_\Omega$. 
The following lemma is essentially a generalisation of Proposition 3.1 in \cite{benoist-hulin} and Lemma 7.2 in \cite{dumas-wolf}.

\begin{lemma}[\textbf{Controlling $f_\Omega$ near boundary}]\label{lemma_fomega}
Let $\Omega\subset\mathbb{RP}^2$ be a properly convex open set. Assume that $l$ is a maximal line segment on $\pa\Omega$ (\ie $l=L\cap\pa\Omega$, where $L\subset\mathbb{RP}^2$ is a line) and $l$ is not a single point.
\begin{enumerate}
\item\label{item_fomega1}
For any $K>1$ and any segment $l'$ contained in the interior of $l$, there exists a quadrilateral $V\subset\Omega$ with edge $l'$, as in the following picture, such that the inequality 
\begin{equation}\label{eqn_fomega}
\frac{1}{K}\leq f_{\Omega}\leq K
\end{equation}
holds on $V$.
\begin{figure}[h]
\centering
\includegraphics[width=1.55in]{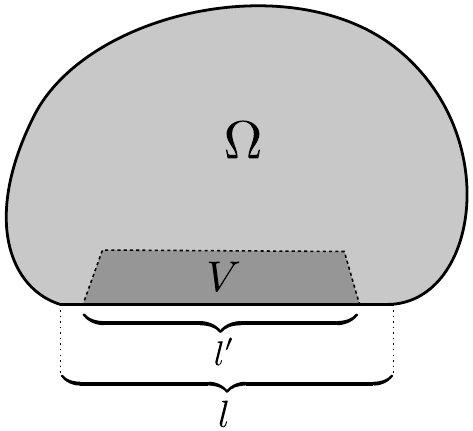}
\end{figure}
\item\label{item_fomega2}
Assume furthermore that $\pa\Omega$ is not $C^1$ at an endpoint $y_0$ of $l$, \ie the two tangent directions to $\pa\Omega$ at $y_0$ form an angle $\theta<\pi$. Then for any $K>1$, any $\theta'<\theta$ and any sub-interval $l'\subsetneqq l$ containing $y_0$, there is a quadrilateral $V$ with edge $l'$ and with angle $\theta'$ at  $y_0$, as in the following picture, such that (\ref{eqn_fomega}) holds on $V$. 
\begin{figure}[h]
\centering\includegraphics[width=1.7in]{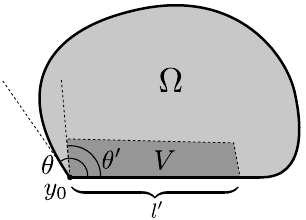}
\end{figure}
\end{enumerate}
\end{lemma}
\begin{proof}
(\ref{item_fomega1}) Fix a distance $d$ on $\mathbb{RP}^2$ compatible with the topology. We define the Hausdorff distance $d_{\mathsf{Hausdorff}}$ between two properly convex open sets $\Omega_1$ and $\Omega_2$ as the usual Hausdorff distance between the closures $\overline{\Omega}_1$ and $\overline{\Omega}_2$ with respect to $d$.

Let $\Delta\subset\Omega$ be the open triangle spanned by $l$ and a point $x_0$ in $\Omega$ (see the first picture in Figure \ref{figure_prooffomega1}). Fix $x_1\in\Delta$.  Fix $K$ and $l'$ as in the statement. By Corollary \ref{coro_bh} (\ref{item_bh2}) and the last property in \S \ref{sec_canometric}, there exists a constant $\varepsilon>0$ such that for any properly convex open set $\Omega'$ satisfying
$$
d_{\mathsf{Hausdorff}}(\Delta, \Omega')<\varepsilon,
$$
we have $x_1\in \Omega'$ and 
$$
\frac{1}{K}\leq f_{\Omega'}(x_1)\leq K.
$$

We claim that there is a quadrilateral $V$ with edge $l'$ such that, for any $x\in V$, if we let $\Phi_{x_1,x}$ denote the unique projective transformation stabilizing $\Delta$ and mapping $x$ to $x_1$, then 
\begin{equation}\label{eqn_prooffomega1}
d_{\mathsf{Hausdorf}}(\Delta, \Phi_{x_1,x}(\Omega))<\varepsilon.
\end{equation}
Such a $V$  proves the required statement because $f_{\Omega}(x)=f_{\Phi_{x_1,x}(\Omega)}(x_1)$ by projective invariance while $\frac{1}{K}\leq f_{\Phi_{x_1,x}(\Omega)}(x_1)\leq K$ by (\ref{eqn_prooffomega1}) and the assumptions on $\varepsilon$ and $V$.

To prove the claim, we work with the affine chart $\mathbb{R}^2\subset\mathbb{RP}^2$ in which the points $x_0$, $x_1$ and the two endpoints of $l$ correspond to $(0,0)$, $(1,1)$ and $(0,\infty)$, $(\infty, 0)$, respectively.   So $\Delta$ is the first quadrant of $\mathbb{R}^2$, while the segments $l_1$ and $l_2$  joining $x_0$ and the two endpoints of $l'$ are rays in $\mathbb{R}^2$ with slopes $k_1$ and $k_2$, respectively, where $0<k_1<k_2$. See the second picture in Figure \ref{figure_prooffomega1}. Put
$$
\Delta'=\left\{(x^1,x^2)\in\mathbb{R}^2\ \big|\  k_1\leq \tfrac{x^2}{x^1}\leq k_2\right\}
$$

\begin{figure}[h]
\centering\includegraphics[width=4in]{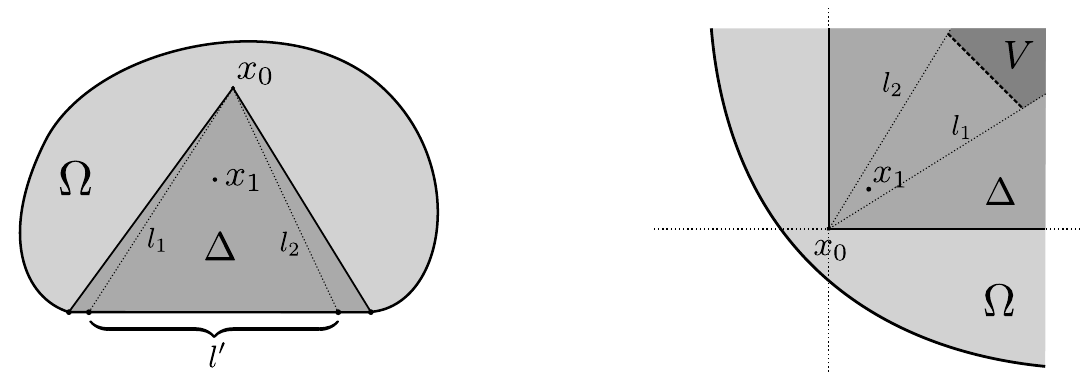}
\caption{Proof of Lemma \ref{lemma_fomega} (\ref{item_fomega1})}
\label{figure_prooffomega1}
\end{figure}

Given $a=(a^1,a^2)\in\mathbb{R}^2$ with $a^1,a^2<0$ and $0<\lambda,\mu<1$, we let $D_{a,\lambda,\mu}$ denote the convex region in $\mathbb{R}^2$ delimited by the two rays issuing from $a$ which slope $-\lambda$ and $-\mu^{-1}$, respectively. Take $x=(x^1,x^2)\in\Delta$. The previously defined projective transformation $\Phi_{x_1,x}$ restricted to the affine chart is the linear map
$$
\Phi_{x_1,x}(y^1,y^2)=\left(\frac{y^1}{x^1},\frac{y^2}{x^2}\right).
$$
One checks that 
\begin{equation}\label{eqn_prooffomega2}
\Phi_{x_1, x}(D_{a,\lambda,\mu})=D_{\left(\frac{a^1}{x^1},\frac{a^2}{x^2}\right), \frac{x^1}{x^2}\lambda,\frac{x^2}{x^1}\mu}.
\end{equation}
A crucial implication of this expression is that, when $x$ tends to infinity within $\Delta'$, the origin of the sector $\Phi_{x_1, x}(D_{a,\lambda,\mu})$ tends to $0$, while the slopes of its boundary rays remain  bounded.

Now the claim is a consequence of  the following assertions. \begin{enumerate}[(a)]
\item\label{item_claim1}
If $a'$ is close enough to $0$ and $\lambda',\mu'$ are small enough, then 
$$
d_{\mathsf{Hausdorf}}(\Delta,D_{a',\lambda',\mu'})<\varepsilon.
$$ 
\item\label{item_claim2}
For any $\lambda,\mu>0$, we can take $a$ far enough from $0$, such that $\Omega\subset D_{a,\lambda, \mu}$.
\end{enumerate}

These assertions can be easily verified by transferring back to the affine chart in the first picture of Figure \ref{figure_prooffomega1}. We omit detailed proofs. 

To prove the claim, we first fix $\lambda'$ and $\mu'$ as small as assertion (\ref{item_claim1}) requires. Set $\lambda=k_1^{-1}\lambda'$ and $\mu=k_2\mu'$. Using assertion (\ref{item_claim2}), we fix $a$ such that $\Omega\subset D_{a,\lambda, \mu}$. 
Now the expression (\ref{eqn_prooffomega2}) implies that 
%$M>0$ such that any $x$ belonging the open set
%$$
%V':=\{x=(x^1,x^2)\in\mathbb{R}^2\mid k_1<x^2/x^1<k_2,\ |x|>M\}
%$$
if we take $V\subset\Delta'$ far away enough from $0$, as in Figure \ref{figure_prooffomega1}, then for any $x\in V$, the origin $a''$ of the sector $D_{a'',\lambda'',\mu''}:=\Phi_{x_1,x}(D_{a,\lambda,\mu})$ is as close to $0$ as assertion (\ref{item_claim1}) requires, while the slopes satisfy 
$$
\lambda''=\frac{x^1}{x^2}\lambda\leq k_1^{-1}\lambda=\lambda',\quad\mu''=\frac{x^2}{x^1}\mu\leq k_2\mu=\mu'.
$$  
Thus the fact $\Omega\subset D_{a,\lambda,\mu}$ and assertion (\ref{item_claim1}) gives
$$
d_{\mathsf{Hausdorf}}(\Delta, \Phi_{x_1, x}(\Omega))\leq d_{\mathsf{Hausdorf}}(\Delta,\Phi_{x_1,x}(D_{a,\lambda,\mu}))=d_{\mathsf{Hausdorf}}(\Delta,D_{a'',\lambda'',\mu''})<\varepsilon
$$
as we wished.

(\ref{item_fomega2})  
In view of part (\ref{item_fomega1}), it is sufficient to prove that the inequality (\ref{eqn_fomega}) holds on an open triangle $V_0$ whose boundary contains $y_0$ as a vertex and contains some segment on $l$ as an edge, such that the angle of $V_0$ at $y_0$ is $\theta'$. See the first picture of Figure \ref{figure_prooffomega2}.
\begin{figure}[h]
\centering\includegraphics[width=4.4in]{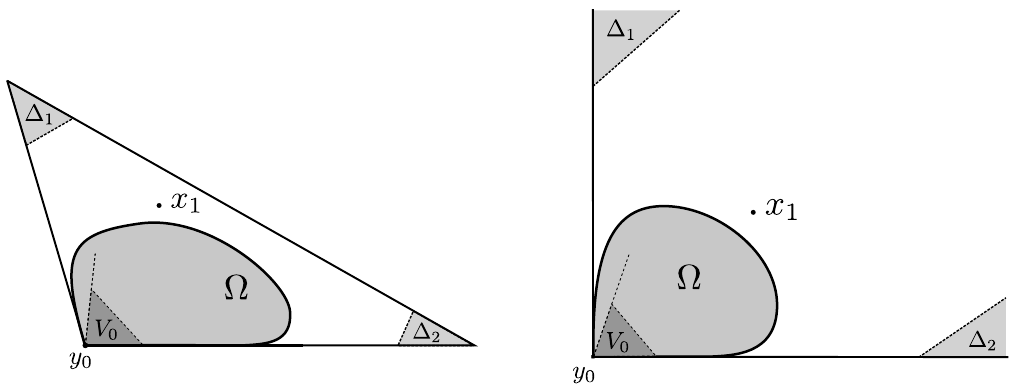}
\caption{Proof of Lemma \ref{lemma_fomega} (\ref{item_fomega2})}
\label{figure_prooffomega2}
\end{figure}

To this end, we first take an open triangle $\Delta$ containing $\Omega$ and tangent to $\Omega$ at $y_0$ as shown in the picture. Fix $x_1\in\Delta$. Let $\Phi_{x_1, x}$ and $\varepsilon$ be defined as in the proof of part (\ref{item_fomega1}). As in the proof of part (\ref{item_fomega1}), it is sufficient to show that (\ref{eqn_prooffomega1}) holds for any $x\in V_0$. 

Take small enough sub-triangles $\Delta_1, \Delta_2\subset\Delta$ as in the picture, such that for any convex subset $\Omega'\subset\Delta$ with $y_0\in\pa\Omega'$, we have $d_\mathsf{Hausdorff}(\Delta,\Omega')<\varepsilon$  whenever $\Omega'$ meets both $\Delta_1$ and $\Delta_2$.  We now only need to find a small enough $V_0$ such that for any $x\in V_0$, the convex set $\Phi_{x_1, x}(\Omega)$ meets $\Delta_1$ and $\Delta_2$, or equivalently, $\Omega$ meets $\Phi_{x, x_1}(\Delta_1)$ and $\Phi_{x, x_1}(\Delta_2)$. But this is easily done by working on an affine chart where $\Delta$ corresponds to the first quadrant (see the second picture of Figure \ref{figure_prooffomega2}), taking account of the fact that the slope of the ray delimiting  $\Phi_{x,x_1}(\Delta_1)$ is bounded because the slope of any $x\in V_0$ is bounded.
\end{proof}

%The above lemma easily implies the last statement in part (\ref{item_ptoc2}) of Theorem \ref{thm_ptoc}. We then deduce the first two statements using the following lemma and some deep results of A. Huber. 

The above lemma readily implies the last statement in part (\ref{item_ptoc2}) of Theorem \ref{thm_ptoc}, as we will see below. We need the following notion in proving the other statements.
\begin{definition}\label{def_complete}
 A Riemannian metric $g$ on $\Sigma$ is said to be \emph{complete at a puncture $p$} if for any sequence of points $(x_n)$ tending to $p$, the distance $d(x_n,y)$ tends to $+\infty$ as $n\rightarrow+\infty$. Here $y\in\Sigma$ is any reference point.
  \end{definition}
 It is easy to see that $g$ is complete in the usual sense if and only if it is complete at every puncture.

\begin{proof}[Proof of Theorem \ref{thm_ptoc} (\ref{item_ptoc2})]
Let $(\dev,\hol)$ be a developing pair of the convex projective structure and put $\Omega=\dev(\widetilde{\Sigma})$.  Fix $\tilde{p}\in\Fa{\Sigma,p}$.

We first prove the last statement. The developing map $\dev$ induces a diffeomorphism $\Sigma\overset\sim\rightarrow\Omega/\hol(\pi_1(\Sigma))$, whereas the Blaschke metric $g$ and Pick differential $\ve{b}$ are pullbacks of $g_\Omega$ and $\ve{b}_\Omega$ by this diffeomorphism. 
Therefore, given $K>1$, we only need to find an open set $V\subset\Omega$ where the inequality (\ref{eqn_fomega}) holds, such that the quotient map $\Omega\rightarrow\Omega/\hol(\pi_1(\Sigma))\cong\Sigma$  projects $V$ to punctured neighborhood of $p$.

Such $V$ can be constructed for geodesic ends, V-ends and broken geodesic ends respectively as follows. 
%(\ref{eqn_fomega}) holds on $V$ by construction, while $V$ projects to a punctured neighborhood of $p$ is easily checked by 
\begin{itemize}
\item 
Assume $\devbd{\tilde{p}}=l$ is a segment. Let $l'\subset \mathsf{int}(l)$ be a sub-interval such that 
$$
\bigcup_{n\in\mathbb{Z}}\hol_{\tilde{p}}{}^n(\mathsf{int}(l'))=\mathsf{int}(l).
$$
We can take a quadrilateral given by Lemma \ref{lemma_fomega} (\ref{item_fomega1}) to be the required $V$. 
\item
Assume $\devbd{\tilde{p}}=l_1\cup l_2$, where $l_1$ and $l_2$ are non-colinear segments. We can take a quadrilateral given by Lemma \ref{lemma_fomega} (\ref{item_fomega2}) with big enough $\theta'$ to be $V$.
\item
Assume $\devbd{\tilde{p}}=\bigcup_{k\in\mathbb{Z}}Y_k$ is a twisted $n$-gon.  $V$ is found by applying Lemma \ref{lemma_fomega} (\ref{item_fomega2}) to the $n+2$ edges $Y_0, Y_1,\cdots, Y_{n+1}$.
\end{itemize}
This completes the proof of the last assertion in (\ref{item_ptoc2}). 

\begin{lemma}
Equip the punctured disk $D^*=\{x\in\mathbb{R}^2\mid 0<|x|<1\}$ with a flat Riemannian metric $h$ and the underlying conformal structure. If $h$ is complete at $0$ and has the form $h=|\ve{b}|^\frac{2}{3}$ for some holomorphic cubic differential $\ve{b}$ on $D^*$ without zeros, then the conformal structure is cuspidal at $0$ and $\ve{b}$ has pole of order $\geq 3$ at $0$.
\end{lemma}

As a consequence, $|\ve{b}|^\frac{2}{3}$ is complete at $p$ and does not have zeros near $p$. Let us show that the conformal structure $\ac{J}$ is cuspidal at $p$. Assume by contradiction that $\ac{J}$ is not cuspidal, then a punctured neighborhood of $p$ not containing zeros of $\ve{b}$ is conformally equivalent to an annulus $\{z\in\mathbb{C}\mid 1<|z|<R\}$ in such a way that points near $p$ correspond to points near the inner boundary $\{|z|=1\}$. Suppose $\ve{b}=b(z)\dz^3\neq 0$ on the annulus.  We make use of the following theorem due to A. Huber (a less general version of Theorem 5 in \cite{huber_1}).
\begin{theorem}[Huber]\label{thm_huber1}
Given real-valued subharmonic function $u$ defined on the annulus $\{1<|z|<R\}$, put
$$
\phi(u):=\lim_{r\rightarrow0^+}\int_{|z|=r}\frac{\pa u}{\pa \ve{n}}\dif s,
$$
where $\ve{n}$ is the inner pointing normal. The limit exists in $\mathbb{R}\cup\{+\infty\}$ by a result of F. Riesz.  If $\phi(u)<+\infty$, then there exists a locally rectifiable path on the annulus tending to the inner boundary $\{|z|=1\}$ with finite length under the  Riemannian metric $e^{u(z)}|\dz|^2$.
\end{theorem}
For the harmonic function $u(z)=\frac{2}{3}\log |b(z)|$, the integral in the above definition of $\phi(u)$ is
$$
\int_{|z|=r}\frac{\pa u}{\pa \ve{n}}\dif s=\frac{2}{3}\int_{|z|=r}\frac{\pa\log|b(z)|}{\pa\ve{n}}\dif s=-\frac{2}{3}\int_{|z|=r}\dif(\arg(b(z)),
$$
where the last equality follows from the fact that $\arg(b(z))$ is a harmonic conjugate of $\log|b(z)|$. Since $\frac{1}{2\pi}\int_{|z|=r}\dif(\arg(b(z))$ is an integer not depending on $r$, we can apply Theorem \ref{thm_huber1} and deduce that $|\ve{b}|^\frac{2}{3}$ is not complete at the puncture $p$, a contradiction. Thus $\ac{J}$ is cuspidal at $p$.

Let $z$ be a conformal local coordinate such that $\{z\mid 0<|z|<1\}$ corresponds to a punctured neighborhood of $p$ and $z=0$ corresponds to $p$.  Assume $\ve{b}=b(z)\dz^3$. It remains to be shown that $z=0$ is a pole of $b(z)$ of order at least 3. To this end, we apply another theorem of Huber (Satz 4 in \cite{huber}). 
\begin{theorem}[Huber]\label{thm_satz4}
Let $u$ be a $C^2$ super-harmonic function on  $\{z\mid 0<|z|<R\}$. Then $u$ can be extended to a super-harmonic function $\{z\mid |z|<R\}\rightarrow\mathbb{R}\cup\{+\infty\}$ if and only if the following conditions are satisfied 
\begin{enumerate}[(a)] 
\item The integral of $|\Delta u|$ on some neighborhood of $0$ is finite. 
\item $\int_{\gamma}\frac{e^{u(z)}}{|z|}|\dz|=+\infty$ for any path $\gamma$ tending to $0$.
\end{enumerate}
\end{theorem}

Let $u: \{z\mid 0<|z|<1\}\rightarrow\mathbb{R}$ be the function defined by
$$
|b(z)|^{\frac{1}{3}}=\frac{e^{u(z)}}{|z|},
$$
or equivalently, $u(z)=\frac{1}{3}\log|z^3\,b(z)|$. Since $|\ve{b}|^\frac{2}{3}$ is complete at $p$ and $u$ is harmonic, conditions (a) and (b) in Theorem \ref{thm_satz4} are fulfilled, so we conclude that $u$ can be extended  to a super-harmonic function $\{z\mid |z|<1\}\rightarrow \mathbb{R}\cup\{+\infty\}$. By lower semi-continuity of super-harmonic functions, $|z^3\,b(z)|$ is bounded from below by a positive constant in vicinity of $z=0$. As a result, $b(z)$ has a pole of order at least $3$ at $z=0$.
\end{proof}

\subsection{Cuspidal hyperbolic metrics}\label{sec_cuspidal}
Two conformal Riemannian metrics are said to be \emph{conformally quasi-isometric} if their ratio has positive upper and lower bounds.
Lemma 5.2 in \cite{benoist-hulin} can be reformulated as follows. 

\begin{lemma}\label{lemma_bilip}
Let $g_1$ and $g_2$ be $C^{\infty}$ Riemannian metrics on $D^*:=\{x\in\mathbb{R}^2\mid 0<|x|<1\}$ such that $g_1$ and $g_2$ are conformal to each other and both satisfy the following conditions:
\begin{enumerate}[(i)]
\item\label{item_bilip1} complete at $0$ (see Definition \ref{def_complete});
\item\label{item_bilip2} the curvature is negatively pinched (\ie bounded from above and below by negative constants) around $0$.
\end{enumerate}
Then $g_1$ and $g_2$ are conformally quasi-isometric near $0$.
\end{lemma}

Although the original statement in \cite{benoist-hulin} has global nature, we can  deduce the above local version from it by completing $g_1$ and $g_2$ to metrics on $\mathbb{R}^2$.

As a consequence of Lemma \ref{lemma_bilip}, any metric $g$ on $D^*$ satisfying conditions (\ref{item_bilip1}) and (\ref{item_bilip2}) belongs to either of the following conformally quasi-isometry classes.
\begin{itemize}
\item 
If the conformal structure underlying $g$ is cuspidal at $0$, then $g$ is conformally quasi-isometric to the hyperbolic metric of finite volume 
\begin{equation}\label{eqn_cusphyp}
\frac{4|\dz|^2}{|z|^2(\log |z|)^2}.
\end{equation}
Here we identify a punctured neighborhood of $p$ with $\{z\in\mathbb{C}\mid0<|z|<\varepsilon\}$.
The metric (\ref{eqn_cusphyp}) is what we call a \emph{cuspidal hyperbolic metric}. The  expression is obtained from the Poincar\'e metric on the upper-half plane $\mathbb{H}$ by identifying $\{0<|z|<\varepsilon\}$ with the quotient $\{\zeta\in\mathbb{H}\mid  \im(\zeta)>-\log\varepsilon\}/\mathbb{Z}$ via $z=\exp(\ima\zeta)$. We emphasis that a different choice of the coordinate $z$ gives rise to a different cuspidal hyperbolic metric, but Lemma \ref{lemma_bilip} implies that they are conformally quasi-isometric around $0$.
\vspace{4pt}
\item
Otherwise, a punctured neighborhood of $0$ is conformally equivalent to $\{z\in\mathbb{C}\mid 1<|z|<R\}$.
%in such a way that points of $D^*$ near $0$  corresponding to points near the inner circle $\{|z|=1\}$.
In this case, $g$ is conformally quasi-isometric to a hyperbolic metric of infinite volume.
\end{itemize}

\subsection{Cusps}\label{sec_cusps}
We briefly review here the proof of part (\ref{item_ptoc1}) of Theorem \ref{thm_ptoc} due to  Benoist and Hulin \cite{benoist-hulin}. The method is similar to the proof of part (\ref{item_ptoc2}).

\begin{proof}[Outline of proof of Theorem \ref{thm_ptoc} (\ref{item_ptoc1})]
Put $f=\kappa_g+1:\Sigma\rightarrow\mathbb{R}$ as before. We first prove the assertion about $\kappa_g$ by showing that, for any $\varepsilon>0$, there exists a punctured neighborhood $U$ of $p$ on which $|\fun|<\varepsilon$. 

%By Lemma \ref{lemma_bilip}, this would imply the last statement in Theorem \ref{thm_ptoc} (\ref{item_ptoc1}).

 The developed boundary $\devbd{\tilde{p}}$ consists of a single point $x_0$ by assumption. Under a suitable coordinate system of $\mathbb{RP}^2$, we can write
$$
x_0=\begin{bmatrix}1\\0\\0\end{bmatrix}, \quad \hol_{\tilde{p}}=\begin{pmatrix}
1&1&\frac{1}{2}\\
0&1&1\\
0&0&1
\end{pmatrix}.
$$
Then the holonomy $\hol_{\tilde{p}}$ preserves a family of conic curves $(C_\lambda)_{\lambda\in\mathbb{R}}$, where
$$
C_\lambda=\left\{\begin{bmatrix}
x\\
y\\
z
\end{bmatrix}\in\mathbb{RP}^2
\ \Big|\
y^2+\lambda z^2=2xz
\right\}.
$$
Let $E_\lambda\subset\mathbb{RP}^2$ denote the open ellipse with $\pa E_\lambda=C_\lambda$. 
Using convexity of $\Omega$ and invariance of $\Omega$ under $\hol_{\tilde{p}}$, one can show that there exists $\lambda_0<\lambda_1$ such that $E_{\lambda_0}\supset \Omega\supset E_{\lambda_1}$, see the following picture. By conjugating the developing pair by a $x_0$-fixing projective transformation, we can assume $\lambda_0=0$. 
%See the following illustrations in different affine charts. The second affine chart is $\{[x:y:1]\mid x,y\in\mathbb{R}\}=\mathbb{R}^2\subset\mathbb{RP}^2$, in which the $C_\lambda$'s are horizontal translates of a parabolic.
\begin{figure}[h]
\centering\includegraphics[width=1.8in]{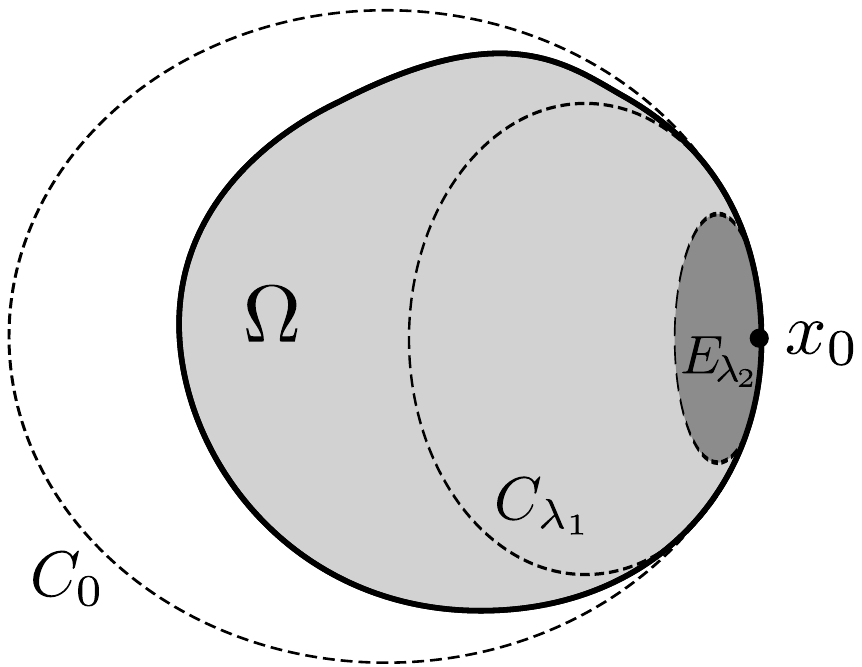}
%\caption{Proof of Lemma \ref{lemma_fomega} (\ref{item_fomega2})}
%\label{figure_prooffomega2}
\end{figure}

For any $\lambda_2\geq \lambda_1$, under the covering map $\Omega\rightarrow\Omega/\hol(\pi_1(\Sigma))\cong \Sigma$,  , the open set $E_{\lambda_2}$ projects to a punctured neighborhood of $p$. Hence, in order to prove the required statement, it is sufficient to show that if $\lambda_2$ is big enough then
\begin{equation}\label{eqn_ptoc1proof}
|f_\Omega(x)|<\varepsilon,\quad \forall x\in E_{\lambda_2}.
\end{equation}

To this end, let $G\subset\Aut(E_0)\subset\SL(3,\mathbb{R})$ denote the group of $x_0$-fixing orientation-preserving projective automorphisms of $E_0=E_{\lambda_0}$,
%Note that $G$ is a semi-direct product of its one-parameter subgroups
%$$
%H=\left\{h_\mu:=\begin{pmatrix}
%\mu^{-1/2}&&\\
%&1&\\
%&&\mu^{1/2}
%\end{pmatrix}
%\right\}_{\mu\in\mathbb{R}_+},\quad P=\left\{\begin{pmatrix}
%1&a&\frac{a^2}{2}\\
%0&1&a\\
%0&0&1
%\end{pmatrix}\right\}_{a\in\mathbb{R}}.
%$$
%We can view $E_0$ as the Beltrami-Klein model of hyperbolic plane and view each $C_\lambda$ with $\lambda>0$ as a horocycle attached at $x_0$. $P$ preserving each $C_\lambda$, while $H$ acts transitively on the family of horocycles $(C_\lambda)_{s>0}$.  Indeed, $h_{\mu}(C_\lambda)=C_{\mu\lambda}$.
which acts freely and transitively on $E_0$. Fix $x_1\in E_0$. For any $x\in E_0$, let $\Phi_{x_1, x}$ denote the unique element of $G$ sending $x$ to $x_1$.   
The required inequality (\ref{eqn_ptoc1proof}) is proved with a similar argument as in part (\ref{item_ptoc2}): 
on one hand, there exists $\delta>0$ such that for a properly convex open set $\Omega'\subset\mathbb{RP}^2$, 
$$
d_{\mathsf{Hausdorff}}(E_0, \Omega')<\delta\ \Longrightarrow \ x_1\in\Omega',\  |f_{\Omega'}(x_1)|<\varepsilon\,;
$$
on the other hand, we can choose $\lambda_2$ big enough such that 
$$
d_{\mathsf{Hausdorff}}(E_0, \Phi_{x_1,x}(\Omega))<\delta,\quad \forall x\in E_{\lambda_2}.
$$
This implies (\ref{eqn_ptoc1proof}) because $f_{\Omega}(x)=f_{\Phi_{x_1,x}(\Omega)}(x_1)$.
Thus we have shown that $\kappa_g(x)$ tends to $-1$ as $x\rightarrow p$.

It is proved in \cite{marquis} that a punctured neighborhood $U$ of $p$ has finite volume with respect to the Hilbert metric if and only if it $p$ is a cusp.
By Corollary \ref{coro_bh} (\ref{item_bh1}), $U$ has finite volume with respect to the Blaschke metric $g$ as well. Since $g$ has negatively pinched curvature on $U$, the discussions following Lemma \ref{lemma_bilip} imply that the conformal structure of $g$ is cuspidal at $p$ and $g$ is conformally quasi-isometric to a cuspidal hyperbolic metric around $p$. 

Let $z$ be a conformal local coordinate centered at $p$ and assume $\ve{b}=b(z)\dz^3$ near $p$.  Since $g$ is conformally quasi-isometric to the cuspidal hyperbolic metric (\ref{eqn_cusphyp})  while the ratio between $|\ve{b}|^\frac{2}{3}$ and $g$ is bounded from above by Corollary \ref{coro_bh} (\ref{item_bh2}), we have
$$
|b(z)|^\frac{2}{3}<\frac{C}{|z|^2(\log |z|)^2}
$$
whenever $|z|$ is small. It follows that $\ve{b}$ has a pole of order $\leq 2$ at $p$ and the proof is complete.
\end{proof}

\section{Local model for $(\Sigma,\ve{b})$ around poles of order $\geq 4$}\label{sec_model}
Let $\Sigma$ be a punctured Riemann surface of finite type, \ie $\Sigma$ is obtained from a closed Riemann surface $\overline{\Sigma}$ by removing finitely many punctures. Let $\ve{b}$ be a holomorphic cubic differential not vanishing constantly on $\Sigma$, such that each puncture is a removable singularity or a pole. 

It is well known that for each puncture $p$, there exists a conformal local coordinate $z$ of $\overline{\Sigma}$ centered at $p$ such that $\ve{b}$ has the following normal form.
\begin{equation}\label{eqn_normalform}
\ve{b}
=\begin{cases}
z^m\dz^3\quad  (m\in\mathbb{Z}\setminus3\mathbb{Z}_-)&\mbox{ if $p$ is a removable singularity or a }\\
&\mbox{pole of order not divisible by $3$,} \\[8pt]
\frac{R}{z^3}\dz^3\quad(R\in\mathbb{C}^*) &\mbox{if $p$ is a third order pole,}\\[8pt]
\left(\frac{1}{z^{l}}+\frac{A}{z}\right)^3\dz^3\quad (A\in\mathbb{C})  &\mbox{ if $p$ is a pole of order $3l$, where $l\geq 2$.}
\end{cases}
\end{equation}
See \cite{strebel} \S 6 for a proof for quadratic differentials, which readily adapts to cubic differentials. We always let $z$ denote such a coordinate and let $U=\{0<|z|<a\}$ denote a punctured neighborhood of $p$ where $z$ is defined.

This section is local in nature -- we only study around a pole $p$ of order $n+3\geq 4$. The goal is to elaborate a construction of a local model for $(\Sigma,\ve{b})$ used in Section \ref{sec_ctop} below.

\subsection{Special directions and sectors in $\T_p\overline\Sigma$}\label{sec_special}
Identifying $\T_p\overline\Sigma$ with $\mathbb{C}$ via the coordinate $z$ introduced above, 
% the $n$ equally spaced rays $d_{\k}$ ($\k\in\mathbb{Z}/n\mathbb{Z}$) in $\T_p\overline\Sigma$ issuing from the origin
we set, for each $\k\in\mathbb{Z}/n\mathbb{Z}$
\footnote{
We always use the boldfaced letter $\k$ to denote an element in $\mathbb{Z}/n\mathbb{Z}$ and let $k$ denote an integer in the congruence class $\k$. This will be convenient in Section \ref{sec_ctop} below.}, 
%,
%$$
%d_{\k}:=e^{2\pi\ima k/n}\mathbb{R}_{\geq 0}.
%$$
%We call these rays \emph{critical directions}.
% The sector delimited by the adjacent critical directions $d_{\k-1}$ and $d_\k$ is denoted by
$$
C_{\k}:=\left\{v\in \T_p\overline{\Sigma}\cong\mathbb{C}\,\Big|\, v\neq 0,\ \frac{2\pi(k-1)}{n}<\arg(v)<\frac{2\pi k}{n}\right\}
$$
and let $C_{\k,\k+1}$ denote the ray along which $C_\k$ and $C_{\k+1}$ are adjacent, \ie
$$
C_{\k,\k+1}:=e^{2\pi\ima k/n}\mathbb{R}_+.
$$
%Roughly speaking, the $D_\k$'s (resp. the $D_{\k, \k+1}$'s) gives rise to the vertices of a twisted $n$-gon, whereas critical directions corresponds to the edges.

%There is another type of distinguished directions in $\T_p\overline{\Sigma}$, defined as follows. 

For each $\k$, consider the three rays which divide $C_\k$ into four equally angled sectors. The two aside rays among the three are called \emph{unstable directions} and are denoted  by $U_\k^-$ and $U_\k^+$, respectively, in counter-clockwise order. In other words,
$$
U_\k^-:=e^{(4k-3)\pi\ima/2n}\mathbb{R}_+,\quad U_\k^+:=e^{(4k-1)\pi\ima/2n}\mathbb{R}_+.
$$

There are $2n$ unstable directions in total, dividing $\T_p\overline\Sigma\setminus\{0\}$ into $2n$ open sectors, called \emph{stable sectors}. 
We let $S_{\k}$ denote the stable sector contained in $C_{\k}$ and let $S_{\k,\k+1}$ denote the stable sector containing $C_{\k,\k+1}$. In other words,
$$
S_\k=\Big\{v\in\T_p\overline{\Sigma}\ \Big|\ \frac{(4k-3)\pi}{2n}<\arg(v)< \frac{(4k-1)\pi}{2n}\Big\},
$$
$$
S_{\k,\k+1}=\Big\{v\in\T_p\overline{\Sigma}\ \Big|\ \frac{(4k-1)\pi}{2n}<\arg(v)< \frac{(4k+1)\pi}{2n}\Big\}.
$$

The following figure illustrate the case $n=5$. 
\begin{figure}[h]
\centering
\includegraphics[width=1.9in]{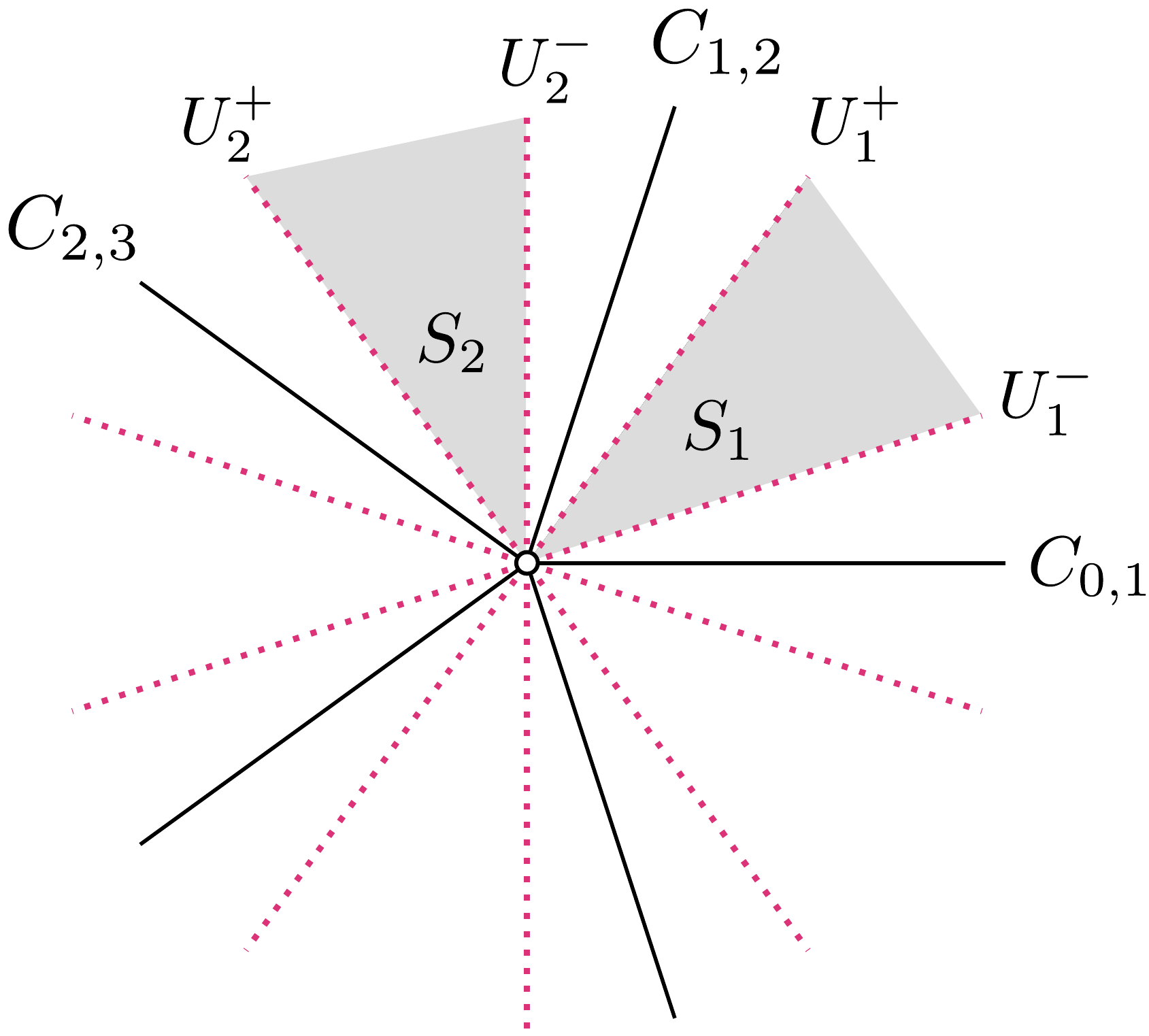}
\end{figure}
%The significance of unstable directions lies deeper: they are directions along which the parallel transport of the connections $\D$ and $\D_0$ (see \S \ref{sec_wang} and \S \ref{sec_wangass} for the definitions) are comparable. Unstable directions play similar roles as \emph{anti-Stokes directions} in the theory of irregular meromorphic connections \cite{}, although an explicit link between them is still to be uncovered.

Significance of these special directions and sectors will be clear in Section \ref{sec_ctop}. Let us roughly mention here that the $C_\k$'s correspond to the vertices of the twisted polygon asserted by Theorem \ref{intro_main}, while the $C_{\k,\k+1}$'s correspond to the edges.

\subsection{Natural half-planes}\label{sec_naturalhalf}
Let $\mathbf{H}:=\{\zeta\mid \re(\zeta)>0\}$ be the right half-plane. Following \cite{dumas-wolf}, we introduce certain conformal maps $\mathbf{H}\rightarrow U$, referred to as ``natural half-planes'',  which pull $\ve{b}$ back to the constant cubic differential $2\,\dzeta^3$. 

We first describe how these maps looks like before going into the construction. Actually we will build $n$ maps 
$$
\Phi_{\k}:\H\rightarrow U=\{0<|z|<a\}\subset\Sigma, \quad \k\in\mathbb{Z}/n\mathbb{Z}.
$$
The tangent cone of the region $\Phi_\k(\H)$ at $p$ is the sector $S_{\k-1,\k}\cup C_\k\cup S_{\k,\k+1}$, while the boundary $\Phi_\k(\pa\H)$ of the region is a curve asymptotic to the unstable directions $U_{\k-1}^+$ and $U_{\k+1}^-$. See the picture below, where the image of $\Phi_1$ is shown in the cases $n=1$, $2$ and $5$ respectively. 
\begin{figure}[h]
\centering
\includegraphics[width=4in]{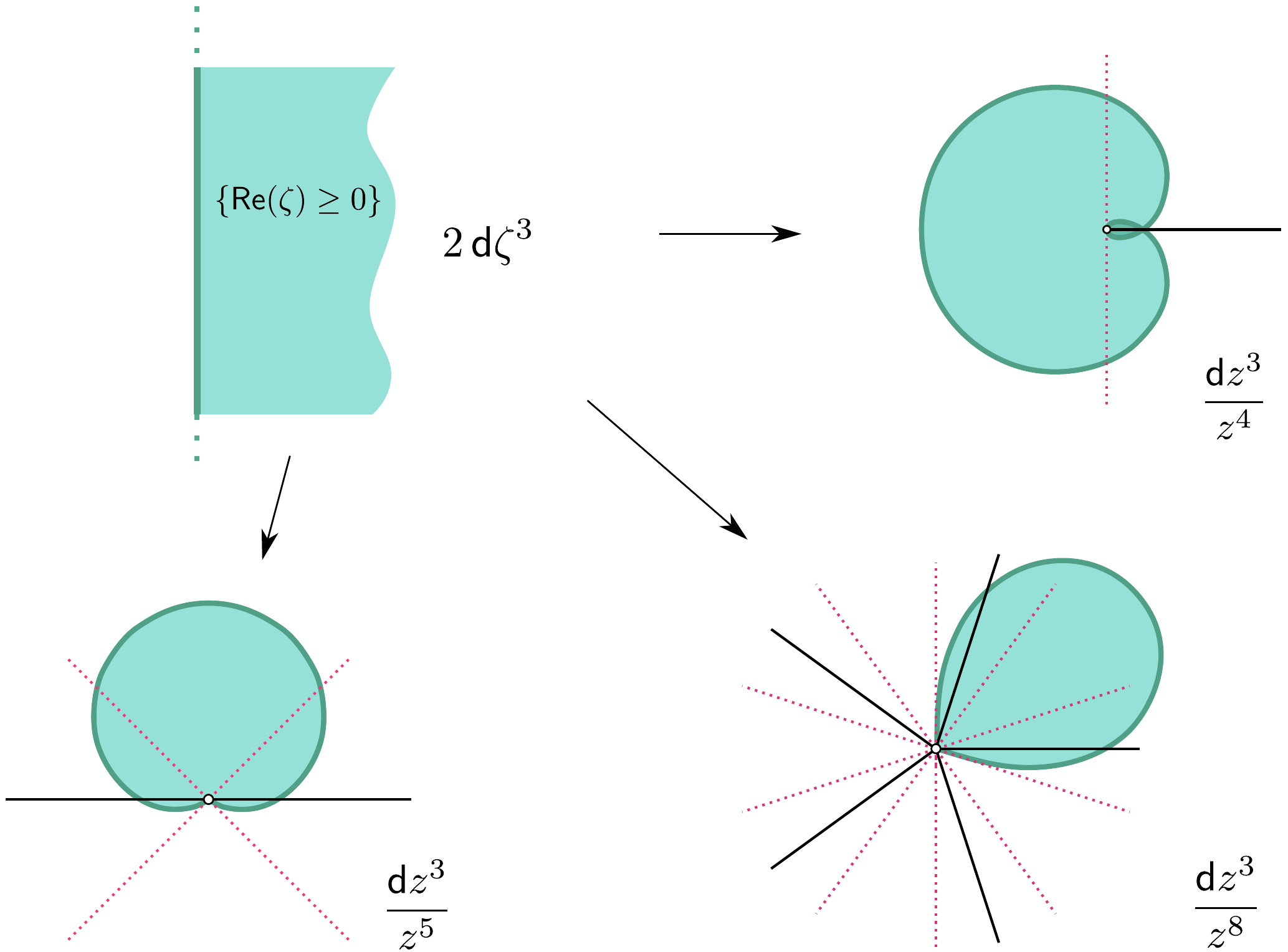}
\end{figure}

The $n=1$ case is special in that there is a single map $\Phi=\Phi_1$, which is not injective and the image covers a punctured neighborhood of $p$. When $n\geq 2$, each $\Phi_\k$ is injective and the union of images covers a punctured neighborhood.

To construct $\Phi_\k$, we first consider the case where $n$ is not divisible by $3$. By definition of the coordinate $z$ at the beginning of this section,  in this case we have $\ve{b}=\frac{\dz^3}{z^{n+3}}$.
We define
\footnote{
Here and below, we take the branch of $\log(\zeta)$ such that
$
-\pi<\im\left(\log(\zeta)\right)<\pi$
for any $\zeta\in\mathbb{C}\setminus\mathbb{R}_{\leq 0}$.
}
%$$
%\ve{b}=A z^{-(n+3)}\dz^3=2(2^{-\frac{1}{3}}z^{-\frac{n}{3}-1}\dz)^3=2\left(\dif\Big(-\frac{3\cdot2^{-\frac{1}{3}}}{n} z^{-\frac{n}{3}}\Big)\right)^3
%$$ 
%just take 
%$$
%\zeta=-\frac{3\cdot2^{-\frac{1}{3}}}{n}z^{-\frac{n}{3}}
%$$
\begin{equation}\label{eqn_chart}
\Phi_\k(\zeta):
%=\left(-\frac{\omega n}{3\cdot2^{-\frac{1}{3}}}\zeta\right)^{-\frac{3}{n}}
=\left(\frac{2^{\frac{1}{3}}n}{3}\right)^{-\frac{3}{n}}\exp\left(-\frac{3}{n}\log(\zeta+B)+\frac{(2k+1)\pi\ima}{n}\right),
\end{equation}
where $B>0$ is big enough such that $\Phi_\k(\H)\subset U$.

A straightforward computation shows that $\Phi_\k$ pulls $\ve{b}$ back to
$$
\Phi_\k^*(\ve{b})=2\,\dzeta^3.
$$ 
The computation can be done, for example, through the following equalities.
\begin{equation}\label{eqn_pull1}
\Phi_\k^*\left(\frac{\dz}{z}\right)=-\frac{3}{n}\cdot\frac{\dzeta}{(\zeta+B)},
\end{equation}
\begin{align}
\Phi_\k^*\left(\frac{1}{z^\frac{n}{3}}\right)&=\Phi_\k(\zeta)^{-\frac{n}{3}}
=\frac{2^\frac{1}{3}n}{3}\exp\left(\log(\zeta+B)-\frac{2k+1}{3}\pi\ima\right)\label{eqn_pull2}\\
&=\frac{2^\frac{1}{3}n}{3}e^{-\frac{2k+1}{3}\pi\ima}(\zeta+B).\nonumber
\end{align}

% $\tilde{\beta}$ satisfies  
%\begin{equation}\label{eqn_tildebeta}
%\tilde{\beta}(0)=0,\quad\lim_\tinf \tilde{\beta}(t)=\infty,\quad \theta(\tilde{\beta}):=\lim_\tinf \arg(\tilde{\beta}(t))\in \left(-\tfrac{\pi}{2},\tfrac{\pi}{2}\right).
%\end{equation}
%Moreover, the asymptotic argument $\theta(\tilde{\beta})$ is a linear function of $\vartheta(\beta)$,
%$$
%\theta(\tilde{\beta})=-\frac{n}{3}\vartheta(\beta)+...
%$$
%The constant term here is set to make $\theta(\tilde{\beta})$ run over $(-\tfrac{\pi}{2},\tfrac{\pi}{2})$ when $\vartheta(\beta)$ runs over (\ref{eqn_thetainzeta}).

When $n$ is a multiple of $3$, we can still construct $\Phi_\k$ with similar properties, as stated in the next proposition. Here we construct the maps on the closed half-plane $\CH=\{\re(\zeta)\geq 0\}$ rather than the open one in order to make sense of $\Phi_k(0)$.

\begin{proposition}[\textbf{Natural half-planes}]\label{prop_natural}
There exists $n$ conformal maps
$$
\Phi_\k:\CH\rightarrow U\subset\Sigma, \quad \k\in\mathbb{Z}/n\mathbb{Z}
$$
such that
\begin{enumerate}[(i)]
\item\label{item_natural1}
$\Phi_\k^*(\ve{b})=2\,\dzeta^3$.
%\item\label{item_natural2}
%$\Phi_\k$ can be extended to a conformal map from a bigger half-plane $\{\zeta\mid \re(\zeta)>-\epsilon\}$ ($\epsilon>0$) to $U$ which still satisfies (\ref{item_natural2}).

\item\label{item_natural3}
Given a parametrized path 
$$
\beta: [0,+\infty)\rightarrow\Sigma, \quad  \lim_\tinf \beta(t)=p,
$$ 
 if $\beta$ is asymptotic
 \footnote{Here the notion of ``asymptotic ray'' is the same as the one defined in the paragraph preceding Proposition \ref{prop_tit} (although in a different context).}
 to a ray $\mathbb{R}_{\geq 0}\cdot v\subset\T_p\overline{\Sigma}$ for some $v\in S_{k-1,k}\cup C_k\cup S_{k,k+1}$, then $\beta([t_0,+\infty))$ is contained in $\Phi_\k(\CH)$ for   sufficiently large $t_0$. Moreover, $\beta|_{[t_0,+\infty)}$ pulls back through $\Phi_\k$
to a path $\tilde{\beta} :[t_0,+\infty)\rightarrow\CH$ tending to $\infty$ with asymptotic argument
$$
\theta(\tilde{\beta}):=\lim_\tinf\arg(\tilde{\beta}(t))=-\frac{n}{3}\arg(v)+\frac{(2k+1)\pi}{3}.
$$

\item\label{item_natural4}
If $n\geq 2$, each $\Phi_\k$ is injective. The union of images $\bigcup_{\k\in\mathbb{Z}/n\mathbb{Z}}\Phi_\k(\H)$ is a punctured neighborhood of $p$. Moreover,
the pre-image of $\Phi_{\k}(\CH)\cap\Phi_{\k+1}(\CH)$ by $\Phi_{\k}$ and  $\Phi_{\k+1}$ are sectors of the form
$$
\left\{\zeta\in\H\,\Big|\, -\tfrac{\pi}{2}\leq\arg(\zeta- a)\leq-\tfrac{\pi}{6}\right\},\quad \left\{\zeta\in\H\,\Big|\, \tfrac{\pi}{6}\leq\arg(\zeta- b)\leq\tfrac{\pi}{2}\right\},
$$
respectively, for some $a, b\in\pa\H$, whereas the transition map $\Phi_{\k+1}\circ\Phi_{\k}^{-1}$ from the former sector to the latter is the map
\begin{equation}\label{equation_map}
\zeta\mapsto \omega \zeta-\omega a+b,\quad \omega=e^{2\pi \ima/3}.
\end{equation}

\item\label{item_natural5}
If $n=1$, the image of $\Phi=\Phi_1$ is a punctured neighborhood of $p$. Moreover, there are two disjoint sectors of the same form as in (\ref{item_natural4}), such that $\Phi(\zeta_1)=\Phi(\zeta_2)$ if and only if $\zeta_1$ and $\zeta_2$ lie in the two sectors respectively and $\zeta_1$ is sent to $\zeta_2$ by the map (\ref{equation_map}).

\end{enumerate}
\end{proposition}

The main consequence that we draw from the expression of $\theta(\tilde{\beta})$ in (\ref{item_natural3}) is the following table, which gives the range of $\theta(\tilde{\beta})$ when $v$ belongs to each of the special directions/sectors, respectively.

\vspace{5pt}

\begin{tabular}{|c||c|c|c||c|c|c|c|c|}
\hline
$v$&$C_{k,k+1}$&$C_k$&$C_{k-1,k}$&
$S_{k,k+1}$&$U_k^+$&$S_{k}$&$U_k^-$&$S_{k-1, k}$
\\[3pt]\hline
$\theta(\tilde{\beta})$&$-\tfrac{\pi}{3}$&$(-\tfrac{\pi}{3},\tfrac{\pi}{3})$&$\tfrac{\pi}{3}$&
$(-\tfrac{\pi}{2},\tfrac{\pi}{6})$&$-\tfrac{\pi}{6}$&$(-\tfrac{\pi}{6},\tfrac{\pi}{6})$&$\tfrac{\pi}{6}$&$(\tfrac{\pi}{6},\tfrac{\pi}{2})$\\[3pt] \hline
\end{tabular}
\vspace{5pt}

\begin{proof}
When $n$ is not divisible by $3$, we have already seen that the $\Phi_\k$ defined by (\ref{eqn_chart}) satisfies properties (\ref{item_natural1}). The other properties are also staightforward.

When $n$ is divisible by $3$, we put $n'=n/3$. The expression of $\ve{b}$ at the beginning of this section can be written as
$$
\ve{b}=\left(\frac{1}{z^{n'+1}}+\frac{A}{z}\right)^3\dz^3.
%=\left(\frac{1}{z^{3n'+3}}+\frac{3A}{z^{2n'+3}}\right)\dz^3.
$$
We use a construction from \cite{dumas-wolf} Appendix A. For $\epsilon>0$ small, put
$$
\H_{\epsilon}=\{\xi\in\mathbb{C}\mid \arg(\xi)\in(-\pi/2-\epsilon,\pi/2+\epsilon)\}.
$$
The key property of $\H_\epsilon$ is that any $\xi_1,\xi_2\in\H_\epsilon$ are joint by a path in $\H_\epsilon$ with length inferior to $\lambda|\xi_1-\xi_2|$, where $\lambda>1$ is a constant determined by $\epsilon$.

The map $\Phi_\k:\H\rightarrow U$ will be constructed as a composition
$$
\Phi_\k=\Psi_\k\circ\phi:\H\overset{\phi}\longrightarrow\H_\epsilon\overset{\Psi_\k}\longrightarrow U,
$$
where $\phi$ and $\Psi_\k$ are defined as follows
\begin{itemize}
\item $\Psi_\k$ is just the previously defined map (\ref{eqn_chart}) extended to $\H_\epsilon$ by the same \mbox{expression}, 
%$$
%\Psi_\k(\xi)
%%=\left(-\frac{\omega n}{3\cdot2^{-\frac{1}{3}}}\zeta\right)^{-\frac{3}{n}}
%=\left(\frac{n}{3\cdot2^{-\frac{1}{3}}}\right)^{-\frac{3}{n}}\exp\left(-\frac{3}{n}\log(\xi+B)+\frac{(2k+1)\pi\ima}{n}\right).
%$$
where we take $B$ big enough  to ensure $\Psi_\k(\H_\epsilon)\subset U$. Using (\ref{eqn_pull1}) and (\ref{eqn_pull2}), one checks that $\Psi_\k$ pulls $\ve{b}$ back to 
%$$
%\Psi^*(\ve{b})=2\left(1+\frac{3A}{2^\frac{1}{3}n}\cdot\frac{1}{\xi+B}\right)\dif\xi^3
%$$
$$
\Psi_\k^*(\ve{b})=2\left(1+\frac{C}{\xi+B}\right)^3\dif\xi^3
$$
for a constant $C\in\mathbb{C}$ depending on $A$ but not on $B$. Put $f(\xi):=1+\frac{C}{\xi+B}$.
We take $B$ further bigger such that $|f(\xi)-1|=\big|\frac{C}{\xi+B}\big|<\lambda^{-1}$ for any $\xi\in \H_\epsilon$. 
\item
Put $F(\xi):=\xi+C\log(\xi+B)$. This is a primitive of $f$.
We define $\phi$ as
$$
\phi(\zeta)=F^{-1}(\zeta+D)  
$$
for a big enough constant $D>0$ which we determine later.
\end{itemize}
In order for $\phi$ to make sense, we need to show that $F$ maps $\H_\epsilon$ injectively to a domain in $\mathbb{C}$ containing $\H+D$.

To check injectivity, assume by contradiction that $F(\xi_1)=F(\xi_2)$ for distinct $\xi_1, \xi_2\in\H_\epsilon$. Let $\gamma$ be a path going from $\xi_1$ to $\xi_2$ with length less than $\lambda|\xi_1-\xi_2|$ as mentioned above. Then
$$
0=F(\xi_2)-F(\xi_1)=\int_{\gamma}f(\xi)\dif\xi=\xi_2-\xi_1+\int_{\gamma}(f(\xi)-1)\dif\xi.
$$
But 
$$
\left|\int_\gamma(f(\xi)-1)\dif\xi\right|\leq \sup_{\xi\in\H_\epsilon}|f(\xi)-1|\cdot\mathsf{Length}(\gamma)< \lambda^{-1}\cdot \lambda|\xi_1-\xi_2|=|\xi_1-\xi_2|,
$$
a contradiction. $F$ is therefore injective on $\H_\epsilon$. 

Furthermore, $F(\H_\epsilon)$ contains $\H+D$ for some $D>0$ because $\lim_{\xi\rightarrow\infty}F(\xi)=\infty$ and $\sup_{\xi\in\pa\H_\epsilon}\re(F(\xi))<+\infty$. This finishes the construction of $\Phi_\k$.

We briefly outline the proofs of the properties (\ref{item_natural1}) to  (\ref{item_natural5}) for $\Phi_\k$. Property (\ref{item_natural1}) is an immediate consequence of the above expressions of $\Psi_\k^*(\ve{b})$ and $\phi$. Property (\ref{item_natural3}) \mbox{follows} from the explicit expression of $\Psi_\k$ and the fact that the map $\phi$ preserves asymptotic arguments. As for (\ref{item_natural4}), firstly, the injectivity follows from injectivity of both $\phi$ and $\Psi_\k$. Secondly, $\phi(\H)$ contains a half-plane of the form $\H+D'$ for some $D'>0$, but the union $\bigcup_{\k}\Psi_\k(\H+D')$ is a punctured neighborhood of $p$. Lastly, to prove the assertion about the pre-image of $\Phi_{\k}(\CH)\cap\Phi_{\k+1}(\CH)$, it is sufficient to show that the curve $\Phi_{\k+1}^{-1}(\Phi_{\k}(\pa\H))$ (resp. $\Phi_\k^{-1}(\Phi_{\k+1}(\pa\H))$  in $\CH$ is a ray with asymptotic argument $\tfrac{\pi}{3}$ (resp. $-\tfrac{\pi}{3}$). This curve is a ray because each $\Phi_\k(\pa\H)$ is a geodesic with respect to the flat metric $|\ve{b}|^\frac{2}{3}$,
% (this can be seen by noting that $\Phi_\k$ can be extended to a bigger half-plane $\{\zeta\geq -\epsilon\}$,
 %still pulling $\ve{b}$ back to $2\,\dzeta^3$),
 whereas the asymptotic argument $\pm\tfrac{\pi}{3}$ can be obtained using property (\ref{item_natural3}). Property (\ref{item_natural5}) is proved similarly as  (\ref{item_natural4}).
 \end{proof}

\begin{figure}[h]
\centering\includegraphics[width=5in]{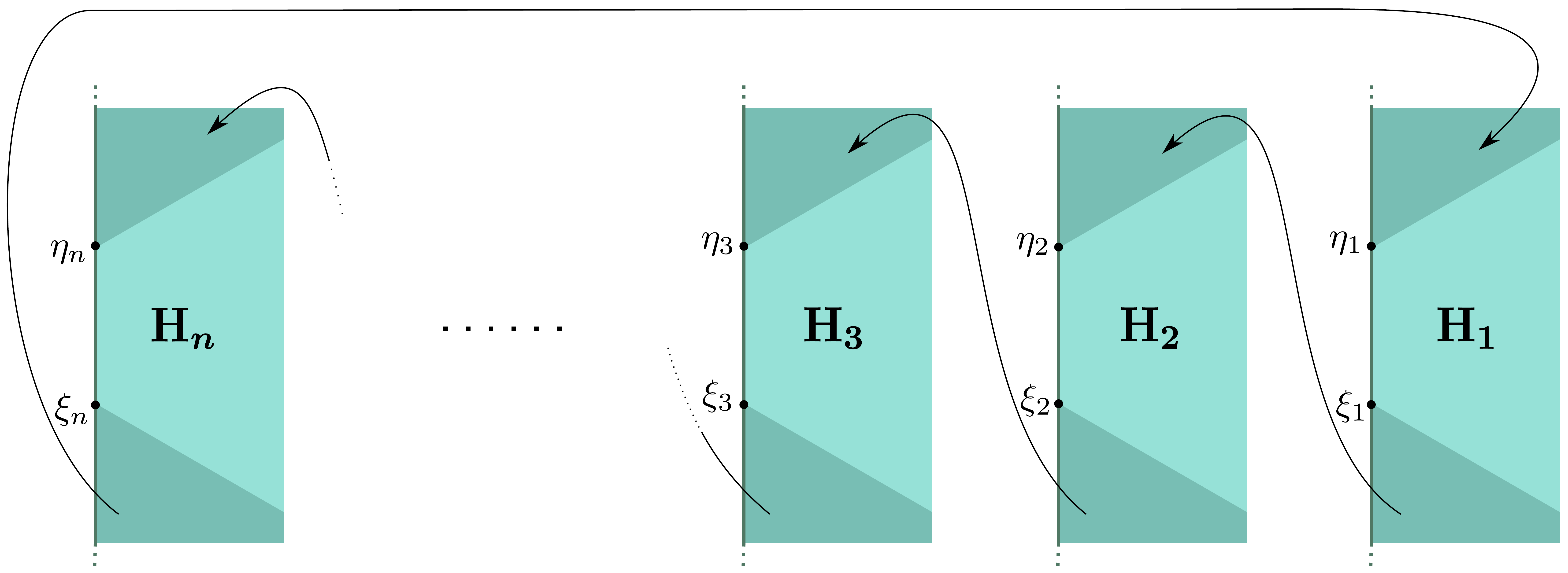}
\caption{A local model for $(\Sigma,\ve{b})$ by gluing $n$ copies of right half-planes: here $\mathbf{H}_{\k}$ is glued to $\mathbf{H}_{\k+1}$ by means of a map of the form $\zeta\mapsto e^{2\pi\ima/3}\zeta+\zeta_0$ which sends  $\xi_{\k}$ to $\eta_{\k+1}$. The cubic differential $2\,\dzeta^3$ is invariant under this map, hence defines a cubic differential on the glued surface $S$.}
\label{figure_glue}
\end{figure}

\subsection{A local model for $(\Sigma, \ve{b})$}\label{sec_local}
Proposition \ref{prop_natural} implies that we can build a local model of $(\Sigma,\ve{b})$ by patching together $n$ half-planes, each equipped with the  cubic differential $2\,\dzeta^3$. The following corollary gives a formal statement.

\begin{corollary}[\textbf{Local model}]
Let $\H_{\k}$ ($\k\in\mathbb{Z}/n\mathbb{Z}$) be  $n$ copies of the right half-plane. Then there exists  $\xi_\k, \eta_\k\in\pa\H_\k$ with $\im(\xi_\k)<0<\im(\eta_\k)$, such that a punctured neighborhood $U'\subset U$ endowed with the cubic differential $\ve{b}$ is isomorphic to the Riemann surface $S$ obtained by gluing the half-planes in the way described by Figure \ref{figure_glue}, endowed with the cubic differential $2\,\dzeta^3$.
\end{corollary}

Clearly, we can further identify the closure $\overline{U'}$ with the surface with boundary $\overline{S}$ obtained by gluing the closed half-planes $\CH_\k$ ($\k\in\mathbb{Z}/n\mathbb{Z}$) in the same way. In what follows, we view $\overline{S}$ and the $\CH_\k$'s as subsets of $\Sigma$.  This point of view is not honest in the $n=1$ case, where a quotient of $\CH$ rather than $\CH$ itself is subset of $\Sigma$. In fact, some arguments in \S \ref{sec_geq} below should indeed be adjusted when applied to this case.
However, the adjustment is just a matter of notations and thus we will not single out the $n=1$ case anymore.

\section{Solving Wang's equation}\label{sec_cv}

Let the Riemann surface $\Sigma$ and the cubic differential $\ve{b}$ be as at the beginning of Section \ref{sec_model}. In this section, we first prove statement (II) in the sketch of proof of Theorem \ref{intro_main} from the introduction. This is given by Theorem \ref{thm_vortex} and Proposition \ref{coro_non} below. We then establish some asymptotic estimates for $g$ around poles of order $\geq 3$ for later use.

\subsection{Existence and uniqueness}\label{sec_exanduq}
\begin{theorem}\label{thm_vortex}
$\,$ 
\begin{enumerate}
\item
There exists a unique complete conformal Riemannian metric $g$ on $\Sigma$ such that
\begin{itemize}
\item
$
\kappa_g=-1+2\|\ve{b}\|_g^2
$;
\item
in a punctured neighborhood of each pole of order $\geq3$ we have $a<\kappa_g<b$  for some constants $-1<a<b$;
\item\label{item_thmwang1}
$\kappa_g(x)\rightarrow-1$ as $x$ tends to a removable singularity or a pole of order $\leq 2$.
\end{itemize}
%\begin{enumerate}[(i)]
%\item\label{item_sov1}
%If $p$ is a pole of $\ve{b}$ of order $\geq3$, then $g$ is conformally quasi-isometric to the flat metric $|\ve{b}|^\frac{2}{3}$ around $p$.
%\item\label{item_sov2}
%If $p$ is a removable singularity or a pole of order $\leq2$, then $\kappa_g$ is negatively pinched around $p$.
%\end{enumerate} 
\item\label{item_thmwang2}
This metric $g$ furthermore satisfies 
\begin{itemize}
\item
$-1\leq\kappa_g\leq 0$;
\item
$\kappa_g(x)\rightarrow 0$ as $x$ tends to a pole of order $\geq 3$.
\end{itemize}
\end{enumerate}
\end{theorem} 

%Note that condition (\ref{item_sov2}) implies $g$ is conformally quasi-isometric to a cuspidal hyperbolic metric as shown in \S \ref{sec_cuspidal}.

\begin{proof}
(\ref{item_thmwang1})
Given a conformal Riemannian metric $g$ on $\Sigma$, we put 
$$
L(g)=1+\kappa_g-2\|\ve{b}\|_g^2.
$$
The standard method of super/sub-solutions consists in the following lemma.
\begin{lemma}\label{lemma_supsub}
Assume that $g_+$ and $g_-$  are continuous conformal Riemannian metrics on $\Sigma$ in the Sobolev class $H^1$ such that $g_-\leq g_+$. If
$$
%\begin{equation}\label{eqn_supsub}
L(g_+)\geq 0,\quad L(g_-)\leq 0
%\end{equation}
$$
in the sense of distributions, then there exists a smooth conformal Riemannian metric $g$ such that 
$$
L(g)=0,\quad g_-\leq g\leq g_+.
$$  
\end{lemma}
To prove the lemma, we first take a smooth complete  background conformal metric $g_0$ and write any conformal metric as $g=e^ug_0$. One checks that
$$
-2e^uL(e^ug_0)=\Delta_0 u-F(x,u),
$$
where $\Delta_0$ is the Beltrami-Laplacian of the metric $g_0$ and the function $F: \Sigma\times\mathbb{R}\rightarrow\mathbb{R}$ is defined by 
$$
F(x,u)=2\kappa_{g_0}(x)-4e^{-2u}\|\ve{b}\|^2_{g_0}(x)+2e^u.
$$
Lemma \ref{lemma_supsub} follows from the method of super/sub-solutions applied to the equation $\Delta_0u-F(x,u)=0$ (see Theorem 9 in \cite{wan})).

\vspace{5pt}

We now construct $g_-$ and $g_+$. First assume $\Sigma$ has negative Euler characteristic, so that $\Sigma$ carries a unique complete hyperbolic conformal metric of finite volume, denoted by $g_\hyp$. We construct the required $g_-$  as
$$
g_-=\max(g_\hyp, 2^\frac{1}{3}|\ve{b}|^\frac{2}{3}).
$$
Using the expression (\ref{eqn_cusphyp}) of $g_\hyp$ around a puncture, we see that $g_-=g_\hyp$ around removable singularities or poles of order $\leq 2$, whereas $g_-=2^\frac{1}{3}|\ve{b}|^\frac{2}{3}$ around poles of order $\geq 3$.

%This is obvious if $p$ is a pole of order $\geq 3$. If $p$ is a removable singularity or pole of order $\leq 2$, we have $\kappa_{g_0}=-1$, while $\|\ve{b}\|_{g_0}^2$ tends to $0$ at $p$ because of the expressions
%$$
%g_0=g_\hyp=\frac{4|\dz|^2}{|z|^2\big(\log|z|\big)^2}\ , \quad b(z)=z^rf(z)\dz^3,\  \ r\geq -2,
%$$
%where $z$ is distinguished local coordinate\footnote{If we parametrize a punctured neighborhood of $p$ by $\zeta\in U/2\pi\mathbb{Z}$, where $U$ is a horodisk centered at $\infty$ in the Poincar\'e upper-half plane, then $z=\exp(\ima\zeta)$.} centered at $p$ and $f(z)$ is a non-vanishing holomorphic function. 

We have
$
L(g_\hyp)=-2\|\ve{b}\|_{g_\hyp}^2\leq 0
$
and 
$
L(2^\frac{1}{3}|\ve{b}|^\frac{2}{3})=0 
$
(away from the zeros of $\ve{b}$). So the required inequality $L(g_-)\leq 0$ holds on the open set where
$g_\hyp\neq 2^\frac{1}{3}|\ve{b}|^\frac{2}{3}$. 
It also holds in the distribution sense at the points where $g_\hyp=2^\frac{1}{3}|\ve{b}|^\frac{2}{3}$ because $g_-$ has non-positive curvature in the distribution sense\footnote{A conformal metric $g=e^u|\dz|^2$ has non-positive curvature in the distribution sense if and only if $u$ is a subharmonic function, whereas the maximum of two subharmonic functions is again subharmonic.}. 

In order to construct $g_+$, we take a smooth conformal metric $g_0$ which coincides with $g_-$ around punctures and satisfies $g_0\geq g_-$. The curvature $\kappa_{g_0}$ and  the pointwise norm $\|\ve{b}\|_{g_0}$ are continuous and are bounded around punctures, hence they are bounded all over $\Sigma$. Therefore, we can take $g_+:=\lambda g_0\geq g_-$ for a sufficiently large constant $\lambda>1$ such that
$$
L(g_+)=1+\lambda^{-1}\kappa_{g_0}-2\lambda^{-3}\|\ve{b}\|_{g_0}^2\geq 0.
$$
This finishes the construction of $g_\pm$ when $\Sigma$ has negative Euler characteristic.

We now suppose that $\Sigma$ has non-negative Euler characteristic, \ie $\Sigma$ is either $\mathbb{C}^*$ or $\mathbb{C}$. By the Riemann-Roch theorem, the cubic canonical  bundle $K^3$ of the Riemann sphere $\mathbb{CP}^1$ has degree $-6$, so we are in one of the following cases.
\begin{enumerate}[(a)]
\item\label{item_nonneg1}
$\Sigma=\mathbb{C}$ and the puncture $\infty$ is a pole of order $\geq 6$.
\item\label{item_nonneg2}
$\Sigma=\mathbb{C}^*$ and exactly one of the two punctures $0$ and $\infty$ is a pole of order $\geq 3$. We assume without loss of generality that $\infty$ is such a pole.
\item\label{item_nonneg3}
$\Sigma=\mathbb{C}^*$ and both punctures are poles of order $\geq 3$.
\end{enumerate}

For each $R>1$, we define an open set $\Sigma_R\subset\Sigma$ in the above cases  respectively by 
$$
\Sigma_R:=
\begin{cases}
\{|z|< R\}&\mbox{in case (\ref{item_nonneg1})},\\
\{0< |z|< R\}&\mbox{in case (\ref{item_nonneg2})},\\
\{R^{-1}<|z|<R\}&\mbox{in case (\ref{item_nonneg3})}.
\end{cases}
$$

Then $\Sigma_R$ carries a unique complete conformal hyperbolic metric $g_R$ (of infinite volume), expressed  by
$$
g_r=
\begin{cases}
\frac{4R^2|\dz|^2}{(R^2-|z|^2)^2}&\mbox{in case (\ref{item_nonneg1})},\\[10pt]
\frac{|\dz|^2}{|z|^2\left(\log \frac{|z|}{R}\right)^2}&\mbox{in case (\ref{item_nonneg2})},\\[14pt]
\frac{(\tfrac{\pi}{2})^2|\dz|^2}{\left(\log R\right)^2|z|^2\cos^2\!\left(\tfrac{\pi}{2}\frac{\log |z|}{\log R}\right)}&\mbox{in case (\ref{item_nonneg3})}.
\end{cases}
$$
%$$
%(\ref{item_nonneg1}) \ \frac{4r^2|\dz|^2}{(r^2-|z|^2)^2}\,, \quad (\ref{item_nonneg2})\ \frac{|\dz|^2}{|z|^2\left(\log \frac{r}{|z|}\right)^2}\ , \quad (\ref{item_nonneg3})\ \frac{(\tfrac{\pi}{2})^2|\dz|^2}{\left(\log r\right)^2|z|^2\cos^2\!\left(\tfrac{\pi}{2}\frac{\log |z|}{\log r}\right)}\ .
%$$
Here, the last two expressions are obtained from the standard hyperbolic metric $\frac{|\dzeta|^2}{(\im \zeta)^2}$ on the upper-half plane $\{\im\zeta>0\}$ and the standard hyperbolic metric $\frac{|\dzeta|^2}{\cos^2(\im\zeta)}$ on the band $\{ -\tfrac{\pi}{2}<\im\zeta<\tfrac{\pi}{2}\}$ by scaling and exponential.

Using these expressions, we check that $g_R\leq 2^\frac{1}{3}|\ve{b}|^\frac{2}{3}$ on the boundary of the subset $\Sigma_{\!\sqrt{R}}\subset\Sigma_R$ provided that $R$ is big enough. Fixing such a $R$, we define the required metric $g_-$ by
$$
g_-=
\begin{cases}
\max(g_R, 2^\frac{1}{3}|\ve{b}|^\frac{2}{3})&\mbox{ on } \Sigma_{\sqrt{R}},\\
2^\frac{1}{3}|\ve{b}|^\frac{2}{3}& \mbox{ outside } \Sigma_{\sqrt{R}}.
\end{cases}
$$
It satisfies $L(g_-)\leq 0$ in the distribution sense for the same reason as in the negative Euler characteristic case. The constructions of $g_0$ and $g_+=\lambda g_0$ are also the same as in the that case.

Now that the required $g_-$ and $g_+$ are obtained, we let $g$ be a metric satisfying $L(g)=0$ and $g_-\leq g\leq g_+$ produced by Lemma \ref{lemma_supsub}. Note that the curvature 
\begin{equation}\label{eqn_kappag}
\kappa_g=-1+2\|\ve{b}\|_g^2=-1+\left(\frac{2^\frac{1}{3}|\ve{b}|^\frac{2}{3}}{g}\right)^3 
\end{equation}
has bounds of the form $-1<a<\kappa_g<b$ on a subset of $\Sigma$ if and only if $g$ and $|\ve{b}|^\frac{2}{3}$ are conformally quasi-isometric on that subset.  Around poles of order $\geq 3$, this holds because $g_-$ and $g_+$ are constant multiples of $|\ve{b}|^\frac{2}{3}$. Note that it implies $g$ is complete at these poles (see Definition \ref{def_complete}) because $|\ve{b}|^\frac{2}{3}$ is.  

Around a removable singularity or a pole of order $\leq 2$, since we have $2^\frac{1}{3}|\ve{b}|^\frac{2}{3}=|z|^{-\frac{2}{3}}f(z)|\dz|^2$ for some bounded function $f(z)$, whereas $g_\hyp=g_-=\lambda^{-1}g_+$ has the expression (\ref{eqn_cusphyp}), the proportion between $2^\frac{1}{3}|\ve{b}|^\frac{2}{3}$ and $g$ tends to $0$ as $z\rightarrow0$, hence $\kappa_g$ tends to $-1$. Moreover, $g$ is complete at these poles because $g_\hyp$ is. 

Thus we have shown that $g$ satisfies all the requirements and have established the existence part of the theorem.

To prove uniqueness, we let $g'=e^wg$ be another  metric fulfilling the requirements. 
Note that $g$ and $g'$ are conformally quasi-isometric because, on one hand, around poles of order $\geq3$, both $g$ and $g'$ are conformally quasi-isometric to $|\ve{b}|^\frac{2}{3}$ as explained above; on the other hand, around removable singularities or poles of order $\leq 2$, it follows from Lemma \ref{lemma_bilip} that $g$ and $g'$ are conformally quasi-isometric as well. Thus  $w$ is a bounded function.

A computation yields
$$
\lap _gw=2(e^w-1)+4\|\ve{b}\|_{g}^2(1-e^{-2w}),
$$
where $\lap _g$ is the Beltrami-Laplacian of $g$. At the points where $w\geq 0$, we have
$$
\lap _gw\geq 2(e^w-1)\geq k w
$$ 
for some constant $k>0$ only depending on $M:=\sup_\Sigma w$

Let us prove $M\leq 0$. Since $g$ is complete and has curvature bounded from below,  we can apply the Omori-Yau maximum principle \cite{yau_max} to $(\Sigma, g)$ and conclude that for any $\epsilon>0$ there exists $x_\epsilon\in \Sigma$ such that 
$$
w(x_\epsilon)\geq M-\epsilon\,,\quad \Delta{}_gw(x_\epsilon)\leq \epsilon.
$$
Assuming by contradiction that $M>0$ and taking $\epsilon\in(0,M)$, we get
$$
\epsilon\geq \lap _gw(x_\epsilon)\geq kw(x_\epsilon)\geq k(M-\epsilon).
$$
This implies $M\leq \epsilon(1+k^{-1})$ for any $\epsilon\in (0, M)$, which is impossible. Thus we have shown that $M\leq 0$, or equivalently, $g'\leq g$. Exchanging the roles of $g$ and $g'$, we get $g'\geq g$ as well. This establishes the uniqueness.

\vspace{5pt}

(\ref{item_thmwang2}) In view of the expression (\ref{eqn_kappag}) of $\kappa_g$, we have $\kappa_g\leq 0$ because $g\geq g_-\geq 2^\frac{1}{3}|\ve{b}|^\frac{2}{3}$ whereas $\kappa_g\geq -1$ because $\|\ve{b}\|^2\geq 0$. 

The fact

\end{proof}

\subsection{Coarse estimate at poles of order $\geq 3$}
In the rest of this section, we let $g$ be the metric produced by Theorem \ref{thm_vortex}. Note that the above proof of the theorem yields bounds of $g$ in terms of $g_+$ and $g_-$:
\begin{itemize}
\item we have $g\geq g_-\geq 2^\frac{1}{3}|\ve{b}|^\frac{2}{3}$ on the whole $\Sigma$. By Wang's equation, this is equivalent to $\kappa_g\leq 0$.
\item around poles of $\ve{b}$ of order $\geq 3$, we have $g\leq g_+=\lambda\cdot 2^\frac{1}{3}|\ve{b}|^\frac{2}{3}$ for a constant $\lambda>1$.
\end{itemize}

The goal of this and the next subsection is to establish a finer upper bound of $g$ near poles of order $\geq 3$, which will be used in Section \ref{sec_ctop} below. The estimates will be stated in terms of the function $u:\Sigma\rightarrow\mathbb{R}\cup\{+\infty\}$ defined in \S \ref{sec_wang} by
$$
u=\log\left(\frac{g}{2^\frac{1}{3}|\ve{b}|^\frac{2}{3}}\right)=-\frac{1}{3}\log(2\|\ve{b}\|_g^2).
$$
The above mentioned bounds of $g$ implies that $u\geq 0$ on the whole $\Sigma$ and that $u$ is bounded from above around poles of order $\geq 3$.  In this subsection, we establishes the following

\begin{proposition}[\textbf{Coarse estimate}]\label{coro_non}
$u(x)$ tends to $0$ when $x$ tends to a pole of order $\geq 3$.
\end{proposition}

\begin{proof}
Let $g_-$ and $g_0$ be the same as in the proof of Theorem \ref{thm_vortex}. We shall construct another upper metric $g_+'$ instead of the metric $g_+$ used in that proof so as to get the required asymptotics.

For each pole $p$ of order $\geq 3$, take a conformal local coordinate $z$ centered at $p$ such that $g_0$ coincides with $2^\frac{1}{3}|\ve{b}|^\frac{2}{3}$ in $U_p:=\{0<|z|< 1\}$. We can assume that the $U_p$'s are disjoint with each other and are away from other punctures. Let $c>0$ be a constant such that
\begin{equation}\label{eqn_g0}
g_0=2^\frac{1}{3}|\ve{b}|^\frac{2}{3}\geq c\,|z|^{-2}|\dz|^2
\end{equation}
in each $U_p$. Put $U^\frac{1}{2}_p:=\{0<|z|<\frac{1}{2}\}\subset U_p$.

Fix a family of monotonically increasing smooth functions $f_\alpha: \mathbb{R}_+\rightarrow [0,1]$
parametrized by $\alpha\in(0,1]$ with
$$
f_\alpha(t)=
\begin{cases}
|t|^\alpha&t\in(0,\frac{1}{2})\\
1&t\geq 1
\end{cases}
$$
such that both $f_\alpha'(t)$ and $f_\alpha''(t)$ have upper bounds which are independent of $\alpha$. 

We now give the construction of $g_+'$. Let $\phi:\Sigma\rightarrow\mathbb{R}_+$ be a smooth function such that
$\phi(z)=f_\alpha(|z|)$ in each  $U_p$ and $\phi=1$ outside the $U_p$'s, where $\alpha\in(0,1]$ is to be determined. We then define 
$$
g_+'=e^{\beta \phi} g_0,
$$ 
where $\beta>0$ is also to be determined. By definitions, we have 
$$
g_+'\geq g_0\geq g_-.
$$ 
It remains to be shown that $L(g_+')\geq 0$ if $\alpha$ is sufficiently small and $\beta$ is sufficiently large.

In each $U_p$ we have $\kappa_{g_0}=0$ and $2\|\ve{b}\|_{g_0}^2=1$, hence
$$
L(g_+')=1+\kappa_{g_+'}-2\|\ve{b}\|_{g_+'}^2=1+e^{-\beta\phi}\left(\kappa_{g_0}-\frac{\beta}{2}\Delta_{g_0}\phi\right)-2\|\ve{b}\|^2_{g_0} e^{-3\beta\phi}
$$
$$=
\begin{cases}
1-\frac{\beta}{2}e^{-\beta\phi}\Delta_{g_0}\phi-e^{-3\beta\phi}&\mbox{ in each }U_p,\\[5pt]
1+e^{-\beta}\kappa_{g_0}-2e^{-3\beta}\|\ve{b}\|^2_{g_0}&\mbox{ outside the $U_p$'s}.
\end{cases}
$$

By the construction of $g_0$, the curvature $\kappa_{g_0}$ and the pointwise norm $\|\ve{b}\|_{g_0}$ are bounded, hence we can take a big enough $\beta$ such that $L(g_+')\geq 0$ outside the $U_p$'s. Moreover, within $U_p\setminus U_p^\frac{1}{2}$, we have $\phi\geq \frac{1}{2}$ whereas $\Delta_{g_0}\phi$ also has a upper bound independent of $\alpha$, therefore for $\beta$ big enough we have $L(g_+')\geq 0$ in $U_p\setminus U_p^\frac{1}{2}$ as well. 

Now we can fix $\beta$ such that $L(g_+')\geq 0$ outside the $U_p^\frac{1}{2}$'s. We proceed to determine $\alpha$ such that $L(g_+')\geq 0$ on the whole $\Sigma$.

Let $\Delta=4\paz\pabz$ be the usual Laplacian with respect to $z$. Using the expression of Laplacian in polar coordinates, we get
$$
\Delta \phi(z)
=f''_\alpha(|z|)+\frac{1}{|z|}f'_\alpha(|z|)=\alpha^2|z|^{\alpha-2}
$$
whenever $z\in U_p^\frac{1}{2}$. 
In view of (\ref{eqn_g0}), it yields
$$
\Delta_{g_0}\phi(z)\leq \frac{|z|^2}{c}\Delta\phi(z)=\frac{\alpha^2}{c}|z|^\alpha=\frac{\alpha^2}{c}\phi(z)
$$
for $z\in U_p^\frac{1}{2}$. Substituting this inequality into the above expression of $L(g_+')$   provides
$$
L(g_+')\geq 1-\frac{\alpha^2}{2c}\beta\phi\, e^{-\beta\phi}-e^{-3\beta\phi}, \ z\in U_p^\frac{1}{2}.
$$
We take $\alpha\in(0,1]$ small enough such that $1-\frac{\alpha^2}{c}te^{-t}-e^{-3t}\geq 0$ for any $t\geq 0$, so that $L(g_+')\geq 0$ on $U_p^\frac{1}{2}$, hence on the whole $\Sigma$.

Starting from $g_-$ and $g_+'$, Lemma \ref{lemma_supsub} yields a metric $g'$ satisfying Wang's equation and the inequalities $g_-\leq g'\leq g_+'$. But these inequalities imply that $g'$ satisfies the conditions 
%....(\ref{item_sov1}) and (\ref{item_sov2}) 
in Theorem \ref{thm_vortex} as well, and the uniqueness part of the theorem says that $g'$ is nothing but $g$. Therefore, near each pole $p$ of order $\geq 3$ we have
$$
0\leq u=\log\left(\frac{g}{2^\frac{1}{3}|\ve{b}|^\frac{2}{3}}\right)\leq \log\left(\frac{g_+'}{g_0}\right)=\beta\phi=\beta |z|^\alpha,
$$
hence $\lim_{z\rightarrow0} u(z)=0$.
\end{proof}

\subsection{Strong estimate at poles of order $\geq 3$}\label{sec_fine}
We now refine the work in the previous subsection to give much more delicate upper bounds of $u$ around poles of order $\geq 3$. By Proposition \ref{coro_non}, we can choose the punctured neighborhood $U$ at the beginning of Section \ref{sec_model} to be small enough such that $u\leq \frac{1}{2}$ on $U$.

First consider a pole $p$ of order $n+3\geq 4$. Take natural half-planes $\H_\k\subset U$, $\k\in\mathbb{Z}/n\mathbb{Z}$ as in \S \ref{sec_local}.  On each $\H_\k$ we can write
$$
g=2e^u|\dzeta|^2
$$
for the natural coordinate $\zeta$ on $\H_\k$ because $\ve{b}=2\,\dif\zeta^3$.

The next theorem, due to Dumas and Wolf, basically says that $u$ and its gradient decays exponentially with respect to the $|\ve{b}|^\frac{2}{3}$-distance from $\zeta$ to a fixed point. This is a far-reaching refinement of  Corollary \ref{coro_non}. 
%Intriguingly, the factor $|\zeta|^{-\frac{1}{2}}$ in the bound plays an essential role in applications (Theorem \ref{thm_compa} (\ref{item_ocompa3})).

\begin{theorem}[\textbf{Strong estimate at poles of order $\geq 4$}]\label{thm_fine}

There is a constant $C$ such that for each $\ve{k}\in\mathbb{Z}/n\mathbb{Z}$ and $\zeta\in\H_\k$, we have
$$
u(\zeta), |\pa_\zeta u(\zeta)|\leq C\,|\zeta|^{-\frac{1}{2}}e^{-2\sqrt{3}|\zeta|}.
$$
\end{theorem}

\begin{proof}
In each $\H_\k$, Wang's equation for $g$ becomes $\lap u=4e^u-4e^{-2u}$, those we can apply results from the appendix.

Let  $X=\overline{\H''}\cup\CH\cup\overline{\H'}$ be the surface studied in Lemma \ref{lemma_a2} from the appendix. There is an obvious map $X\rightarrow \Sigma$ identifying $\CH''$, $\CH$ and $\CH'$ with $\CH_{\k-1}$, $\CH_\k$ and $\CH_{\k+1}$, respectively. The function $u$ pulls back through this map to a function on $X$ satisfying the hypotheses of Lemma \ref{lemma_a2}, hence the required estimates follows from Lemma \ref{lemma_a2} and Corollary \ref{coro_a}. 
\end{proof}

Let us proceed to the parallel result for $p$ a third order pole. As explained at the beginning of Section \ref{sec_model}, we take a conformal local coordinate $z$ around $p$ such that $\ve{b}=Rz^{-3}\dz^3$. Still let $U=\{0<|z|<a\}$ be a small enough punctured neighborhood where $u\leq \frac{1}{2}$.

\begin{theorem}[\textbf{Strong estimate at third order poles}]\label{thm_fine}
There is a constant $C$ such that for each 
$$
0\leq u(\zeta), |\pa_\zeta u(\zeta)|\leq C\,e^{-2\sqrt{3}|\zeta|}.
$$
\end{theorem}

\section{From $\V_0(\Sigma)$ to $\P_0(\Sigma)$}\label{sec_ctop}
Still let the Riemann surface $\Sigma$ and the cubic differential $\ve{b}$ be as in the previous two sections. Moreover, let $g$ be the conformal metric produced by Theorem \ref{thm_vortex}, so that $(g,\ve{b})\in\V_0(\Sigma)$.

The goal of this section is to prove statement (III) in the sketch of proof of Theorem \ref{intro_main} from the introduction. Therefore we fix a base point $m\in\Sigma$ in order to  define Wang's developing pair $(\dev, \hol)$ associated to $(g,\ve{b})$ (\cf \S \ref{sec_wang}). Also fix a puncture $p$ and a point $\tilde{p}$ in the Farey set $\Fa{\Sigma, p}$ so as to define the developed boundary $\devbd{\tilde{p}}$ (\cf \S \ref{sec_farey}).

\subsection{Poles of order $\geq 4$}\label{sec_geq}

We first establish statement (III) when $p$ is a pole of order $n+3\geq 4$. Theorem \ref{thm_finding} below is a precise statement, where the required twisted polygon is found as developing limits of some particular paths on $\Sigma$, defined as follows.

\subsubsection{Special collections of paths converging to $p$}\label{sec_winding}

\begin{definition}\label{def_U}
Let $\C$ denote the set of oriented smooth paths $\beta$ on $\Sigma$ issuing from $m$, converging to $p$ and fulfilling the following requirements.
\begin{enumerate}[(i)]
\item\label{item_U1}
$\beta$ belongs to the class $\tilde{p}$ (see \S \ref{sec_farey} for the definition).
\item\label{item_U2}
$\beta$ is eventually 
\footnote{
An oriented path $\beta$ \emph{eventually} satisfying a condition means that a final portion of $\beta$, \ie the part of $\beta$ in some neighborhood of $p$, satisfies that condition.
}
a geodesic with respect to the flat metric $|\ve{b}|^\frac{2}{3}$.
\item\label{item_U3}
$\beta$ is asymptotic to some ray $\dr(\beta)=\mathbb{R}_+\cdot v\subset\T_p\overline{\Sigma}$ (\cf Proposition \ref{prop_natural}  for the notion of ``asymptotic ray''). 
\end{enumerate}
Two  paths $\beta, \beta'\in\C$ are said to be \emph{equivalent} if they satisfy
\begin{enumerate}[(a)]
\item\label{item_winding1}
$\beta$ eventually coincides $\beta'$;
 \item\label{item_winding2}
 $\beta$ can be homotoped to $\beta'$ by a homotopy which fixes the starting point $m$ and a common final portion of $\beta$ and $\beta'$. 
\end{enumerate}
\end{definition}
%\begin{remark}
%Using the local model of $(\Sigma,\ve{b})$ around poles of order $\geq 4$ given in \S \ref{sec_local}, one sees that every $|\ve{b}|^\frac{2}{3}$-geodesic converging to $p$ admits an asymptotic ray, hence condition (\ref{item_U3}) is actually dispensable. Nevertheless, it is indispensable for poles of order $3$.
%\end{remark}

\vspace{5pt}

Paths $\beta, \beta\in\C'$ merely satisfying condition (\ref{item_winding1}) are not necessarily equivalent basically because they can wrap around $p$ differently many times before becoming the same geodesic. An equivalence class can be pinned down by taking the ``winding number'' into account.

More precisely, we define the \emph{winding number} $\vartheta(\beta)\in\mathbb{R}$ of $\beta$ with respect to a reference path $\beta_0\in \C$ as follows. Since both $\beta_0$ and $\beta$ belong to the class $\tilde{p}$, we can deform  $\beta_0$ into $\beta$ through a continuous family of paths $\beta_s\in\C$ (where $s\in[0,1]$, $\beta_1=\beta$). Roughly speaking, $\vartheta(\beta)$ is the angle swept out by the asymptotic ray $\dr(\beta_s)$ as $s$ goes from $0$ to $1$. See Figure \ref{figure_winding} for some examples. In what follows, we always choose $\beta_0$ such that $\dr(\beta_0)=\mathbb{R}_+\subset\T_p\overline{\Sigma}\cong\mathbb{C}$ as in the figure.
\begin{figure}[h]
\centering\includegraphics[width=1.9in]{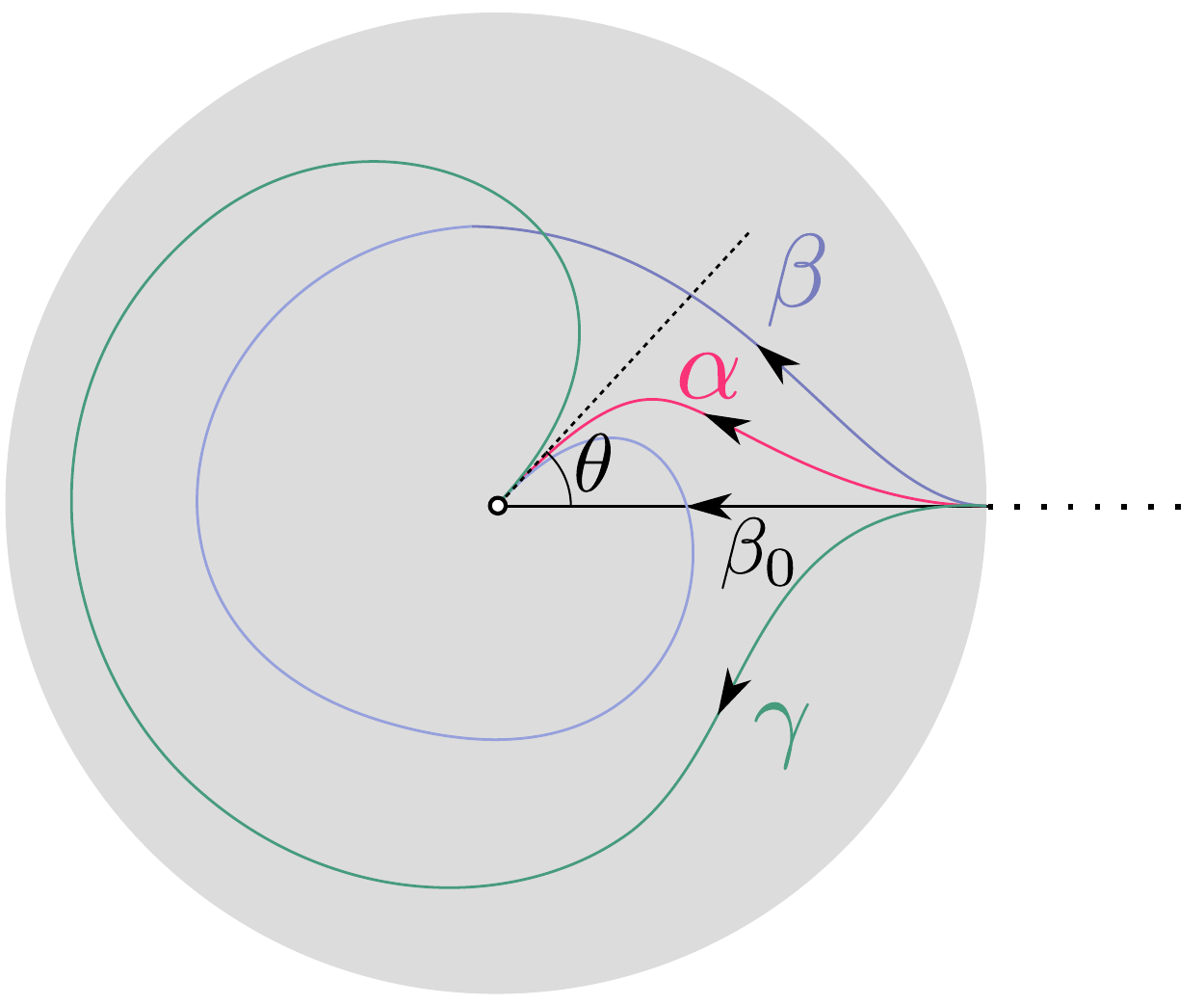}
\caption{Examples of winding numbers: with respect to $\beta_0$, we have $\vartheta(\alpha)=\theta$, $\vartheta(\beta)=\theta+2\pi$ and $\vartheta(\gamma)=\theta-2\pi$.
}
\label{figure_winding}
\end{figure}

In this paper, we content ourselves with the above intuitive definition of $\vartheta(\beta)$ and will make use of the following intuitively obvious facts.

\begin{enumerate}
\item\label{item_winding1}
$\vartheta(\beta)\equiv\arg(\dr(\beta))\!\mod2\pi$. 
\item\label{item_winding2}
Two paths $\beta, \beta'\in\C$ are equivalent if and only if $\beta$ eventually coincides with $\beta'$
and  $\vartheta(\beta)=\vartheta(\beta')$.
\item\label{item_winding3}
$\vartheta(\beta\cdot\gamma_{\tilde{p}})=\vartheta(\beta)-2\pi$ (see \S \ref{sec_farey} for the definition of $\gamma_{\tilde{p}}$).
\end{enumerate}

%\begin{remark}
%We will build below a local model of $(\Sigma, \ve{b})$ around $p$ as a \emph{$e^{2\pi\ima/3}$-translation surface} obtained by gluing a number of copies of the right half-plane $\H$ endowed with a constant cubic differential. Using this model, one easily proves the following facts, which will not be used in this paper.
%
%\begin{itemize}
%\item
%Condition (\ref{item_U3}) in Definition \ref{def_U} is dispensable because it is a consequence of condition (\ref{item_U2}). This is not true for poles of order $3$, see \S \ref{sec_leq} below.
%\item
%Given $\theta_0\in\mathbb{R}$, all equivalence classes of paths $\beta\in\C$ with $\vartheta(\beta)=\theta_0$ form a one-parameter family. Indeed, in the above mentioned model, each $\beta$ in this family is eventually a ray in $\H$ of the form 
%$$
%t\mapsto e^{\theta_0\ima}(t+s\ima),
%$$
%where $\theta_0\in (-\tfrac{\pi}{2},\tfrac{\pi}{2})$ is a constant depending on $\theta_0$ and $s$ is the parameter.
%\end{itemize}
%\end{remark}

We shall further introduce some particular subsets of $\C$.
\begin{definition}
For each $k\in\mathbb{Z}$, define
\begin{align*}
\C_k:=\big\{\beta\in\C\ \big|\ \tfrac{2(k-1)\pi}{n}<\vartheta(\beta)<\tfrac{2 k\pi}{n}\big\}, \quad
\C_{k,k+1}:=\big\{\beta\in\C\ \big|\ \vartheta(\beta)=\tfrac{2\pi k}{n}\},
\\[6pt]
\U_k^-:=\big\{\beta\in\C\ \big|\ \vartheta(\beta)=\tfrac{(4k-3)\pi}{2n}\big\},\quad \U_k^+:=\big\{\beta\in\C\ \big|\ \vartheta(\beta)=\tfrac{(4k-1)\pi}{2n}\},
\end{align*}
\begin{align*}
\Se_k=\big\{\beta\in\C\ \big|\  \tfrac{(4k-3)\pi}{2n}<\vartheta(\beta)< \tfrac{(4k-1)\pi}{2n}\big\},
\\[6pt]
\Se_{k,k+1}=\big\{\beta\in\C\ \big|\ \tfrac{(4k-1)\pi}{2n}<\vartheta(\beta)< \tfrac{(4k+1)\pi}{2n}\big\}.
\end{align*}
A path $\beta\in \C$ is said to be \emph{unstable} if it belongs to $\U_k^-$ or $\U_k^+$ for some $k$, otherwise it is said to be \emph{stable}.
\end{definition}
These definitions clearly stem from the special directions and sectors introduced in \S \ref{sec_special}. For example, given $\k\in\mathbb{Z}/n\mathbb{Z}$, the union $\bigcup_{k\in\k}\C_k$ is the set of all $\beta\in\C$ such that the asymptotic ray $\dr(\beta)$ is in $C_\k$; a unstable path is just path $\beta$ such that $\dr(\beta)$ is a unstable direction, \etc.

We proceed to single out a representatives in each equivalence class of paths $\beta\in\C$ in order to apply the local model built in \S \ref{sec_local} to the present study. First let us introduce some more notations, illustrated by the picture below.

\pt Concatenating the segment of $\pa\H_\k$ from $0$ to $\xi_\k$ and the segment of  $\pa\H_{\k+1}$ from $\eta_{\k+1}$ to $0$ (see \S \ref{sec_local} for the notations) yields a path going from $0\in \CH_\k$ to $0\in\CH_{\k+1}$, denoted by $\gamma_{\k,\k+1}$. 

\pt
Let $(\alpha_k)_{k\in\mathbb{Z}}$ be a family of paths such that $\alpha_k$ goes from $m$ to $0\in\CH_\k$ and $\alpha_{k+1}$ is homotopic to $\gamma_{k,k+1}\cdot \alpha_k$ for any $k$. 

\pt Let $\beta_0\in \C$ be the concatenation of $\alpha_0$ with the ray $e^{-\pi \ima/3}\mathbb{R}_{\geq 0}$ in $\H_{0}$. From on now, we take this $\beta_0$ as the reference path when considering winding numbers.
\begin{figure}[h]
\centering\includegraphics[width=2.9in]{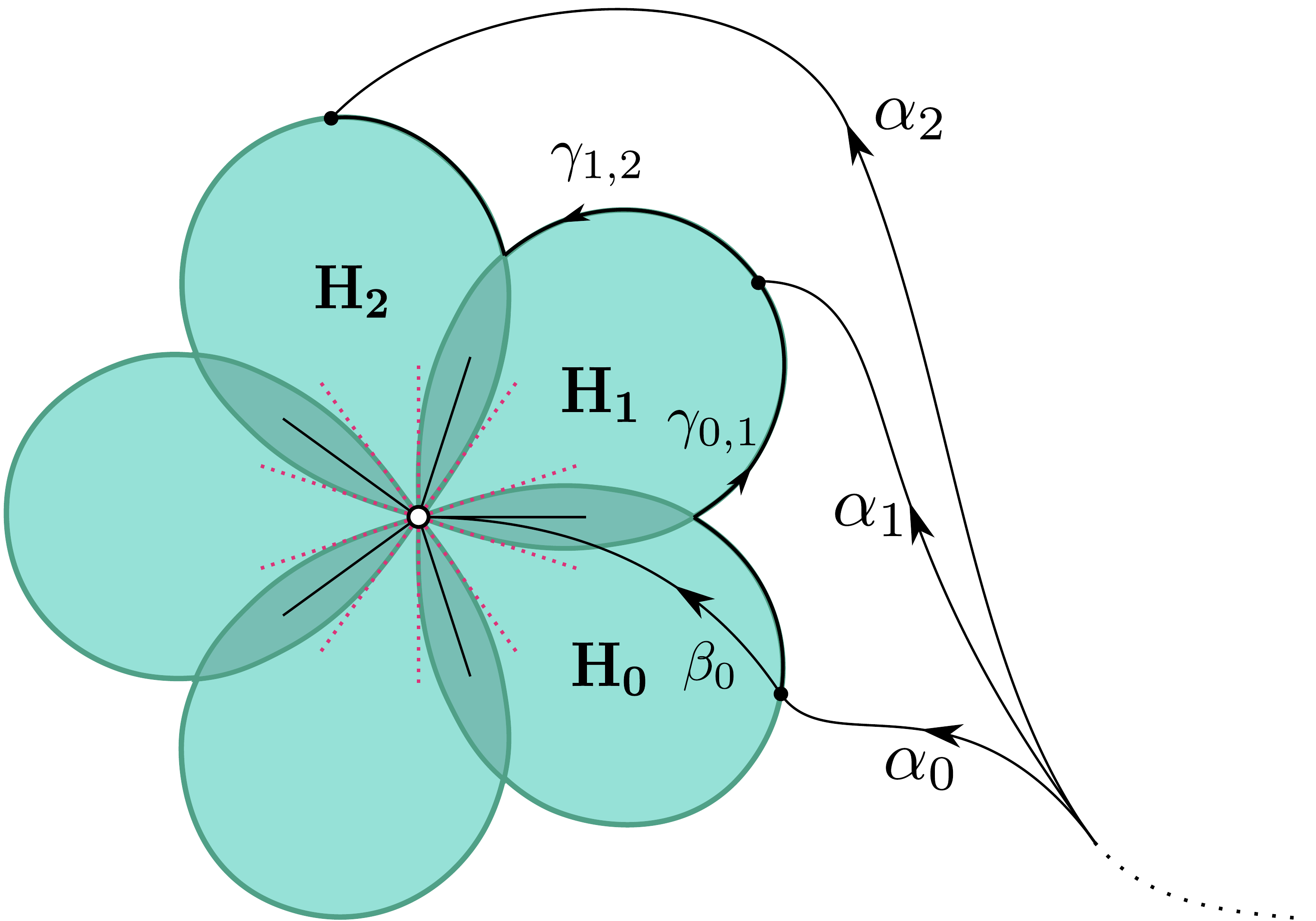}
\end{figure}

Using property (\ref{item_natural3}) in Proposition \ref{prop_natural} and the table following Proposition \ref{prop_natural}, we get the following lemma.
\begin{lemma}[\textbf{Reducing paths to a half-plane}]\label{lemma_alpha}
A path $\beta\in \C$ belongs to $\Se_{k-1,k}\cup\Se_k\cup\Se_{k,k+1}$ 
%$$
%\tfrac{2\pi k}{n}-\tfrac{\pi}{2n}<\vartheta(\beta)< \tfrac{2\pi(k+1)}{n}+\tfrac{\pi}{2n}
%$$ 
if and only if $\beta$ is equivalent to $\tilde{\beta}\cdot\alpha_k$ for some path $\tilde{\beta}$ in $\CH_\k$ such that $\tilde{\beta}$ is eventually a ray with asymptotic argument
$$
\theta(\tilde{\beta}):=\lim_\tinf \arg(\tilde{\beta}(t))\in (-\tfrac{\pi}{2},\tfrac{\pi}{2}).
$$
Moreover, $\beta$ belongs to each of the sets in the first row of the following table if and only if $\theta(\tilde{\beta})$ belongs to the interval or takes the specific value indicated by the second row.

\vspace{5pt}

\begin{tabular}{|c||c|c|c||c|c|c|c|c|}
\hline
$\beta$&$\C_{k,k+1}$&$\C_k$&$\C_{k-1,k}$&
$\Se_{k,k+1}$&$\U_k^+$&$\Se_{k}$&$\U_k^-$&$\Se_{k-1, k}$
\\[3pt]\hline
$\theta(\tilde{\beta})$&$-\tfrac{\pi}{3}$&$(-\tfrac{\pi}{3},\tfrac{\pi}{3})$&$\tfrac{\pi}{3}$&
$(-\tfrac{\pi}{2},\tfrac{\pi}{6})$&$-\tfrac{\pi}{6}$&$(-\tfrac{\pi}{6},\tfrac{\pi}{6})$&$\tfrac{\pi}{6}$&$(\tfrac{\pi}{6},\tfrac{\pi}{2})$\\[3pt] \hline
\end{tabular}
\end{lemma}
\vspace{5pt}

\subsubsection{Statement of the main result}
Given $\beta\in\C$, we always let 
$$ 
[0,+\infty)\rightarrow \Sigma,\quad t\mapsto \beta(t)
$$
denote the parametrization of $\beta$ by arc-length with respect to the metric $|\ve{b}|^\frac{2}{3}$. Recall from \S \ref{sec_devlim} that $\beta_{[0,t]}:=\beta([0,t])$ denotes the truncation of length $t$ and $\dev(\beta_{[0,t]}^{-1})\in\mathbb{P}(E_m)$ denotes the developing image of the point in $\widetilde{\Sigma}_m$ represented by $\beta_{[0,t]}^{-1}$. The developing limit of $\beta$ is $\devlim(\beta):=\lim_\tinf\dev(\beta_{[0,t]}^{-1})\in\mathbb{P}(E_m)$ if the latter limit exists.

\begin{theorem}[\textbf{developed boundary at higher order poles}]\label{thm_devhigher}
The developing limit $\devlim(\beta)$ exists for any stable $\beta\in\C$ and admits the following descriptions.
\begin{enumerate}
\item\label{item_finding1}
There exists $X_k\in\mathbb{P}(E_m)$  such that  
$\devlim(\beta)=X_k$ for any $\beta\in\C_k$.

\item\label{item_finding2}
As $\beta$ runs over $\C_{k,k+1}$, the set of developing limits
$$
X_{k,k+1}^\circ:=\{\devlim(\beta)\}_{\beta\in\C_{k,k+1}}
$$
is the interior of a segment $X_{k,k+1}\subset\mathbb{P}(E_m)$ with endpoints $X_{k}$ and $X_{k+1}$. The adjacent segments $X_{k-1,k}$ and $X_{k, k+1}$ are not collinear.
\item\label{item_finding3}
The holonomy $\hol_{\tilde{p}}$ maps $X_k$ to $X_{k+n}$, thus $\bigcup_{k\in\mathbb{Z}}X_{k,k+1}$ is a $n$-gon twisted by $\hol_{\tilde{p}}$. This twisted $n$-gon and its accumulation points constitue $\devbd{\tilde{p}}$.
\end{enumerate}
\label{thm_finding}
\end{theorem}

The above theorem is a generalisation of Theorem 6.3 in \cite{dumas-wolf} by Dumas and Wolf. Our proof is an adaptation of theirs as well.

\subsubsection{Auxiliary developing limits}\label{sec_auxpoly}
We let $\dev_0: \widetilde{\Sigma}_m\rightarrow\mathbb{P}(E_m)$ be Wang's developing map associated to $(2^\frac{1}{3}|\ve{b}|^\frac{2}{3},\ve{b})$ (see \S \ref{sec_wangass}). Let $\devlima(\beta)$ denote the corresponding developing limit for $\beta\in\C$, \ie 
$$
\devlima(\beta):=\lim_\tinf\dev_0(\beta_{[0,t]}^{-1})
$$
if the limit exists.

The pair $(2^\frac{1}{3}|\ve{b}|^\frac{2}{3},\ve{b})$ is expressed as $(2|\dzeta|^2, 2\,\dzeta^3)$ on each half-plane $\H_\k$ in the local model. But Wang's developing map associated to the latter pair is well known as we have seen in \S \ref{sec_titeica}. Therefore, in order to determine all the $\devlima(\beta)$'s,  one only need to patch together the developing limits given by Proposition \ref{prop_tit} for each $\H_\k$.  We  introduce some notations in order to state the result.

Let $(e_1^\k,e_2^\k,e_3^\k)$ be the equilateral frame (see \S \ref{sec_wang} for the definition) of $E$ over $\CH_\k$ associated to the natural coordinate. 
Let $E_{0,\k}$ be the fiber of $E$ at $0\in \CH_\k$ and let $\Delta^\k\subset\mathbb{P}(E_{0,\k})$ be the triangle 
$$
\Delta^\k=\big\{\big[xe_1^\k(0)+ye_2^\k(0)+ze_3^\k(0)\big]\mid xyz\geq 0\big\}.
$$
Here $[v]$ stands for the projectivization of the vector $v$. As in \S \ref{sec_titeica}, we put
$$
X_i^\k=[e_i^\k(0)],\quad X_{ij}^\k=\big\{\big[(1-s)\,e_i^\k(0)+s\,e_j^\k(0)\big]\big\}_{s\in[0,1]}
$$
We emphasis that the ``$0$'' in the notation ``$e^\k_i(0)$'' means the origin of $\CH_\k$ and represents different points for different $\k$.

It turns out that for any $k\in\mathbb{Z}$,  the triangle $\Delta^\k$ gets mapped by the parallel transport $\a_0(\alpha_k^{-1}):\mathbb{P}(E_{0,\k})\rightarrow\mathbb{P}(E_m)$ to the same triangle in $\mathbb{P}(E_m)$. However, the map permutes the vertices according to $k$. To describe this precisely, we put
$$
\hat{X}_1:=\a_0(\alpha_1^{-1})X^1_1,\quad \hat{X}_2:=\a_0(\alpha_1^{-1})X^1_3,\quad\hat{X}_3:=\a_0(\alpha_1^{-1})X^1_2,
$$
$$
\hat{X}_{12}:=\a_0(\alpha_1^{-1})X^1_{13},\quad \hat{X}_{23}:=\a_0(\alpha_1^{-1})X^1_{23},\quad \hat{X}_{31}:=\a_0(\alpha_1^{-1})X^1_{12}.
$$
That is, the $\hat{X}_i$'s and $\hat{X}_{ij}$'s are the vertices and edges of the triangle $\a_0(\alpha_1^{-1})(\Delta^1)$ in $\mathbb{P}(E_m)$. Note that they are labelled differently from those of $\Delta^1$.

\begin{lemma}\label{lemma_mod3}
For any $k\in\mathbb{Z}$, let $[k]\in\{1,2,3\}$ denote the mod $3$ value of $k$. Then
\begin{align*}
\a_0(\alpha_k^{-1})X_1^\k=\hat{X}_{[k]},\quad
\a_0(\alpha_k^{-1})X_2^\k=\hat{X}_{[k-1]},\quad
\a_0(\alpha_k^{-1})X_3^\k=\hat{X}_{[k+1]},
\end{align*}
$$
\a_0(\alpha_k^{-1})X_{12}^\k=\hat{X}_{[k-1],[k]},\quad
\a_0(\alpha_k^{-1})X_{23}^\k=\hat{X}_{[k+1],[k-1]},\quad
\a_0(\alpha_k^{-1})X_{31}^\k=\hat{X}_{[k],[k+1]}.
$$
\end{lemma}

\begin{proof}
When $k=1$, these are exactly the definitions. In order to prove for other $k$'s, we claim that 
\begin{align*}
\a_0(\alpha_{k}^{-1})[e_{i}^{\k}(0)]=\a_0(\alpha_{k-1}^{-1})[e_{i-1}^{\k-1}(0)]
\end{align*}
for any $k$ and $i$. Here the indices $i, i-1\in\{1,2,3\}$ are counted mod $3$. 
The required equalities then follow from the claim and the $k=1$ case by induction.

To prove the claimed, we rewrite it as
\begin{equation}\label{eqn_auxpoly3}
[e^\k_i(0)]=\a_0(\gamma_{k-1,k})[e^{\k-1}_{i-1}(0)].
\end{equation}
Here $\gamma_{k-1,k}\cong\alpha_k\cdot\alpha_{k-1}^{-1}$ is defined in \S \ref{sec_winding} as the concatenation of a segment $\gamma'\subset\pa\H_{\k-1}$ followed by a segment $\gamma''\subset \pa\H_\k$. Now (\ref{eqn_auxpoly3}) is a  consequence of the following two facts: first, the transition map from $\CH_{\k-1}$ to $\CH_\k$ is $\zeta\mapsto e^{2\pi\ima/3}\zeta+\zeta_0$, so a straightforward computation with the definition of equilateral frames shows that $e^\k_i=e^{\k-1}_{i-1}$ on $\CH_{\k-1}\cap\CH_\k$;
second, as we have seen in \S \ref{sec_titeica}, the matrix expressions of $\a_0$ are diagonal with respect to the equilateral frame, thus $\a_0(\gamma')$ (resp. $\a_0(\gamma'')$) preserves the sections $[e^{\k-1}_i]$ (resp. $[e^\k_i]$) of the projectivized bundle $\mathbb{P}(E)$. 
\end{proof}

\begin{proposition}[\textbf{Auxiliary developing limits}]\label{prop_auxpoly}
The developing limit $\devlima(\beta)$ exists for any $\beta\in\C$. Moreover,
 \begin{enumerate}
\item\label{item_aux1}
$\devlima(\beta)=\hat{X}_{[k]}$ for any $\beta\in\C_{k}$;
\item\label{item_aux2}
$\{\devlima(\beta)\}_{\beta\in\C_{k,k+1}}=\hat{X}^\circ_{[k],[k+1]}$.
\end{enumerate}
\end{proposition}
\begin{proof}
By Lemma \ref{lemma_alpha}, $\beta\in \C_k$ if and only if $\beta=\tilde{\beta}\cdot\alpha_k$, where the path $\tilde{\beta}\subset\CH_\k$ is eventually a ray with asymptotic argument $\theta\in(-\tfrac{\pi}{3},\tfrac{\pi}{3})$. The developing limit $\devlima(\tilde{\beta})\in\mathbb{P}(E_{0,\k})$, defined using the developing map with respect to the base point $0\in\CH_\k$, is related to $\devlima(\beta)$ by 
$$
\devlima(\beta)=\a_0(\alpha_k^{-1})\devlima(\tilde{\beta})
$$
Proposition \ref{prop_tit} gives $\devlima(\tilde{\beta})=X_1^\k$. Using Lemma \ref{lemma_mod3}, we get
\begin{equation}\label{eqn_auxpoly1}
\devlima(\beta)=\a_0(\alpha_{k}^{-1})X_1^\k=\hat{X}_{[k]}
\end{equation}
for any $\beta\in \C_k$, as required. This is proofs statement (\ref{item_aux1}).

Similarly, if $\beta\in \C_{k,k+1}$ then $\theta=-\tfrac{\pi}{3}$ and Proposition \ref{prop_tit} gives $\left\{\devlima(\tilde{\beta})\right\}_{\beta\in \C_{k,k+1}}\!\!=X^\k_{31}$. Again, we use Lemma \ref{lemma_mod3} to get statement (\ref{item_aux2}).
\end{proof}

\subsubsection{Comparing parallel transports}\label{sec_compa}

Now that the $\devlima(\beta)$'s are known, in order to find $\devlim(\beta)$, we compare the parallel transports $\a_0$ and $\a$ which generate $\dev_0$ and $\dev$, respectively.

\begin{definition}\label{def_compa}
${}$
\begin{enumerate}
\item
Given $\beta\in\C$, we define the \emph{comparison transformation} 
%(between the parallel transports of $\D$ and $\D_0$) 
along the truncated path $\beta_{[0,t]}$ as the projective transformation
\begin{equation}\label{eqn_pbetadef}		
P_\beta(t):=\para(\beta_{[0,t]}^{-1})\a_0(\beta_{[0,t]})\in \SL(E_m).
\end{equation}
Denote $P_\beta:=\lim_\tinf P_\beta(t)\in\SL(E_m)$ if the limit exists.

\item
For each $k$, we define one-parameter unipotent subgroups $G_k^\pm\subset\SL(E_m)$ as follows. Endow $E_m$ with a basis whose projection to $\mathbb{P}(E_m)$ is the triple of points $(\hat{X}_{[k]},\hat{X}_{[k-1]}, \hat{X}_{[k+1]})$ (see \S \ref{sec_auxpoly} for the notation). Writing elements in  $\SL(E_m)$ as matrices under this basis, we define
$$
G_k^-
:=
\left\{
\begin{pmatrix}
1&&*\\
&1&\\
&&1
\end{pmatrix}
\right\}
,\quad
G_k^+
:=
\left\{
\begin{pmatrix}
1&*&\\
&1&\\
&&1
\end{pmatrix}
\right\}.
$$
Let $\Pi^\pm_k$ denote the projection $\SL(E_m)\rightarrow \SL(E_m)/G_k^\pm$. 
\end{enumerate}
\end{definition}

%One can reformulate Lemma 6.4 and Lemma 6.5 in \cite{dumas-wolf} as saying that $P_\beta$ is well-defined and is invariant under deformations for stable $\beta$'s, whereas the transition of $P_\beta$ between $\Se_k$ and $\Se_{k,k+1}$ (resp. between $\Se_{k,k+1}$ and $\Se_{k+1}$) is given by an element in $G_k^-$ (resp. $G_k^+$). We give below a formal statement and an almost self-contained proof, adapted from the original one, for the sake of completeness.

The following theorem is a reformulation of Lemma 6.4 and Lemma 6.5 in \cite{dumas-wolf}. We include the proof here for the sake of completeness.
\begin{theorem}[\textbf{Limit of comparison transformation}]\label{thm_compa}
${}$
\begin{enumerate}
\item\label{item_ocompa1}
The limit in (\ref{eqn_pbetadef}) exists if $\beta$ is stable.
\item\label{item_ocompa2}
Given $k$, any $\beta\in\Se_{k}$ (resp. $\beta\in\Se_{k,k+1}$) gives rise to the same $P_\beta$, which we denote by $P_k$ (resp. $P_{k,k+1}$).
\item\label{item_ocompa3}
$P_{k-1,k}^{-1}P_{k}$ and $P_{k,k+1}^{-1}P_{k}$ belongs to $G_k^-$ and $G^+_k$, respectively.
\end{enumerate}
\end{theorem}

%Note that a reformulation of the last statement is $\Pi_k^-(P_k)=\Pi_k^-(P_{k-1,k})$ and $\Pi_k^+(P_k)=\Pi_k^+(P_{k,k+1})$

We first reformulate the problem in terms of the local model from \S \ref{sec_local}. 

We fix $k$ henceforth and denote the half-plane $\CH_\k$ simply by $\CH$. To prove the theorem, we can restrict our attention to paths $\beta\in\Se_{k}\cup \Se_{k, k+1}\cup \Se_{k+1}$. By Lemma \ref{lemma_alpha}, each such $\beta$ is equivalent to a composition $\tilde{\beta}\cdot\alpha_k$ for some path $\tilde{\beta}$ in  $\CH$ issuing from $0$. We shall study alternatively the comparison transformation $P_{\tilde{\beta}}(t)\in\SL(E_0)$ of $\tilde{\beta}$, defined with respect to the base point $0\in \CH$. 

Let
$
\a, \a_0: \CH\times \CH\rightarrow \SL(3,\mathbb{R})
$
be the two-pointed parallel transport maps of the connections $\D=\dif+A$ and $\D_0=\dif+A_0$ (see \S \ref{sec_wang} and \S \ref{sec_wangass} for the notation), respectively, under the equilateral frame $(e_1, e_2, e_3)$ of $E$ over $\CH$. The matrix representative of $P_{\tilde{\beta}}(t)$ under the basis  $(e_1(0), e_2(0), e_3(0))$  is
$$
P_{\tilde{\beta}}(t)=\a(0,\tilde{\beta}(t))\a_0(\tilde{\beta}(t),0).
$$
Here $\tilde{\beta}(t)$ is the parametrization of $\tilde{\beta}$ by arc-length with respect to $|\ve{b}|^\frac{2}{3}$.
Note that the original comparison transformation that we wish to study is given by 
\begin{equation}\label{eqn_pbetadef2}
P_\beta(t+t_0)=\a(\alpha_k^{-1})\,P_{\tilde{\beta}}(t)\,\a_0(\alpha_k),
\end{equation}
where $t_0$ is the  $|\ve{b}|^\frac{2}{3}$-length of $\alpha_k$.

Clearly $P_\beta(t)$ converges if and only if $P_{\tilde{\beta}}(t)$ does. In view of Lemma \ref{lemma_alpha}, we can restate part (\ref{item_ocompa1}), (\ref{item_ocompa2})  and (\ref{item_ocompa3})  of the proposition as statements (\ref{item_compa1}'), (\ref{item_compa2}') and (\ref{item_compa3}') below, respectively. In these statements,  $\tilde{\beta}:[0,+\infty)\rightarrow\CH$
is any path satisfying $\tilde\beta(0)=0$, $|\tfrac{\dif}{\dif t}\tilde{\beta}(t)|\equiv1$ and 
$$
\tilde\beta(t)=e^{\theta\ima}+\zeta_0, \quad \forall t\geq M
$$ 
for some $\zeta_0\in\CH$, $\theta=\theta(\tilde{\beta})\in(-\tfrac{\pi}{2}, \tfrac{\pi}{2})$ and $M>0$. The paths $\tilde{\beta}_0$ and $\tilde{\beta}_1$ are under the some hypotheses. 

\vspace{5pt}
\begin{enumerate}[(1')]
\item\label{item_compa1}
%Given a path parametrized by arc-length
%$$
%\tilde{\beta} :[0,+\infty)\rightarrow\CH
%$$ 
%such that $\tilde{\beta}(0)=0$ and $\tilde{\beta}([T,+\infty))$ is a ray when $T$ is sufficiently large, 
If $\theta(\tilde{\beta})$ belongs to either of the intervals
$\left(-\tfrac{\pi}{2},-\tfrac{\pi}{6}\right)$, $\left(-\tfrac{\pi}{6},\tfrac{\pi}{6}\right)$ or $\left(\tfrac{\pi}{6},\tfrac{\pi}{2}\right)$, then the limit
$P_{\tilde{\beta}}:=\lim_\tinf P_{\tilde{\beta}}(t)\in\SL(E_0)$ exists.

\item\label{item_compa2}
Given $\tilde{\beta}_0$ and $\tilde{\beta}_1$, if $\theta(\tilde{\beta}_0)$ and $\theta(\tilde{\beta}_1)$ are both in one of the intervals $\left(-\tfrac{\pi}{2},-\tfrac{\pi}{6}\right)$, $\left(-\tfrac{\pi}{6},\tfrac{\pi}{6}\right)$ or $\left(\tfrac{\pi}{6},\tfrac{\pi}{2}\right)$, then 
$P_{\tilde{\beta}_0}=P_{\tilde{\beta}_1}$.
\item\label{item_compa3}
$$
P_{\tilde{\beta}_1}^{-1}P_{\tilde{\beta}_0}=
\begin{cases}
\begin{pmatrix}
1&&*\\
&1&\\
&&1
\end{pmatrix}
\mbox{ if } \theta(\tilde{\beta}_0)\in \left(-\tfrac{\pi}{6},\tfrac{\pi}{6}\right) \mbox{ and }\theta(\tilde{\beta}_1)\in\left(\tfrac{\pi}{6},\tfrac{\pi}{2}\right),\\[20pt]
\begin{pmatrix}
1&*&\\
&1&\\
&&1
\end{pmatrix}
\mbox{ if }\theta(\tilde{\beta}_0)\in \left(-\tfrac{\pi}{6},\tfrac{\pi}{6}\right) \mbox{ and }\theta(\tilde{\beta}_1)\in\left(-\tfrac{\pi}{2},-\tfrac{\pi}{6}\right).
\end{cases}
$$
\end{enumerate}
Statement (\ref{item_compa3}') is equivalent to (\ref{item_ocompa3}) because, on one hand, $P_{\beta_1}^{-1}P_{\beta_0}$ is conjugate to $P_{\tilde{\beta}_1}^{-1}P_{\tilde{\beta}_0}$ through  $\a_0(\alpha_k^{-1})$ by (\ref{eqn_pbetadef2}); on the other hand, by Lemma \ref{lemma_mod3}, the basis used in the definition of $G_k^\pm$  is the $\a_0(\alpha_k^{-1})$-translates of the basis  $(e^\k_0(0),e^\k_1(0),e^\k_2(0))$ used in  (\ref{item_compa3}').

%For any $\beta_0,\beta_1\in \Se_{k}\cup\Se_{k,k+1}\cup\Se_{k+1}$, we have $P_{\beta_1}^{-1}P_{\beta_0}=\a_0(\alpha_k^{-1})P_{\tilde{\beta}_1}^{-1}P_{\tilde{\beta}_0}\,\a_0(\alpha_k)$. Thus an equivalent statement of the part of  (\ref{item_ocompa3}) concerning $P_k^{-1}P_{k,k+1}$ is that, if $\beta_0\in\Se_{k,k+1}$ and $\beta_1\in\Se_k$, then $P_{\tilde{\beta}_1}^{-1}P_{\tilde{\beta}_0}$ is a unipotent element in $\SL(E_0)$ stabilizing the line in $\mathbb{P}(E_0)$ spanned by $[e_1^\k(0)]$ and $[e_3^\k(0)]$ while pointwise fixing the line spanned by $[e_1^\k(0)]$ and $[e_2^\k(0)]$. But such an element is exactly one with the first matrix expression in (\ref{item_compa3}'). The part concerning $P_{k+1}^{-1}P_{k,k+1}$ is similar.

\vspace{8pt}

We proceed to prove the statements (\ref{item_compa1}'), (\ref{item_compa2}') and (\ref{item_compa3}'). The idea is to take derivatives of the $\SL(3,\mathbb{R})$-valued functions in question, resulting in linear ODEs, and then apply certain asymptotic result for such ODEs to get the required limit. 

In the following proof, we re-denote $\tilde{\beta}$ by $\beta$ for tidiness. This is not to be confused with the $\beta$ from the original statement of Theorem \ref{thm_compa}.

 \begin{proof}
(\ref{item_compa1}')
We compute the derivative of $P_\beta(t)$ using the last property of two-pointed parallel transport maps given in \S \ref{sec_pt}, obtaining
\begin{align*}
\tfrac{\dif}{\dif t}P_\beta(t)&=\a(0,\beta(t))A(\dot{\beta}(t))\a_0(\beta(t),0)-\a(0,\beta(t))A_0(\dot{\beta}(t))\a_0(\beta(t),0)\\
&=P_\beta(t)\Ad_{\a_0(0,\beta(t))}\left(A(\dot{\beta}(t))-A_0(\dot{\beta}(t))\right)=P_\beta(t) N(t),
\end{align*}
where we put 
$$N(t):=\Ad_{\a_0(0,\beta(t))}(A(\dot{\beta}(t))-A_0(\dot{\beta}(t))).$$
A result on asymptotics of linear ODEs (Lemma B.1 in \cite{dumas-wolf}) then ensures that $\lim_\tinf P_\beta(t)$ exists if
\begin{equation}\label{eqn_condi}
\int_0^{+\infty}\|N(t)\|\dif t<+\infty.
\end{equation}
Here $\|\cdot\|$ denotes a matrix norm.

The expressions of $\D$ and $\D_0$ from (\ref{eqn_d}) and (\ref{eqn_wangcon0}) yield
\begin{align*}
A-A_0=\Q^{-1}
\begin{pmatrix}
\pa u&(e^{-u}-1)\dbzeta&0\\[7pt]
(e^{-u}-1)\dzeta&\bpa u&0\\[7pt]
(e^u-1)\dbzeta&(e^u-1)\dzeta&0
\end{pmatrix}
\Q,
\end{align*}
where the function $u$ is as in \S \ref{sec_fine}.
Therefore, Lemma \ref{thm_fine} and the assumption that $|\dot{\beta}(t)|=1$ imply  
\begin{equation}\label{eqn_condi1}
\|A(\dot{\beta}(t))-A_0(\dot{\beta}(t))\|\leq C\,|\beta(t)|^{-\frac{1}{2}}e^{-2\sqrt{3}|\beta(t)|}, \quad \forall\  t\geq 0
\end{equation}

On the other hand, by virtue of the discussion on the spectral radius $\Ad_{\a_0(0,\zeta)}$ in \S \ref{sec_eigen}, our assumption on $\theta(\beta)=\lim_\tinf \arg(\beta(t))$ implies that 
\begin{equation}\label{eqn_condi2}
\rho(\Ad_{\a_0(0,\beta(t))})=e^{\varpi(\arg(\beta(t)))|\beta(t)|}\leq e^{(2\sqrt{3}-\delta)|\beta(t)|}
\end{equation}
for some constant $\delta>0$ when $t$ is sufficiently large.

Combining the estimates  (\ref{eqn_condi1}) and (\ref{eqn_condi2}), we get
$$
\left\|\Ad_{\a_0(0,\beta(t))}\left(A(\dot{\beta}(t))-A_0(\dot{\beta}(t))\right)\right\|\leq C|\beta(t)|^{-\frac{1}{2}}e^{-\delta|\beta(t)|}\leq C't^{-\frac{1}{2}}e^{-\delta t},
$$
where the second equality is because $t-\mu\leq |\beta(t)|\leq t$ for some constant $\mu$, as implied by the hypotheses on $\beta$. The required condition (\ref{eqn_condi}) follows.

(\ref{item_compa2}')
For $t\geq 0$ and $s\in[0,1]$, put 
%The interval $[\theta_0, \theta_1]$ is contained in either $(-\tfrac{\pi}{2},-\tfrac{\pi}{6})$, $(-\tfrac{\pi}{6}, \tfrac{\pi}{6})$ or $(\tfrac{\pi}{2}, \tfrac{\pi}{6})$ by assumption. 
\begin{align*}
\beta_s(t)&:=(1-s)\beta_0(t)+s\beta_1(t)\in \CH
\end{align*}
It is easy to see that there exist $\lambda>1$ and $\mu>0$ such that 
\begin{equation}\label{eqn_abst}
\frac{1}{\lambda}t- \mu\leq |\beta_s(t)|\leq t,\quad \big|\tfrac{\pa}{\pa s}\beta_s(t)\big|\leq \lambda t+\mu
\end{equation}
for any $s$ and $t$. Put
\begin{align}\label{eqn_defq}
Q(s,t)&:=\a_0(0,\beta_s(t))\a(\beta_s(t),\beta_0(t))\a_0(\beta_0(t),0)\in \SL(3,\mathbb{R}).
\end{align}
Note that 
$$
Q(1,t)=\a_0(0,\beta_1(t))\a(\beta_1(t),0)\a(0,\beta_1(t))\a_0(\beta_0(t),0)=P_{\beta_1}(t)^{-1}P_{\beta_0}(t).
$$
Therefore $P^{-1}_{\beta_1}P_{\beta_0}$ is the limit of $Q(1,t)$ as $t\rightarrow+\infty$. By part (\ref{item_compa1}'), this limit exists if neither $\theta(\beta_0)$ nor $\theta(\beta_1)$ equals $\pm\tfrac{\pi}{6}$.

In order to evaluate the limit of $Q(1,t)$, we compute $\tfrac{\pa}{\pa s}Q(s,t)$ similarly as above, obtaining
 \begin{align*}
\tfrac{\pa}{\pa s}Q(s,t)&=\a_0(0,\beta_s(t))A_0\left(\tfrac{\pa}{\pa s}\beta_s(t)\right)\a(\beta_s(t),\beta_0(t))\a_0(\beta_0(t),0)\\
&\quad-\a_0(0,\beta_s(t))A\left(\tfrac{\pa}{\pa s}\beta_s(t)\right)\a(\beta_s(t),\beta_0(t))\a_0(\beta_0(t),0)\\
&=\Ad_{\a_0(0,\beta_s(t))}\left(A_0\left(\tfrac{\pa}{\pa s}\beta_s(t)\right)-A\left(\tfrac{\pa}{\pa s}\beta_s(t)\right)\right)  Q(s,t).
\end{align*}
Let us set
\begin{align}\label{eqn_defm}
M(s,t):=\Ad_{\a_0(0,\beta_s(t))}\left(A_0\left(\tfrac{\pa}{\pa s}\beta_s(t)\right)-A\left(\tfrac{\pa}{\pa s}\beta_s(t)\right)\right),
\end{align}
so that 
$
\tfrac{\pa}{\pa s}Q(s,t)=M(s,t)Q(s,t)
$.
This can be viewed as a family of ODEs on the interval $[0,1]$ depending on the parameter $t\in[0,+\infty)$. The following 
asymptotic result applies.

\begin{lemma}\label{lemma_ode}
Let $Q(s,t)$ and $M(s,t)$ be smooth functions on $[0,1]\times[0,+\infty)$ taking values in $\GL(3,\mathbb{R})$ and $\gl_3\mathbb{R}$, respectively, satisfying the following conditions:
\begin{enumerate}[(i)]
\item
$\tfrac{\pa}{\pa s}Q(s,t)=M(s,t)Q(s,t)$.
\item
$Q(0,t)=\id$ for any $t$.
\item 
There exist $X\in\frak{gl}_3\mathbb{R}$ and a real-valued function $f(s,t)$ such that \begin{equation}\label{eqn_b21}\tag{a}
\lim_\tinf \max_{s\in[0,1]}\|M(s,t)-f(s,t)X\|=0,
\end{equation}
\begin{equation}\label{eqn_b22}\tag{b}
\sup_{t\in[0,+\infty)}\int_0^1|f(s,t)|\dif s<+\infty.
\end{equation}
\end{enumerate}
Then we have
$$
\lim_\tinf \left\|Q(1,t)-\exp\left(X\int_0^1f(s,t)\dif s\right)\right\|=0.
$$
\end{lemma}
\vspace{10pt}
This lemma is an immediate consequence of Lemma B.2 in \cite{dumas-wolf}.

We need to prove $\lim_\tinf Q(1,t)=\id$. Using the above lemma, it is sufficient to show that 
\begin{equation}\label{eqn_proofcompa2}
\lim_\tinf \max_{s\in[0,1]}\|M(s,t)\|=0.
\end{equation}
But this is similar to the proof of part (\ref{item_compa1}'): on one hand, Lemma \ref{thm_fine} yields
\begin{equation}\label{eqn_a0ast}
\left\|A_0\left(\tfrac{\pa}{\pa s}\beta_s(t)\right)-A\left(\tfrac{\pa}{\pa s}\beta_s(t)\right)\right\|\leq C\,|\beta_s(t)|^{-\frac{1}{2}}e^{-2\sqrt{3}|\beta_s(t)|}\left|\tfrac{\pa}{\pa s}\beta_s(t)\right|,
\end{equation}
whereas the discussion on $\rho(\Ad_{\a_0(0,\zeta)})$ in \S \ref{sec_eigen} and the assumption that $\theta(\beta_0)$ and $\theta(\beta_1)$ are both in either $(-\tfrac{\pi}{2}, -\tfrac{\pi}{6})$, $(-\tfrac{\pi}{6}, \tfrac{\pi}{6})$ or $(\tfrac{\pi}{6}, \tfrac{\pi}{2})$ yield 
$$
\rho(\Ad_{\a_0(0,\beta_s(t))})\leq e^{(2\sqrt{3}-\delta)|\beta_s(t)|}.
$$
Eq.(\ref{eqn_proofcompa2}) now follows from the above two  inequalities and (\ref{eqn_abst}).

 (\ref{item_compa3}') First assume
$$
\theta(\beta_0)\in\left(-\tfrac{\pi}{6}, \tfrac{\pi}{6}\right),\quad \theta(\beta_1)\in\left(\tfrac{\pi}{6}, \tfrac{\pi}{2}\right).
$$
 Part (\ref{item_compa2}') implies that $P_{\beta_1}^{-1}P_{\beta_0}$ is the same for any such $\beta_0$ and $\beta_1$, thus we can choose 
$$
\beta_0(t)=t,\quad \beta_1(t)=t\,e^{\tfrac{\pi}{3}\ima}.
$$

For $t\geq 0$ and $s\in[0,1]$, put
$$
\beta_s(t):=t\,e^{\tfrac{\pi}{3}s\ima}.                                                                                                                                        
$$
Then define $Q(s,t)$ and $M(s,t)$ in the same way as above.
%(\ref{eqn_defq}) and (\ref{eqn_defm}), 

We still have $\tfrac{\pa}{\pa s}Q(s,t)=M(s,t)Q(s,t)$ and $P^{-1}_{\beta_1}P_{\beta_0}=\lim_\tinf Q(1,t)$.  We shall apply Lemma \ref{lemma_ode} again to obtain the latter limit.

The estimate (\ref{eqn_a0ast}) still holds and gives
\begin{equation}\label{eqn_a0ast2}
\left\|A_0\left(\tfrac{\pa}{\pa s}\beta_s(t)\right)-A\left(\tfrac{\pa}{\pa s}\beta_s(t)\right)\right\|\leq C\sqrt{t}\,e^{-2\sqrt{3}t}.
\end{equation}

One the other hand, with the notations from \S \ref{sec_eigen}, the eigenvalue of $\Ad_{\a_0(0,\beta_s(t))}$ on $E_{ij}$ is 
$$\exp\big(\varpi_{ij}(\arg(\beta_s(t)))|\beta_s(t)|\big)=\exp\big(\varpi_{ij}(\tfrac{\pi}{3}s)t\big).$$
Since $\tfrac{\pi}{3}s$ takes values in the interval $[0,\tfrac{\pi}{3}]$ in which the only odd multiple of $\frac{\pi}{6}$ is $\frac{\pi}{6}$ itself, fact (\ref{item_eigen2}) in \S \ref{sec_eigen} implies that, for any $(i,j)\neq (1,3)$, the eigenvalue on $E_{ij}$ is bounded by $e^{(2\sqrt{3}-\delta)t}$ for some $\delta>0$. In view of (\ref{eqn_a0ast2}), we conclude that
$$
\max_{s\in[0,1]}\|M(s,t)-M_{13}(s,t)E_{13}\|\leq C\,\sqrt{t}\,e^{-\delta t},
$$
where we let $M_{13}(s,t)$ denote the $(1,3)$ entry of $M(s,t)$. Thus condition (\ref{eqn_b21}) in Lemma \ref{lemma_ode} is satisfied for $X=E_{13}$ and $f=M_{13}$. Moreover, by definition,
\begin{align*}
\varpi_{13}(\theta)=2\re(e^{\theta\ima}-e^{2\pi\ima/3}e^{\theta\ima})=2(\cos(\theta)-\cos(\theta+\tfrac{2\pi}{3}))=2\sqrt{3}\cos(\theta-\tfrac{\pi}{6}).
\end{align*}
Therefore
\begin{align*}
|M_{13}(s,t)|&\leq C\sqrt{t}\,e^{-2\sqrt{3}\, t\left(1-\cos(\theta-\tfrac{\pi}{6})\right)}=C\sqrt{t}\,e^{-4\sqrt{3}\, t\sin^2\left(\frac{\theta-\pi/6}{2}\right)}\leq C\sqrt{t}\,e^{-c\, t(s-\frac{1}{2})^2}
\end{align*}
for some constants $c>0$. Up to a constant factor, the last term above, viewed as a function of $s$ parametrized by $t$, is a Gaussian distribution with average $\frac{1}{2}$ and with variance tending to $0$ as $t\rightarrow+\infty$. Hence its integral over $[0,1]$ tends to a constant. Thus condition (\ref{eqn_b22}) is satisfied as well.

Applying  Lemma \ref{lemma_ode}, we conclude that if we let $(t_n)$ be a sequence in $\mathbb{R}_{\geq 0}$ tending to $+\infty$ such that $\int_0^1f(s,t_n)\dif s$ admits a limit $r\in\mathbb{R}$ as $n\rightarrow+\infty$ (such a sequence exists because $|\int_0^1f(s,t_n)\dif s|\leq \int_0^1|f(s,t_n)|\dif s$ is bounded as we have just seen), then 
$\lim_{n\rightarrow+\infty}Q(1,t_n)=\exp(rE_{13})$. But we already know that $Q(1,t)$ converges, so we obtain $P^{-1}_{\beta_1}P_{\beta_0}=\lim_\tinf Q(1,t)=\exp(rE_{13})$ as required.

In the case $\theta(\beta_0)\in\left(-\tfrac{\pi}{2}, -\tfrac{\pi}{6}\right)$,  $\theta(\beta_1)\in\left(-\tfrac{\pi}{6}, \tfrac{\pi}{6}\right)$, we apply the same argument to
 $$
\beta_0(t)=t,\quad \beta_1(t)=t\,e^{-\tfrac{\pi}{3}\ima},\quad\beta_s(t)=t\,e^{-\tfrac{\pi}{3}s\ima}.
$$
Now $\arg(\beta_s(t))=-\frac{\pi}{3}s$ takes values in $[-\frac{\pi}{3},0]$, in which the only odd multiple of $\frac{\pi}{6}$ is $-\frac{\pi}{6}$. Fact (\ref{item_eigen2}) in \S \ref{sec_eigen} now implies that the only eigenvalue not controlled by $e^{2\sqrt{3}-\delta}$ is the one on $E_{12}$. Applying  Lemma \ref{lemma_ode} to $X=E_{12}$ and $f=M_{12}$ similarly as above, we obtain $P^{-1}_{\beta_1}P_{\beta_0}=\lim_\tinf Q(1,t)=\exp(rE_{12})$ as required.
\end{proof}

\subsubsection{Proof of Theorem \ref{thm_finding}}\label{sec_prooffinding}
We now deduce Theorem \ref{thm_finding} from Proposition \ref{prop_auxpoly} and Theorem \ref{thm_compa}.

By definition of Wang's developing map (see \S \ref{sec_wang}), if the limits in the definitions of $P_\beta$ and $\devlima(\beta)$ both exist, then 
\begin{align*}
\devlim(\beta)&=\lim_\tinf\para(\beta_{[0,t]}^{-1})\underline{1}_{\pi(\beta_{[0,t]}^{-1})}=\lim_\tinf\para(\beta_{[0,t]}^{-1})\a_0(\beta_{[0,t]}^{-1})\a_0(\beta_{[0,t]}^{-1})\underline{1}_{\pi(\beta_{[0,t]}^{-1})}\\
&=\lim_\tinf P_\beta(t)\dev_0(\beta_{[0,t]}^{-1})=P_\beta(\devlima(\beta)).
\end{align*}
Here, recall that $\pi:\widetilde{\Sigma}_m\rightarrow \Sigma$ is the projection and $\underline{1}$ is the canonical section of $E$.

As $\beta$ runs over each of the sets given in the first row of the following table, Proposition \ref{prop_auxpoly} and Theorem \ref{thm_compa} guarantee existence of both limits and provide a description of $P_\beta$ and the set formed by the $\devlima(\beta)$'s, as shown in the second and third row of the table, respectively
\footnote{
Here and below, by an abuse of notation, we do not distinguish the point $\hat{X}_{[k]}$ and the set with one element formed by this point.
}.

\vspace{5pt}

%\begin{tabular}{|c|c|c|c|c|c|c|c|}
%\hline
%$\beta$&$\C_{k-1,k}$&$\C_{k}\cap\Se_{k-1,k}$&$\U_k^-$&$\Se_{k}$&$\U_k^+$&$\C_{k}\cap \Se_{k, k+1}$&$\C_{k,k+1}$
%\\[4pt]\hline
%$P_\beta$&$P_{k-1,k}$&$P_{k-1,k}$&N/A&$P_k$&N/A&$P_{k,k+1}$&$P_{k,k+1}$
%\\[4pt]\hline
%$\{\devlima(\beta)\}$&$\hat{X}_{[k-1],[k]}^\circ$&\multicolumn{5}{|c|}{$\hat{X}_{[k]}$}&$\hat{X}_{[k],[k+1]}^\circ$
%\\[3pt]\hline
%\end{tabular}

\begin{tabular}{|c|c|c|c|c|c|}
\hline
$\beta$&$\C_{k-1,k}$&$\C_{k}\cap\Se_{k-1,k}$&$\Se_{k}$&$\C_{k}\cap \Se_{k, k+1}$&$\C_{k,k+1}$
\\[4pt]\hline
$P_\beta$&$P_{k-1,k}$&$P_{k-1,k}$&$P_k$&$P_{k,k+1}$&$P_{k,k+1}$
\\[4pt]\hline
$\{\devlima(\beta)\}$&$\hat{X}_{[k-1],[k]}^\circ$&\multicolumn{3}{|c|}{$\hat{X}_{[k]}$}&$\hat{X}_{[k],[k+1]}^\circ$
\\[3pt]\hline
\end{tabular}

\vspace{5pt}

Stable paths belonging to $\C_k$ correspond to the three columns in the middle where $\devlima(\beta)=\hat{X}_{[k]}$. Theorem \ref{thm_compa}  (\ref{item_ocompa3}) implies that  both $P_{k-1,k}^{-1}P_k$ and $P_{k,k+1}^{-1}P_k$ fixes $\hat{X}_{[k]}$, thus $\devlim(\beta)=P_\beta(\devlima(\beta))$ is the point 
$$
X_k:=P_{k-1,k}(\hat{X}_{[k]})=P_k(\hat{X}_{[k]})=P_{k,k+1}(\hat{X}_{[k]})
$$
for any stable $\beta\in\C_k$. Part (\ref{item_finding1}) of the theorem is proved.

%In order to prove part (\ref{item_finding1}) of the theorem, it remains to be shown that $\devlim(\beta)=X_k$ for $\beta\in\U_k^\pm$ as well.
%
%To this end, we make use of the $G_k^\pm$-valued function $\Xi_k^\pm(t)$ produced by Proposition \ref{prop_gpm}. For any $\beta\in\U_k^\pm$, we can write
%\begin{align*}
%\devlim(\beta)=\lim_\tinf P_\beta(t)\,\Xi_k^\pm(t)^{-1}  \lim_\tinf \Xi_k^\pm(t)\dev_0(\beta_{[0,t]}^{-1}).
%\end{align*}
%The first limit is $P_k$ by construction. Since the growth rate of $\Xi_k^\pm(t)$ is controlled by $\sqrt{t}$ while $\dev_0(\beta_{[0,1]}^{-1})$ converges exponentially fast to $\hat{X}_{[k]}$, one easily sees that the second limit above is $\hat{X}_{[k]}$. Thus $\devlim(\beta)=P_k(\hat{X}_{[k]})=X_k$ as required. The proof of part (\ref{item_finding1}) is now complete.

Part (\ref{item_finding2}) is proved by setting
\begin{align*}
X_{k-1,k}&:=P_{k-1,k}(\hat{X}_{[k-1],[k]})=P_k(\hat{X}_{[k-1],[k]}),\\
X_{k,k+1}&:=P_{k,k+1}(\hat{X}_{[k],[k+1]})=P_k(\hat{X}_{[k],[k+1]}),
\end{align*}
where the second equalities in both lines follows from the fact that $P_{k-1,k}^{-1}P_k$ and $P_{k,k+1}^{-1}P_k$ pointwise fix the segments $\hat{X}_{[k-1],[k]}$ and $\hat{X}_{[k],[k+1]}$, respectively, as \mbox{implied} by Theorem \ref{thm_compa} (\ref{item_ocompa3}). Thus we get 
$$
\{\devlim(\beta)\}_{\beta\in\C_{k-1,k}}=\{P_\beta(\devlima(\beta))\}_{\beta\in\C_{k-1,k}}=P_{k-1,k}(\hat{X}_{[k-1],[k]}^\circ)=X_{k-1,k}^\circ
$$
as required, and similarly $\{\devlim(\beta)\}_{\beta\in\C_{k,k+1}}=X_{k,k+1}^\circ$. Moreover, $X_{k-1,k}$ and $X_{k,k+1}$ are non-collinear segments sharing the endpoint $X_k=P_k(\hat{X}_{[k]})$ because they are $P_k$-translates of the segments $\hat{X}_{[k-1],[k]}$ and $\hat{X}_{[k],[k+1]}$, which are non-collinear and share the endpoint $\hat{X}_{[k]}$ by construction.

To prove the first statement of part (\ref{item_finding3}), we note that, by property (\ref{item_winding3}) of winding numbers in \S \ref{sec_winding},
 $\beta$ belongs to $\C_{k+n}$ if and only if $\beta\cdot \gamma_{\tilde{p}}$ belongs to  $\C_{k}$. 
%On the other hand, viewing the truncated path $\beta_{[0,t]}$ as a point in $\widetilde{\Sigma}_m$, the composition $\gamma_{\tilde{p}}\cdot\beta_{[0,t]}=(\gamma_{\tilde{p}}\cdot\beta)_{[0,t]}$ is just the translate of $\beta_{[0,t]}$ by the deck action of $\gamma_{\tilde{p}}\in\pi_1(\Sigma, m)$, so equivariance of the developing map
Using $\hol$-equivariance of $\dev$, we get,  for any $\beta\in\C_{k+n}$ 
$$
X_{k}=\devlim(\beta\cdot\gamma_{\tilde{p}})=
\lim_\tinf \dev(\gamma_{\tilde{p}}^{-1}\cdot\beta_{[0,t]}^{-1})
=
\lim_\tinf\hol_{\tilde{p}}^{-1}(\dev(\beta_{[0,t]}^{-1}))=\hol_{\tilde{p}}^{-1}(X_{k+n})
$$
as required. 

Finally, looking at each case in Lemma \ref{lemma_rough} respectively, one sees that, in general, if a developed boundary $\devbd{\tilde{p}}$ contains some polygon $C$ twisted by $\hol_{\tilde{p}}$, then the whole $\devbd{\tilde{p}}$ is the union of $C$ and the accumulation points of $C$. The last statement in part (\ref{item_finding3}) follows.

\subsection{Poles of order $3$}\label{sec_leq}
Now assume that $p$ is a third order pole. We can suppose either $\Sigma= \mathbb{C}^*$ or $\Sigma$ has negative Euler characteristic, because if $\Sigma=\mathbb{C}$ then the pole $\infty$ has order at least $6$. Let $R\in\mathbb{C}^*$ be the residue of $\ve{b}$ at $p$.
%because the line bundle $K^3$ over $\mathbb{CP}^1$ has degree $-6$.

%The \emph{residue} of $\ve{b}$ at $p$, denoted by $R\in\mathbb{C}^*$, is by definition the coefficient of the degree $-3$ term in the Laurent expansion of $\ve{b}$ with respect to a conformal local coordinate centered at $p$. One readily checks that $R$ does not depend on the choice of coordinate (this is not true when the pole has order $\geq 4$). 

In this case, statement (III) from the introduction is contained in Theorem \ref{thm_pole3} below. 
As an intermediate step in the proof, we determine the end holonomy $\hol_{\tilde{p}}$ in Theorem \ref{thm_holo}. This covers Theorem \ref{intro_thm2} from the introduction.

\subsubsection{Statements of the main results}
We define the collection of paths $\C$ using the same Definition \ref{def_U} as in  the discussion of higher order poles. The counterpart to Theorem \ref{thm_finding} in the present case is the following
%Note that each $\beta\in\C$ is eventually a ray in $H/\sim$ along a direction perpendicular to $\pa H$.

%Unlike the case of poles of order $\geq 4$, now Condition (\ref{item_U3}) in Definition \ref{def_U} does not follow from the other two conditions. Indeed, there exists $|\ve{b}|^\frac{2}{3}$-geodesics converging to $p$ not along any direction, namely, rays in $H$ which are not perpendicular to $\pa H$. Seen in the coordinate $z$, these are spirals converging to $p$ while wrapping around $p$ infinitely many times.
\begin{theorem}[\textbf{developed boundary at third order poles}]\label{thm_pole3}
${}$
\begin{enumerate}
%\item\label{item_pole31} If $R=0$ then $p$ is a cusp.
\item\label{item_pole33}
If $\re(R)<0$ then $p$ is a V-end. $\devlim(\beta)$ is the saddle fixed point of the hyperbolic holonomy $\hol_{\tilde{p}}$ for any $\beta\in\C$.
\item\label{item_pole34}
If $\re(R)>0$ then $p$ is a geodesic end with hyperbolic holonomy. Let $x_+$ and $x_-$ denote the attracting and repelling fixed points of $\hol_{\tilde{p}}$, respectively.
\begin{itemize}
\item if $\im(R)<0$ then $\devlim(\beta)=x_+$ for any $\beta\in\C$.

\item if $\im(R)>0$ then $\devlim(\beta)=x_-$ for any $\beta\in\C$.

\item if $R\in\mathbb{R}_+$ then $\devlim(\beta)$ exists for any $\beta\in\C$ and $\{\devlim(\beta)\}_{\beta\in \C}$ is the interior of a segment joining $x_+$ and $x_-$.
\end{itemize}

\item\label{item_pole32}
If $R\in\ima\mathbb{R}^*$ then $p$ is a geodesic end with quasi-hyperbolic or planar 
holonomy and $\devlim(\beta)$ is the double fixed point of $\hol_{\tilde{p}}$ for any $\beta\in\C$.
\end{enumerate}
\end{theorem}
In part (\ref{item_pole32}), although a planar end holonomy has two double fixed points, only one of them is in the developed boundary and $\devlim(\beta)$ is supposed to be this one. Note that planar holonomy only occurs in case (\ref{item_tri2}) of Proposition \ref{prop_triangle} (see Lemma \ref{lemma_planar}). Theorem \ref{thm_pole3} can actually be proved by direct calculations in this case.

%Loftin has obtained many statements in Theorem \ref{thm_pole3} more or less inexplicitly in \cite{loftin_compactification}.

The statements concerning $\devlim(\beta)$ will be proved using similar ideas as in our treatment of higher order poles: we first investigate the ``auxiliary developing limits'' given by Wang's developing map associated to $(2^\frac{1}{3}|\ve{b}|^\frac{2}{3}, \ve{b})$; then we compare the actual developing limits with the auxiliary ones by studying the comparison transformation $P_{\beta}$ defined in \S \ref{sec_compa}. 

However, while a serious investigation of unstable directions
is avoidable for poles of order $\geq4$, it is inevitable for third order poles. In fact, when $\re(R)=0$, all directions that we are concerned with are ``unstable''. A improvement of the ODE techniques used earlier will be developed in order to tackle this situation.

Another complication for third order poles is that the information about $\devlim(\beta)$ does not determine $\devbd{\tilde{p}}$ immediately. In fact, when $R\notin\mathbb{R}_{\geq 0}$, as stated in Theorem \ref{thm_pole3}, the $\devlim(\beta)$'s only detect a single point, while $\devbd{\tilde{p}}$ is expected to have a continuum of points. 
Our idea to determine $\devbd{\tilde{p}}$ is to further study the asymptotic direction along which the path $\dev(\beta_{[0,t]}^{-1})$ converges to $\devlim(\beta)$.

\subsubsection{Local model for $(\Sigma,\ve{b})$}\label{sec_half3}
As at the beginning of \S \ref{sec_model}, let $z$ be a coordinate around $p$
such that $\ve{b}$ has the normal form 
\begin{equation}\label{eqn_b3}
\ve{b}=Rz^{-3}\dz^3.
\end{equation}
Let $U=\{0<|z|<a\}\subset\Sigma$ be a punctured neighborhood of $p$ where $z$ is defined. Since a dilation of $z$ does not change the expression (\ref{eqn_b3}), we can assume $a>1$.

By virtue of the expression (\ref{eqn_b3}), one easily constructs a local model of $(\Sigma, \ve{b})$ in the spirit of \S \ref{sec_local} using the exponential map: pick a cubic root  $(R/2)^\frac{1}{3}$ of $R/2$ and consider the rotated half-plane 
$$
H=-(R/2)^\frac{1}{3}\H,
$$ 
where $\H$ is the right half-plane as before. 

The map
$H\rightarrow U$, $\zeta\mapsto\exp\big((R/2)^{-\frac{1}{3}}\zeta\big)$ 
is invariant under the translation
\begin{equation}\label{eqn_trans}
H\rightarrow H,\quad\zeta\mapsto \zeta+2\pi\ima(R/2)^\frac{1}{3}
\end{equation}
and pulls  $\ve{b}$ back to $2\,\dzeta^3$. Let $H/\sim$ denote the quotient of $H$ by this translation.
We can identify the punctured neighborhood $\{0<|z|<1\}$ of $p$, endowed with the cubic differential $\ve{b}$, with $(H/\sim,\ 2\,\dzeta^3)$. Furthermore, we identify $\{0<|z|\leq 1\}\subset\Sigma$ with $\overline{H}/\sim$.

As before, assume that the metric $g$ is expression in $\overline{H}/\sim$ as
$$
g=2e^u|\dzeta|^2.
$$
Then $u$ satisfies Wang's equation $\lap u=4e^u-4e^{-2u}$. Corollary \ref{coro_non}  implies that $u$ is non-negative and bounded.

The asymptotic estimate for $u$ in Theorem \ref{thm_fine} no longer holds, but we have a slightly weaker estimate. Indeed, since $u$ can be viewed as a function on $\overline{H}$ invariant under the translation (\ref{eqn_trans}), we can apply Lemma \ref{lemma_a1} (\ref{item_a12}) and Corollary \ref{coro_a} from the appendix and get the following
\begin{lemma}[\textbf{Strong estimate for third order pole}]\label{lemma_3estimate}
There  is a constant $C$ such that
$$
0\leq u(\zeta), |\pa_\zeta u(\zeta)|\leq C e^{-2\sqrt{3}\ \dist(\zeta,\pa H)}
$$
for any $\zeta\in H$. Here $\dist(\zeta,\pa H)$ denotes the distance from $\zeta$ to $\pa H$ with respect to the metric $|\dzeta|^2$.
\end{lemma}

\subsubsection{Notations}
We slightly simplify the settings and introduce some notations for the proof of Theorem \ref{thm_pole3}. 

First of all, recall that the definition of $H$ depends on a choice of the cubic root $(R/2)^\frac{1}{3}$. For the sake of determinancy, we now only let $(R/2)^\frac{1}{3}$ denote the root satisfying $$\arg((R/2)^\frac{1}{3})\in(-\tfrac{\pi}{6}, \tfrac{\pi}{2}].$$

The choices of the base point $m\in\Sigma$ and the point $\tilde{p}\in\Fa{\Sigma, p}$ are not essential for the statement of Theorem \ref{thm_pole3}.
%: if the theorem holds for one choice, it holds for any other as well.
Thus we work henceforth with the following simplest choices: 
\begin{itemize}
\item 
Take the origin $0\in\overline{H}/\sim$ of the local model to be the base point, so that Wang's developing map $\dev$ takes values in $\mathbb{P}(E_0)$, where $E_0$ denotes the fiber of the vector bundle $E\rightarrow\Sigma$ at the base point. 
\item
 Let $\tilde{p}\in\Fa{\Sigma,p}$ be represented by the ray $\alpha$ in $\overline{H}/\sim$ issuing from $0$ and perpendicular to $\pa H$. 
 \end{itemize}
 
 With this choice of $\tilde{p}$, we take $\gamma_{\tilde{p}}$ (see \S \ref{sec_farey} for the definition) to be the loop whose lift to $\overline{H}$ is the segment on $\pa H$ going from $0$ to $-2\pi\ima(R/2)^\frac{1}{3}$.

Any $\beta\in\C$ is equivalent (in the sense of Definition \ref{def_U}) to a path in $\overline{H}/\sim$ which issues from $0$ and eventually becomes a ray. 
The ray is parallel to $\alpha$, otherwise $\beta$ is a spiral viewed in the coordinate $z$ and does not fulfill condition (\ref{item_U3}) in Definition \ref{def_U}.  We do not need to distinguish equivalent paths, thus we re-define $\C$ as
$$
\C=\big\{\beta:[0,+\infty)\rightarrow \overline{H}\ \big|\ \beta(0)=0, \ |\dot{\beta}(t)|\equiv1,\  \beta(t)=e^{\theta \ima}t+\zeta_0 \mbox{ for $t$ big enough} \big\},
$$
where $\theta=\arg(-(R/2)^\frac{1}{3})$.

The equilateral frame of $\T\overline{H}\oplus\underline{\mathbb{R}}$ (see \S \ref{sec_wang} for the definition) associated to the natural coordinate $\zeta$ is invariant under translation, hence defines a frame $(e_1, e_2, e_3)$ of $E$ over $\overline{H}/\sim$.  Let $E_0$ denote the fiber of $E$ at $0\in\overline{H}$. In what follows, whenever we write an element of $\SL(E_0)$ as a metric, it is with respect to the basis $(e_1(0), e_2(0), e_3(0))$. Finally, as in \S \ref{sec_titeica}, put 
$$
X_i=[e_i(0)], \quad X_{ij}:=\left\{\big[(1-s)\,e_i(0)+s\,e_j(0)\big]\right\}_{s\in[0,1]}
$$

\subsubsection{Auxiliary developing limits: the case $\re(R)\neq 0$}\label{sec_devbdrerneqzero}
Let $(\dev_0, \hol_0)$ be Wang's developing pair associated to $(2^\frac{1}{3}|\ve{b}|^\frac{2}{3}, \ve{b})$ (\cf \S \ref{sec_wangass}) and let $\devlima(\beta)$ be the corresponding developing limit as in \S \ref{sec_auxpoly}.

Within  $\overline{H}/\sim$, the pair $(2^\frac{1}{3}|\ve{b}|^\frac{2}{3}, \ve{b})$ is expressed under the natural coordinate $\zeta$ of $\overline{H}$ as $(2|\dzeta|^2, 2\,\dzeta^3)$, thus the connection $\D_0$, the parallel transport $\a_0$ and the map $\dev_0$ have the same expressions as in the \c{T}i\c{t}eica example. The investigation in \S \ref{sec_titeica} yields the next proposition.

\begin{proposition}\label{prop_aux3}
Put $\omega:=e^{2\pi\ima/3}$ and
$$
\lambda_1:=e^{-4\pi\im\big((R/2)^\frac{1}{3}\big)},\quad  \lambda_2:=e^{-4\pi\im\big(\omega^2(R/2)^\frac{1}{3}\big)},\quad \lambda_3:=e^{-4\pi\im\big(\omega(R/2)^\frac{1}{3}\big)}.
$$
Then
\begin{enumerate}
\item\label{item_aux31} 
The end holonomy $\hol_0(\gamma_{\tilde{p}})$ is
$$
\hol_0(\gamma_{\tilde{p}})=
\begin{pmatrix}
\lambda_1&&\\
&\lambda_2&\\
&&\lambda_3
\end{pmatrix}.
$$
\item\label{item_aux32}
The developing limit $\devlim_0(\beta)\in\mathbb{P}(E_0)$ exists for any $\beta\in\C$ and we have
$$
\{\devlim_0(\beta)\}_{\beta\in \C}=
\begin{cases}
\big\{X_2\big\}&\mbox{ if } \re(R)>0, \ \im(R)<0,\\
X_{23}^\circ&\mbox{ if }  R\in\mathbb{R}_+,\\
\big\{X_3\big\}&\mbox{ otherwise.}\\
\end{cases}
$$
\end{enumerate}
\end{proposition}

%In the last statement, the notion of ``converging along a direction'' is just the one introduced in \S \ref{sec_winding} in a different context. The \emph{tangent cone} of $\Delta$ at a vertex $x$ is by definition the open sector in $\T_x\mathbb{RP}^2$ delimited by the two segments issuing from $x$.

\begin{proof}
(\ref{item_aux31}) The lift of the loop $\gamma_{\tilde{p}}$ to $\overline{H}$ goes from $0$ to $-2\pi\ima(R/2)^\frac{1}{3}$, hence, in terms of the two-pointed parallel transport map $\a_0(\,\cdot\,,\,\cdot\,): \overline{H}\times\overline{H}\rightarrow\SL(3,\mathbb{R})$ with respect to the equilateral frame, 
$$
\hol_0(\gamma_{\tilde{p}})=\a_0(-2\pi\ima(R/2)^\frac{1}{3}, 0).
$$
The expression of $\a_0(\,\cdot\,,\,\cdot\,)$ given in \S \ref{sec_titeica} then yields the required result.

(\ref{item_aux32}) In the three cases of the required equality,  $\theta=\arg(-(R/2)^\frac{1}{3})=\lim\arg(\beta(t))$ takes values in $\big(\frac{5\pi}{6},\pi\big)$, $\{\pi\}$ and $\big[\pi, \frac{3\pi}{2}\big]$, respectively. In the first and last cases, the required equality follows from Proposition \ref{prop_tit} (\ref{item_tit1}), whereas in the second case it follows from Proposition \ref{prop_tit} (\ref{item_tit2}).

\end{proof}

The hyperbolic/planar nature of $\hol_0(\gamma_{\tilde{p}})$, as well as
the attracting/repelling/ saddle nature of each fixed point, depends on relative largeness of the eigenvalues, as summarized in the following table.

\vspace{8pt}

\begin{tabular}{|c|c|c|c|c|c|c|}
\hline
$\arg(R)$&$(-\tfrac{\pi}{2},0)$&0&$(0,\tfrac{\pi}{2})$&$(\tfrac{\pi}{2},\frac{3\pi}{2})$&$-\tfrac{\pi}{2}$&$\tfrac{\pi}{2}$
\\[4pt]\hline
$\arg(-(R/2)^\frac{1}{3})$&$(\frac{5\pi}{6},\pi)$&$\pi$&$(\pi,\frac{7\pi}{6})$&$(\frac{7\pi}{6},\frac{3\pi}{2})$&$\frac{3\pi}{2}$&$\frac{7\pi}{6}$
\\[4pt]\hline
$X_1$&\multicolumn{3}{|c|}{$0$}&$-$&$-$&$-$
\\[4pt]\hline
$X_2$&\multicolumn{3}{|c|}{$+$}&$+$&$+$&$+$
\\[4pt]\hline
$X_3$&\multicolumn{3}{|c|}{$-$}&$0$&$+$&$-$
\\[4pt]\hline
$\{\devlima(\beta)\}_{\beta\in\C}$&$X_2$&$X_{23}^\circ$&\multicolumn{4}{|c|}{$X_3$}
\\[5pt]\hline
\end{tabular}
\vspace{8pt}

Each column of the table stands for a case with $\arg(R)$ contained in the interval given in the first row; a ``$+$'' (resp. ``$-$'', ``$0$'') sign on the row of $X_i$ means that $\lambda_i$ is the largest (resp. the smallest, intermediate) amongst the three eigenvalues. Thus two ``$-$'' or two ``$+$'' on the same column means $\hol_0(\gamma_{\tilde{p}})$ is planar.

From the table we read off the following
\begin{corollary}\label{coro_pole3} 
${}$
%\begin{enumerate}
%\item\label{item_pole3p11}
If $\re(R)\neq 0$ then $\hol_0(\gamma_{\tilde{p}})$ is hyperbolic. Let $\hat{X}_+$, $\hat{X}_-$ and $\hat{X}_0$ denote the attracting, repelling and saddle fixed points of $\hol_0(\gamma_{\tilde{p}})$, respectively. Then the $\devlima(\beta)$'s can be described as follows.
\begin{itemize}
\item If $\re(R)<0$, then $\devlima(\beta)=\hat{X}_0$ for any $\beta\in\C$.

\item If $R\in\mathbb{R}_+$, then $\{\devlima(\beta)\}_{\beta\in \C}$ is the interior of a segment in $\mathbb{P}(E_m)$ joining $\hat{X}_+$ and $\hat{X}_-$.

\item If $\re(R)>0$ and $\im(R)< 0$, then $\devlima(\beta)=\hat{X}_+$ for any $\beta\in\C$.

\item If $\re(R)>0$ and $\im(R)>0$, then $\devlima(\beta)=\hat{X}_-$ for any $\beta\in\C$.
\end{itemize}
%\item\label{item_pole3p12}
%If $R\in\ima\mathbb{R}^*$ then $\hol_0(\gamma_{\tilde{p}})$ is planar. The isolated fixed point of $\hol_0(\gamma_{\tilde{p}})$ is $[e_1(0)]$ (resp. $[e_2(0)]$) if $R\in\ima\mathbb{R}_-$ (resp. $R\in\ima\mathbb{R}_+$), while $\devlima(\beta)=[e_3(0)]$ for any $\beta\in\C$.
%\end{enumerate}
\end{corollary}

%
%Recall that $\dev_0(\beta)$ denotes the curve in $\Omega$ parametrized by
%$t\mapsto \dev_0(\beta_{[0,t]}^{-1})$ (see \S \ref{sec_devlim} for the notation).
%As mentioned earlier, knowing the limit of $\dev_0(\beta)$ is not enough for our purpose and we shall further study the asymptotic direction of $\dev_0(\beta)$ at the limit. Combining the above table and the statement in Proposition \ref{prop_tit} (\ref{item_tit1}) on asymptotics, we get
%
%\begin{lemma}\label{lemma_x23}
%\pt  If $\re(R)>0$ and $\im(R)\neq0$, then $\dev_0(\beta)$ is asymptotic to $X_{23}$.
%
%\pt If $R\in\ima\mathbb{R}_-$ (resp. $\ima\mathbb{R}_+$)  then $\dev_0(\beta)$ is asymptotic to $X_{31}$ (resp. $X_{23}$).
%\end{lemma}

\subsubsection{Comparing parallel transports: the case $\re(R)\neq 0$}
As in \S \ref{sec_compa}, using the equilateral frame and two-pointed parallel transport map, the comparison transformation $P_\beta(t):=\a(\beta_{[0,t]}^{-1})\a_0(\beta_{[0,t]})\in \SL(E_0)$ is expression as
$$
P_\beta(t)=\a(0,\beta(t))\,\a_0(\beta(t), 0).
$$

The following proposition is a counterpart to Theorem \ref{thm_compa}.\begin{proposition}\label{prop_rera}
Suppose $\re(R)\neq 0$. Then the limit $P_\beta:=\lim_\tinf P_\beta(t)$ exists and takes the same value $P\in\SL(E_0)$ for any $\beta\in\C$.
\end{proposition}

\begin{proof}
The proof is basically the same as parts (\ref{item_ocompa1}) and (\ref{item_ocompa2}) of  Theorem \ref{thm_compa}, so we only give here a sketch.

The derivative of $P_\beta(t)$ is given by
$
\tfrac{\dif}{\dif t}P_\beta(t)=P_\beta(t)N(t)$, 
where
$$N(t)=\Ad_{\a_0(0,\beta(t))}(A(\dot{\beta}(t))-A_0(\dot{\beta}(t))).
$$ 
Note that $|\beta(t)|\leq t$ because $\beta(0)=0$ and $|\dot{\beta}(t)|\equiv1$. The asymptotic argument $\theta=\lim_\tinf\arg(\beta(t))=\arg(-(R/2)^\frac{1}{3})$
is not an odd multiple of $\tfrac{\pi}{6}$ because $\re(R)\neq 0$. By virtue of fact (\ref{item_eigen1}) in \S \ref{sec_eigen}, the spectral radius of $\Ad_{\a_0(0,\beta(t))}$ is controlled by $e^{(2\sqrt{3}-\delta)t}$ for some $\delta>0$, while  Lemma \ref{lemma_3estimate} implies that $\|A(\dot{\beta}(t))-A_0(\dot{\beta}(t))\|$ is controlled by $e^{-2\sqrt{3}\,t}$.
So $\|N(t)\|$ decays exponentially, hence $P_\beta(t)$ converge by Lemma B.1 in \cite{dumas-wolf}.

In order to prove that $P_{\beta_1}^{-1}P_{\beta_0}=\id$ for any $\beta_0,\beta_1\in \C$, define $\beta_s(t)$ and $Q(s,t)$ in the same way in the proof of Theorem \ref{thm_compa}, so that we have $P_{\beta_1}^{-1}P_{\beta_0}=\lim_\tinf Q(1,t)$ and 
$$
\tfrac{\pa}{\pa s}Q(s,t)=M(s,t)Q(s,t),
$$
where
$$
M(s,t)= \Ad_{\a_0(0, \beta_s(t))}\left(A_0\big(\tfrac{\pa}{\pa s}\beta_s(t)\big)-A\big(\tfrac{\pa}{\pa s}\beta_s(t)\big)\right).
$$
Noting that $\left|\tfrac{\pa}{\pa s}\beta_s(t)\right|=|\beta_0(t)-\beta_1(t)|$ is bounded, Lemma \ref{lemma_3estimate} again implies $\|A(\dot{\beta}(t))-A_0(\dot{\beta}(t))\|=O(e^{-2\sqrt{3}\,t})$, while the spectral radius of $\Ad_{\a_0(0,\beta_s(t)}$ is $O(e^{(2\sqrt{3}-\delta)t})$ for the same reason as above. Therefore $\max_{s\in[0,1]}\|M(s,t)\|$ decays exponentially and Lemma \ref{lemma_ode} gives $\lim_\tinf Q(1,t)=0$ as we wished.
\end{proof}

\subsubsection{Comparing parallel transports: the case $\re(R)= 0$}
Suppose $\re(R)= 0$, so that $\hol_0(\gamma_{\tilde{p}})\in\SL(E_0)$ is planar.

In this case, every $\beta\in\C$ is eventually a ray with asymptotic argument $\frac{7\pi}{6}$ or $\frac{3\pi}{2}$. Such rays are ``unstable'' in the same sense as in the case of higher order poles, that is, the proof for existence of $\lim_\tinf P_\beta(t)$ based on ODE asymptotics (Proposition \ref{prop_rera}) fails. However,  modifying the method, we can still obtain some asymptotic information on $P_\beta(t)$, as stated in the next proposition.

Let $G\subset\SL(E_0)$ denote the subgroup consisting of unipotent elements which commute with $\hol_0(\gamma_{\tilde{p}})$ and fix the point $\devlima(\beta)=X_3$. In terms of \mbox{matrices} under the basis $(e_1(0), e_2(0), e_3(0))$, 
$$
G=
\begin{cases}
\left\{
\begin{pmatrix}
1&&\\
&1&\\
&*&1
\end{pmatrix}
\right\}
&\mbox{ if } R\in\ima\mathbb{R}_-,
\\[25pt]
\left\{
\begin{pmatrix}
1&&\\
&1&\\
*&&1
\end{pmatrix}
\right\}
&\mbox{ if } R\in\ima\mathbb{R}_+.
\end{cases}
$$
Let $\Pi: \SL(E_0)\rightarrow\SL(E_0)/G$ be the projection.

\begin{proposition}\label{prop_rerb}
There exists a $G$-valued function $\Xi(t)$ on $[0,+\infty)$ with the following properties.
\begin{itemize}
\item The only non-vanishing off-diagonal entry of $\Xi(t)$, denoted by $F(t)$, is a primitive of a bounded smooth function. In particular, $|F(t)|=O(t)$.
\item For any $\beta\in\C$, $P_\beta(t)\,\Xi(t)^{-1}$ admits a limit in $\SL(E_0)$,  denoted by $P'_\beta$. 
\end{itemize}
Moreover, the left $G$-coset $\Pi(P'_\beta)\in \SL(E_0)/G$ is independent of $\beta$.
\end{proposition}

%Note that the asymptotic condition on the matrix norm $\|\Xi(t)\|$ is just a condition on the only non-vanishing off-diagonal entry of $\Xi(t)$.

The proof is based on the following generalisation of Lemma B.1 in \cite{dumas-wolf}.
\begin{lemma}\label{lemma_new}
Let $P(t)$ be a $\SL(3,\mathbb{R})$-valued $C^1$-function and $N(t)$ a  $\sl(3,\mathbb{R})$-valued \mbox{continuous} function on $[0,+\infty)$ such that
$\dot{P}(t)=P(t)N(t)$.
Denote the $(i,j)$-entry of $N(t)$ by $N_{ij}(t)$. Let $f(t):=N_{i_0j_0}(t)$ be an off-diagonal entry and $F(t)$ be a primitive of $f(t)$. Suppose
$$
\int_0^{+\infty}\left|N_{ij}(t)\cdot F(t)^l\right|\dif t<+\infty,
$$
for any $(i,j)\neq (i_0,j_0)$ and $l=0,1,2$. Then $P(t)$ is asymptotic to 
$$
\Xi(t):=\exp\big(F(t)E_{i_0j_0}\big)
$$ 
in the sense that $P(t)\,\Xi(t)^{-1}$ admits a limit in $\SL(3,\mathbb{R})$ as $\tinf$.
\end{lemma}

\begin{proof}
We compute the derivative of 
$R(t):=P(t)\,\Xi(t)^{-1}$:
\begin{align*}
\dot{R}=\dot{P}\,\Xi^{-1}-P\,\Xi^{-1}\,\dot{\Xi}\,\Xi^{-1}
%=PN\,\Xi^{-1}-P\,\Xi^{-1}\,\dot{\Xi}\,\Xi^{-1}
=R(\,\Xi N\,\Xi^{-1}-\dot{\Xi}\,\Xi^{-1}).
\end{align*}
A straightforward computation shows that each entry of $\,\Xi N\,\Xi^{-1}-\dot{\Xi}\,\Xi^{-1}$ is a sum of terms of the form $N_{ij}(t)\cdot F(t)^l$ with $(i,j)\neq (i_0,j_0)$ and $l=0,1,2$. Therefore  
$$
\int_0^{+\infty}\|\,\Xi N\,\Xi^{-1}-\dot{\Xi}\,\Xi^{-1}\|\dif t<+\infty
$$
by assumption. Applying Lemma B.1 from \cite{dumas-wolf}, we conclude that $R(t)$ admits a limit.
\end{proof}
\vspace{7pt}
\begin{proof}[Proof of Proposition \ref{prop_rerb}]
We will actually prove the following assertions.
\begin{enumerate}
\item
\label{item_rer1}
For each $\beta\in \C$, there exists $\Xi_\beta(t)\in G$ satisfying the two requirements. 
\item\label{item_rer2}
 Fix a $\Xi_\beta(t)$ as above for each $\beta\in\C$ and set $P'_\beta:=\lim_\tinf P_\beta(t)\,\Xi_\beta(t)^{-1}$. Then $$\Xi_{\beta_0}(t) P_{\beta_1}(t)^{-1}P_{\beta_0}(t)\,\Xi_{\beta_0}(t)^{-1}$$ admits a limit in $G$ for any  $\beta_0,\beta_1\in\C$.
\end{enumerate}
To deduce the proposition from these assertions, note that 
$$
\Xi_{\beta_1}(t) P_{\beta_1}(t)^{-1}P_{\beta_0}(t)\,\Xi_{\beta_0}(t)^{-1}=\Big(
\Xi_{\beta_1}(t)\Xi_{\beta_0}(t)^{-1}\Big)\Big( \Xi_{\beta_0}(t)P_{\beta_1}(t)^{-1}P_{\beta_0}(t)\,\Xi_{\beta_0}(t)^{-1}\Big).
$$
The left-hand side converges to $P'_{\beta_1}{\!\!}^{-1}P'_{\beta_0}$ by  (\ref{item_rer1}), while the first factor on the right-hand side is in $G$ and the second factor  converges to an element in $G$ by (\ref{item_rer2}). Thus, we deduce two consequences: $\Xi_{\beta_1}(t)\Xi_{\beta_0}(t)^{-1}$ converges, and $P'_{\beta_1}{\!\!}^{-1}P'_{\beta_0}$ is contained in $G$. The first consequence implies that we can take the same $G$-valued function $\Xi(t):=\Xi_{\beta_0}(t)$ as $\Xi_\beta(t)$ for every $\beta\in \C$; the second implies the last statement in the proposition. 

\vspace{5pt}

In order to prove the assertions, we define $P_\beta(t)$, $N(t)$, $Q(s,t)$ and $M(s,t)$ in the same way as in the proof of Proposition \ref{prop_rera} and follow the same line of arguments.

First consider the case $R\in\ima\mathbb{R}_+$, so that $\theta=\arg(-(R/2)^\frac{1}{3})=\frac{7\pi}{6}$. 
Looking at the exponential growth rate of each eigenvalue of $\Ad_{\a_0(0,\beta(t))}$ as in the proof of Proposition  \ref{thm_compa} (\ref{item_ocompa3}), we see that the biggest eigenvalue is asymptotic to $e^{2\sqrt{3}\,t}$ and is associated to the  eigenspace $\mathbb{R}E_{31}$, while all the other eigenvalues are controlled by $e^{(2\sqrt{3}-\delta)t}$. Therefore, every entry of $N(t)$ except the $(3,1)$-one decays to $0$ exponentially while $f(t):=N_{31}(t)$ is bounded. We let $F(t)$ be a primitive of $f(t)$ and define
$$
\Xi_\beta(t):=\exp(F(t)E_{31}).
$$ 
Then $P_\beta(t)\,\Xi_\beta(t)^{-1}$ converges by Lemma \ref{lemma_new} as required. Assertion (\ref{item_rer1}) is proved.

Put
$$
Q'(s,t)=\Xi_{\beta_0}(t) Q(s,t)\, \Xi_{\beta_0}(t)^{-1},\quad M'(s,t)=\Xi_{\beta_0}(t) M(s,t)\, \Xi_{\beta_0}(t)^{-1}.
$$
Assertion (\ref{item_rer2}) just says that $Q'(1, t)$ has a limit in $G$ as $\tinf$.  But we have
$$
\tfrac{\pa}{\pa s}Q'(s,t)=M'(s,t)Q'(s,t).
$$

Similarly as in the proof of Proposition \ref{thm_compa} (\ref{item_ocompa3}),
applying Lemma \ref{lemma_ode} to $X=E_{31}$ and $f=M'_{31}$, we see that assertion (\ref{item_rer2}) holds if $M'(s,t)$ fulfills the following conditions:
\begin{enumerate}[(a)]
\item\label{item_m1} 
If $(i,j)\neq (3,1)$ then $|M'_{ij}(s,t)|=O(e^{-\delta t})$  for some $\delta>0$;
\item\label{item_m2} 
$\sup_{t\in[0,+\infty)}\int_0^1|M'_{31}(s,t)|\dif s<+\infty $
\end{enumerate}

Using the expression
$$
\Xi_{\beta_0}(t)=
\begin{pmatrix}
1&&\\
&1&\\
O(t)&&\ \  1
\end{pmatrix},
$$
we see that $M'(s,t)$ satisfies the above conditions if and only if $M(s,t)$ does. 
It is thus sufficient to verify (\ref{item_m1}) and (\ref{item_m2}) for $M(s,t)$ instead of $M'(s,t)$. 

As before, by Lemma \ref{lemma_3estimate} we have 
\begin{equation}\label{eqn_xx}
\|A_0\big(\tfrac{\pa}{\pa s}\beta_s(t)\big)-A\big(\tfrac{\pa}{\pa s}\beta_s(t)\|=O(e^{-2\sqrt{3}\, t}).
\end{equation}

To analyze the effect of $\Ad_{\a_0(0,\beta_s(t))}$, we put $\theta_{s,t}:=\arg(\beta_s(t))$.  With the notations from \S \ref{sec_eigen}, the eigenvalue of $\Ad_{\a_0(0,\beta_s(t))}$ on $E_{ij}$ is 
$$
\lambda_{ij}(s,t):=\exp(\varpi_{ij}\big(\theta_{s,t})|\beta_s(t)|\big).
$$ 
Since $\beta_0(t)$ and $\beta_1(t)$ are parallel rays with asymptotic argument $\frac{7\pi}{6}$ for $t$ sufficiently large, $\theta_{s,t}$ converges to $\frac{7\pi}{6}$ uniformly in $s$. 
Noting that $|\beta_s(t)|\leq t$, fact (\ref{item_eigen2}) in  \S \ref{sec_eigen} implies that if $(i,j)\neq (3,1)$ then
$\lambda_{ij}(s,t)\leq e^{(2\sqrt{3}-\delta)t}$.
Combining this with (\ref{eqn_xx}), we get condition (\ref{item_m1}). 

As for $\lambda_{31}(s,t)$, we have
$$
\varpi_{31}(\theta)=2\re(e^{2\pi\ima/3}e^{\theta\ima}-e^{\theta\ima})=2\big(\cos(\theta+\tfrac{2\pi}{3})-\cos(\theta)\big)=2\sqrt{3}\cos(\theta-\tfrac{7\pi}{6}).
$$
Let $\alpha$ be the ray parallel to $\beta_0$ passing through $0$ and let 
 $y_i$ be the signed distance from $\alpha$ to $\beta_i$  ($i=0,1$). We have
$$
\big|\theta_{s,t}-\tfrac{7\pi}{6}\big|\geq c\,\big|\tan(\theta_{s,t}-\tfrac{7\pi}{6})\big|\approx c\tfrac{(1-s)y_0+sy_1}{t}
$$
for some constant $c>0$ when $t$ is sufficiently large. Therefore,
$$
-2\sqrt{3}+\varpi_{31}(\theta_{s,t})=-2\sqrt{3}(1-\cos(\theta-\tfrac{7\pi}{6}))=-4\sqrt{3}\sin^2\big(\tfrac{\theta_{s,t}-\frac{7\pi}{6}}{2}\big)\leq -c'\big(\tfrac{(1-s)y_0+sy_1}{t}\big)^2
$$
for some $c'$. Combining these with (\ref{eqn_xx}), we get
\begin{align*}
|M_{31}(s,t)|&\leq C\,e^{\lambda_{ij}(s,t)}e^{-2\sqrt{3}\,t}\leq Ce^{(-2\sqrt{3}+\varpi_{31}(\theta_{s,t}))t}\\
&\leq C\exp\left(-\frac{c'}{t}\big[(1-s)y_0+sy_1\big]^2\right).
\end{align*}
The last term, viewed as a family of functions of $s$ parametrized by $t$, converges uniformly to $1$ as $\tinf$. This implies property  (\ref{item_m2}) and concludes the proof of assertion (\ref{item_rer2}) in the case $R\in\ima\mathbb{R}_+$.

The proof for  the case $R\in\ima\mathbb{R}_-$ is almost identical, the only difference being that $\theta_{s,t}$ now tends to $\frac{3\pi}{2}$, hence the fastest-growing eigenvalue  is $\lambda_{32}(s,t)$, as the table in \S \ref{sec_eigen} shows. 
\end{proof}

\subsubsection{Determining the end holonomy}
%The end holonomy $\hol_{\tilde{p}}$ can be recovered from the auxiliary end holonomy $\hol_0(\gamma_{\tilde{p}})$ and the limit of comparison transformation $P_\beta$, as stated in the next lemma. 
%
%
%
%\begin{lemma}\label{lemma_pp}
%Let $\beta,\tilde{\beta}\in\C$ be such that $\tilde{\beta}$ is equivalent to $\beta\cdot\gamma_{\tilde{p}}$. If $\re(R)\neq 0$ then 
%$$\hol_{\tilde{p}}=P_{\beta} \cdot\hol_0(\gamma_{\tilde{p}})\cdot P_{\tilde{\beta}}^{-1}.$$
%If $\re(R)=0$, then 
%$\hol_{\tilde{p}}=P'_{\beta} \cdot\hol_0(\gamma_{\tilde{p}})\cdot P'_{\tilde{\beta}}{\!}^{-1}$.
%\end{lemma}
From the above obtained information on comparison transformations, we can derive an expression of $\hol_{\tilde{p}}$, as stated in the next theorem. It contains Theorem \ref{intro_thm2} from the introduction. 

Here $P$ and $P_\beta'$ are given by Proposition \ref{prop_rera} and Proposition \ref{prop_rerb}. 
\begin{theorem}[\textbf{Holonomy of third order pole}]\label{thm_holo}
${}$
\begin{enumerate}
\item\label{item_coroholo1}
If $\re(R)\neq0$ then the $\lambda_i$'s  are distinct and 
$$
\hol_{\tilde{p}}=P
\begin{pmatrix}
\lambda_1&&\\
&\lambda_2&\\
&&\lambda_3
\end{pmatrix}
P^{-1};
$$
\item\label{item_coroholo2}
If $R\in\ima\mathbb{R}_-$ then $\lambda_1<\lambda_2=\lambda_3$ and there is $a\in\mathbb{R}$ such that for any $\beta\in \C$,
$$
\hol_{\tilde{p}}=P'_\beta
\begin{pmatrix}
\lambda_1&&\\
&\lambda_2&\\
&a&\lambda_2
\end{pmatrix}
 P'_\beta{\!}^{-1}.
$$
\item\label{item_coroholo3}
If $R\in\ima\mathbb{R}_+$ then $\lambda_1=\lambda_3<\lambda_2$ and there is $a\in\mathbb{R}$ such that for any $\beta\in \C$,
$$
\hol_{\tilde{p}}=P'_\beta
\begin{pmatrix}
\lambda_1&&\\
&\lambda_2&\\
a&&\lambda_1
\end{pmatrix}
P'_\beta{\!}^{-1}.
$$
\end{enumerate}
\end{theorem}

Remark that if statements (\ref{item_coroholo2}) and (\ref{item_coroholo3})  hold for one $\beta\in \C$ then they hold for any other $\beta$ as well, because changing $\beta$ amounts to right-multiplying $P'_\beta$ by an element of $G$, but any element of $G$ commutes with 
$$
\begin{pmatrix}
\lambda_1&&\\
&\lambda_2&\\
&a&\lambda_2
\end{pmatrix}
\ \big(\mbox{ resp. }
\begin{pmatrix}
\lambda_1&&\\
&\lambda_2&\\
a&&\lambda_2
\end{pmatrix}
\ \big)
$$
when $R\in\ima\mathbb{R}_-$ (resp. $R\in\ima\mathbb{R}_+$).

\begin{proof}
In this proof we work on the punctured neighborhood $U=\{0<|z|<a\}$ of $p$ directly rather than using the local model $H/\sim$.

%Recall that $\beta_{[0,t]}$ denotes the initial portion of $\beta$ of length $t$ (with respect to the flat metric $|\ve{b}|^\frac{2}{3}$), which goes from the base point $m$ to $\beta(t)$.

Recall that $t\mapsto\beta(t)$ is the parametrization of $\beta$ by arc-length with respect to $|\ve{b}|^\frac{2}{3}$ and we have, by definition,
\begin{equation}\label{eqn_holostudy1}
P_\beta(t)=\a(\beta_{[0,t]}^{-1})\a_0(\beta_{[0,t]}).
\end{equation}

For each $t$, let $\gamma_t$ denote the oriented loop based at $\beta(t)$ such that $\gamma_t$ runs over the circle \mbox{$\{|z|=|\beta(t)|\}$} clockwise. The loop $\beta_{[0,t]}^{-1}\cdot\gamma_t\cdot\beta_{[0,t]}$ is homotopic to $\gamma_{\tilde{p}}$, thus
\begin{equation}\label{eqn_holostudy2}
\hol_0(\gamma_{\tilde{p}})=\a_0(\beta_{[0,t]}^{-1})\a_0(\gamma_t)\a_0(\beta_{[0,t]}).
\end{equation}

Put $\tilde{\beta}:=\beta\cdot\gamma_{\tilde{p}}$. Suppose $\gamma_{\tilde{p}}$ has length $t_0$ with respect to $|\ve{b}|^\frac{2}{3}$. Then $\tilde{\beta}_{[0,t+t_0]}$ is homotopic to $\beta_{[0,t]}\cdot\gamma_t$, thus 
\begin{equation}\label{eqn_holostudy3}
P_{\tilde{\beta}}(t+t_0)^{-1}=\a_0(\tilde{\beta}^{-1}_{[0,t+t_0]})\a(\tilde{\beta}_{[0,t+t_0]})=\a_0(\beta^{-1}_{[0,t]}\cdot\gamma_t^{-1})\a(\gamma_t\cdot\beta_{[0,t]}).
\end{equation}

Concatenating (\ref{eqn_holostudy1}), (\ref{eqn_holostudy2}) and (\ref{eqn_holostudy3}), we obtain, for any $t$,
$$
P_\beta(t)\cdot\hol_0(\gamma_{\tilde{p}})\cdot P_{\tilde{\beta}}(t+t_0)^{-1}=\a(\beta_{[0,1]}^{-1}\cdot\gamma_t\cdot \beta_{[0,1]})=\hol_{\tilde{p}}.
$$
When $\re(R)\neq 0$, letting $t$ tend to $+\infty$ and using Proposition \ref{prop_rera}, we get part (\ref{item_coroholo1}) of the theorem.

 When $\re(R)=0$, we take the $G$-valued function $\Xi(t)$ provided by Proposition \ref{prop_rerb} into account. Members of $G$ commute with $\hol_0(\gamma_{\tilde{p}})$ by definition, hence (\ref{eqn_holostudy3}) yields
\begin{equation}\label{eqn_holostudy4}
P_\beta(t)\,\Xi(t)^{-1}\,\cdot\hol_0(\gamma_{\tilde{p}})\,\Xi(t)\,\Xi(t+t_0)^{-1}\cdot\Xi(t+t_0)P_{\tilde{\beta}}(t+t_0)^{-1}=\hol_{\tilde{p}}.
\end{equation}
The non-vanishing off-diagonal entry of $\Xi(t)$ is a primitive of some bounded function, so the non-vanishing off-diagonal entry of $\Xi(t)\,\Xi(t+t_0)^{-1}$ is bounded. As a result, $\Xi(t)\,\Xi(t+t_0)^{-1}$  converges to some $\Xi_0\in G$ at least along a subsequence $(t_n)\subset [0,+\infty)$. Letting $t$ tend to $+\infty$ along this subsequence in (\ref{eqn_holostudy4}), we get
$$
\hol_{\tilde{p}}=P'_\beta\cdot\hol_0(\gamma_{\tilde{p}})\, \Xi_0\cdot P'_{\tilde{\beta}}{\!}^{-1}.
$$
The last statement in Proposition \ref{prop_rerb} implies $P'_{\tilde{\beta}}=P'_{\beta}\,\Xi_1^{-1}$ for some $\Xi_1\in G$, so we further get 
$$
\hol_{\tilde{p}}=P'_\beta\cdot\hol_0(\gamma_{\tilde{p}})\, \Xi_0\,\Xi_1\cdot P'_{\beta}{\!}^{-1}.
$$
This establishes the required expressions in part (\ref{item_coroholo2}) and (\ref{item_coroholo3}) of the theorem.
\end{proof}
%Clearly, the above lemma is just a consequence of the definitions and does not particularly require $p$ to be a pole of order $3$. However, when $p$ is a pole of order $3$ it yields a full description of $\hol_{\tilde{p}}$ as in the following lemma, which implies Theorem \ref{intro_thm2} from the introduction.

\begin{remark} 
Eq. (\ref{eqn_holostudy3}) has nothing to do with the order of pole, hence it also provides some information on $\hol_{\tilde{p}}$ for higher order poles. Suppose that $p$ is a pole of order $\geq 4$ and let $\beta\in\C$ be a stable path. Taking the limit $\tinf$ in (\ref{eqn_holostudy3}), we see that $\hol_{\tilde{p}}$ is conjugate to $\hol_0(\gamma_{\tilde{p}})\cdot P_\beta^{-1}P_{\tilde{\beta}}$. Assuming $\beta\in\Se_{0,1}$ and applying \mbox{Theorem} \ref{thm_compa} (\ref{item_ocompa3}) consecutively, we can write $P_\beta^{-1}P_{\tilde{\beta}}$ as a product $u_1^-u_1^+u_2^-u_2^+\cdots u_n^-u_n^+$, where $u_k^\pm\in G_k^\pm$. 
\end{remark}

\subsubsection{Proof of Theorem \ref{thm_pole3}}
If $\re(R)\neq 0$
%let $P\in\SL(E_0)$ be as in Proposition \ref{prop_rera} (\ref{item_3compa2}). 
then Corollary \ref{thm_holo} (\ref{item_coroholo1}) says that $\hol_{\tilde{p}}$ is conjugate to $\hol_0(\gamma_{\tilde{p}})$ through $P$, whereas $\devlim(\beta)$ is the translate of $\devlima(\beta)$ by $P$ because 
$
\devlim(\beta)=P_\beta(\devlim_0(\beta))
$
(see \S \ref{sec_prooffinding}). Therefore, the statements about $\devlim(\beta)$ in parts (\ref{item_pole33}) and (\ref{item_pole34}) of Theorem \ref{thm_pole3} are consequences of Corollary \ref{coro_pole3}.

\vspace{6pt}

If $\re(R)\in\ima\mathbb{R}^*$, the above argument fails because $P_\beta=\lim P_\beta(t)$ does not make sense. But we can use the $\Xi(t)$ provided by Proposition \ref{prop_rerb} to remedy it.
Using the expressions
$$
\dev_0(\beta_{[0,t]}^{-1})=
\begin{bmatrix}
e^{2\re(\beta(t))}\\[3pt]
e^{2\re(\omega^2\beta(t))}\\[3pt]
e^{2\re(\omega\beta(t))}
\end{bmatrix},\quad
\Xi(t)=
\begin{pmatrix}
1&&\\
&1&\\
O(t)&&\ 1
\end{pmatrix}
\mbox{ or }
\begin{pmatrix}
1&&\\
&1&\\
&O(t)&\ 1
\end{pmatrix}
$$
and,  for $t$ sufficiently large,
$$
\beta(t)=
\begin{cases}
e^{\frac{3\pi\ima}{2}}t+\zeta_0&\mbox{ if }R\in\ima\mathbb{R}_-,\\[4pt]
e^{\frac{7\pi\ima}{6}}t+\zeta_0&\mbox{ if }R\in\ima\mathbb{R}_+,
\end{cases}
$$
we see that 
$$
\lim_\tinf\Xi(t)\dev_0(\beta_{[0,t]}^{-1})=X_3
$$
%Since $\dev_0(\beta_{[0,t]}^{-1})$ converges exponentially fast to $\devlima(\beta)$ while $\Xi(t)$ has sub-linear growth,  $\Xi(t)\dev_0(\beta_{[0,t]}^{-1})$ converges to $\devlim(\beta)$ as well. 
(a short reasoning: $\dev_0(\beta_{[0,t]}^{-1})$ converges exponentially fast to $\devlima(\beta)=X_3$, hence \mbox{operating} it by the sub-linearly growing $\Xi(t)$ does not effect the limit).
Therefore,
\begin{align*}
\devlim(\beta)&=\lim_\tinf P_\beta(t)\dev_0(\beta_{[0,t]}^{-1})=\lim_\tinf P_\beta(t)\,\Xi(t)^{-1} \lim_\tinf \Xi(t)\dev_0(\beta_{[0,t]}^{-1})\\
&=P'_\beta(X_3).
\end{align*}
But this is exactly the double fixed point of $\hol_{\tilde{p}}$ as Theorem \ref{thm_holo}  implies.

%In view of the matrix expression of $\hol_{\tilde{p}}$ in Theorem \ref{thm_holo},  is the double fixed point if $\hol_{\tilde{p}}$ is quasi-hyperbolic and lies on the fixed line if $\hol_{\tilde{p}}$ is planar.

\vspace{6pt}

It remains to prove the statements about the type of end.

In the case $R\in\mathbb{R}_+$, $\hol_{\tilde{p}}$ is hyperbolic and the $\devlim(\beta)$'s form the interior of the \mbox{principal} segment $I$, but a case-by-case check for simple ends and non-simple ends (see Proposition \ref{prop_classification} and Lemma \ref{lemma_rough}) shows that the only situation where $\devbd{\tilde{p}}$ contain a point in the interior of $I$ is when $\devbd{\tilde{p}}=I$.

In the case $\re(R)<0$, where $\hol_{\tilde{p}}$ is hyperbolic and $\devlim(\beta)$ is the saddle fixed point, Proposition \ref{prop_classification} (\ref{item_endclas3}) implies that $p$ is a V-end.

In the remaining cases, if $\hol_{\tilde{p}}$ is planar, of course the type of end is immediately determined by Lemma \ref{lemma_planar}. Otherwise, we study in addition  the asymptotic direction of the path $t\mapsto\dev(\beta_{[0,t]}^{-1})$, obtaining the following proposition.
%Here, given a projective transformation $a\in\SL(3,\mathbb{R})$ fixing a point $x\in\mathbb{RP}^2$, by an \emph{invariant open cone} in $\T_x\mathbb{RP}^2$ with respect to $a$, we   mean a minimal open sector in $\T_x\mathbb{RP}^2$ based at $0$ preserved by the action of $a$.  We are only concerned with the following  cases:
%\begin{itemize}
%\item $a$ is hyperbolic 
%\item $a$ is quasi-hyperbolic and $x$ is the double fixed point.
%\item $a$ is planar and $x$ is on the fixed line.
%\end{itemize}
%In each case, there are exactly two lines passing through $x$ which are stabilized by $a$, giving rise to two one-dimensional subspaces of $\T_x\mathbb{RP}^2$. An invariant open cone is just one of the four sectors delimited by both subspaces.

\begin{proposition}
Suppose $\re(R)\geq0$, $\im(R)\neq0$ and assume that $\hol_{\tilde{p}}$ is not planar. Let $I$ denote the principal segment of $\hol_{\tilde{p}}$. Then  for any $\beta\in\C$, the path $t\mapsto\dev(\beta_{[0,t]}^{-1})$ converges to an endpoint of $I$ and is asymptotic to $I$.
\end{proposition}
\begin{proof}
The proof is based on the following simple fact: if $t\mapsto \xi(t)$ is a parametrized curve in $\mathbb{RP}^2$ converging to $x_0\in\mathbb{RP}^2$ and asymptotic to a ray $l$ issuing from $x_0$, and $a(t)$ is a continuous $\SL(3,\mathbb{R})$-valued function converging to $a_0\in\SL(3,\mathbb{R})$, then the curve $t\mapsto a(t)(\xi(t))$ converges to $a_0(x_0)$ and is asymptotic to $a_0(l)$.

First consider the case $\re(R)>0$, $\im(R)> 0$, so that $\arg(-(R/2)^\frac{1}{3})\in(\pi,\frac{7\pi}{6})$. Proposition \ref{prop_tit} (\ref{item_tit1}) says that  $t\mapsto\dev_0(\beta_{[0,t]}^{-1})$ is asymptotic to $X_{23}$ while converging to $X_3$. By the above mentioned fact, $t\mapsto \dev(\beta_{[0,t]}^{-1})=P_\beta(t)(\dev(\beta_{[0,t]}^{-1}))$ converges to $P(X_3)$ and is asymptotic to $P(X_{23})$. But the matrix expression of $\hol_{\tilde{p}}$ in Theorem \ref{thm_holo} and the table in \S \ref{sec_devbdrerneqzero} shows that $P(X_{23})$ is a segment joining the attracting and repelling fixed points of $\hol_{\tilde{p}}$, hence can only be the principal segment (otherwise there could not be a path in $\Omega$ asymptotic to it).

The proof for the case $\re(R)>0$, $\im(R)< 0$ is similar.

Let us treat the case $R\in\ima\mathbb{R}_-$.
%$\hol_{\tilde{p}}$ is quasi-hyperbolic and $\devlim_0(\beta)$ is the double fixed point $[e_3(0)]$ for any $\beta\in\C$. We only need to show that $\dev_0(\beta)$ is asymptotic to some a segment 
Using the expressions of $\dev_0(\beta_{[0,t]}^{-1})$ and $\Xi(t)$ exhibited earlier, we not only see that the path
$$
t\mapsto\xi_\beta(t):=\Xi(t)\dev_0(\beta_{[0,t]}^{-1})
$$
converges to $X_3$, but also see that it is asymptotic to $X_{31}$. Noting that
$$
\dev(\beta_{[0,t]}^{-1})=P_\beta(t)\,\Xi(t)^{-1} (\xi_\beta(t)),
$$
we apply the above mentioned fact  and conclude that $t\mapsto\dev(\beta_{[0,t]}^{-1})$ is \mbox{asymptotic} to $P'_\beta(X_{31})$. The latter is a segment  joining the two fixed points of $\hol_{\tilde{p}}$ as Theorem \ref{thm_holo} implies, hence it can only be the principal segment.
\end{proof}

In view of this proposition, the following lemma implies that 
%$p$ is a geodesic end when $\hol_{\tilde{p}}$ is hyperbolic and $\devlim(\beta)$ is the attracting/repelling fixed point (\ie the case $\re(R)>0$, $\im(R)\neq 0$) or when $\hol_{\tilde{p}}$ is quasi-hyperbolic (\ie the case $\re(R)=0$), 
$p$ is a geodesic end in the remaining cases, completing the discussion of third order poles.

\begin{lemma}
Let $(\dev,\hol)$ be a convex projective structure on $\Sigma$ such that $\hol_{\tilde{p}}$ is either hyperbolic or quasi-hyperbolic. Let $I$ be the principal segment of $\hol_{\tilde{p}}$.
Suppose there exists a path $t\mapsto\beta(t)\in\Sigma$ issuing from $m$ and converging to $p$, such that $t\mapsto\dev(\beta_{[0,t]}^{-1})$ converges to an endpoint of $I$ and is asymptotic to $I$. Then $p$ is a geodesic end.
\end{lemma}

We remark that, in contrast to the situation described in the lemma, when $p$ is a V-end or a non-simple end with hyperbolic holonomy (\cf Lemma \ref{lemma_rough}), one can still construct $\beta$ such that $x=\devlim(\beta)$ is an endpoint of the principal segment --  just construct the developed path $t\mapsto\dev(\beta_{[0,t]}^{-1})$ first and then take the quotient by $\pi_1(\Sigma)$. However, $\dev(\beta_{[0,t]}^{-1})$ should 
be asymptotic to the segment joining $x$ and the saddle fixed point, otherwise $\beta$ does not converges to $p$. 

\begin{proof}
By definition of $\devbd{\tilde{p}}$, it is sufficient to prove that, for any punctured neighborhood $U$ of $p$, the closure of $\dev(\widetilde{U})$ contains $I$.

Let $C$ denote the trajectory of $\dev(\beta_{[0,t]}^{-1})$ in $\dev(\widetilde{U})$. The assumption that $C$ is asymptotic to $I$ and dynamical properties of (quasi-)hyperbolic projection transformations imply that either the accumulation set $\underset{k\rightarrow +\infty}{\Ac}(\hol_{\tilde{p}}^k(C))$ or $\underset{k\rightarrow -\infty}\Ac(\hol_{\tilde{p}}^k(C))$ (see \S \ref{subsubsec_devbd} for the notation) is $I$. Since $\dev(\widetilde{U})$ is $\hol_{\tilde{p}}$-invariant, 
the $\hol_{\tilde{p}}^k(C)$'s are contained in $\dev(\widetilde{U})$, hence the accumulation set $I$ is in the closure of $\dev(\widetilde{U})$, as required.
\end{proof}

\vspace{10pt}
\subsection{Poles of order $\leq 2$}
The following theorem completes the proof of statement (III) from the introduction, hence finishes the proof of Theorem \ref{intro_main}.
\begin{theorem}[\textbf{Loftin, Benoist-Hulin}]
Suppose $p$ is a pole of order $\leq 2$ or a removable singularity of $\ve{b}$. Then $p$ is a cusp of the convex projective structure.
\end{theorem}
This theorem was proved by Loftin \cite{loftin_compactification} and Benoist-Hulin \cite{benoist-hulin} using \mbox{different} methods. We outline both (short) proofs for convenience of the reader. Note that \cite{loftin_compactification} only treats the case where $\Sigma$ has negative Euler characteristic, so we add to it a discussion for annuli.

\begin{proof}[First proof]
Lofin \cite{loftin_compactification} showed that if $p$ is pole of order $\leq 2$ then the eigenvalues of $\hol_{\tilde{p}}$ are all $1$ (in fact, he proved Theorem \ref{intro_thm2} from the introduction for all poles of order $\leq 3$ in a unified way, the result that we quote here being the $R=0$ case), so $\hol_{\tilde{p}}$ is parabolic by the classification in \S \ref{sec_auto}. 

Let us show that $p$ is a cusp of the convex projective structure. A case-by-case check for simple ends and non-simple ends (see Proposition \ref{prop_classification} and Lemma \ref{lemma_rough}) shows that the only situation where $\hol_{\tilde{p}}$ is parabolic and $p$ is not a cusp is case (\ref{item_rough3}) of Lemma \ref{lemma_rough}. Suppose by contradiction that this is the case, so that we can assume that $\Sigma=\mathbb{C}^*$, $p=0$ and the other puncture $\infty$ is a cusp. But $\infty$ is a pole of order $\geq 4$ because the cubic canonical  bundle $K^3$ of $\mathbb{CP}^1$ has degree $-6$. Our earlier result says that $\infty$ is a polygonal end, a contradiction.
\end{proof}

\begin{proof}[Second proof]

A straightforward computation shows that a punctured neighborhood $U$ of $p$ has finite volume with respect to the metric $|\ve{b}|^\frac{2}{3}$ if and only if $p$ is a removable singularity or a pole of order $\leq 2$. In this case, $U$ has finite volume with respect to the Hilbert metric as well by Corollary \ref{coro_bh} (\ref{item_bh1}). But ends with finite Hilbert volume are exactly cusps as shown in \cite{marquis}.
\end{proof}

%\section{Comments and questions}
%\subsection{Moduli spaces}
%
%
%\subsection{Duality}
%An important  property of Wang's equation (\ref{eqn_vortex}), which we have not made use of before, is that if the pair $(g,\ve{b})$ satisfies the equation then for any $\theta\in\mathbb{R}$, the pair $(g,e^{\theta\ima}\ve{b})$ satisfies the equation as well. This yields a natural action of the circle $S^1=\{e^{\theta\ima}\}$ on $\mathcal{W}(\Sigma)\cong\P(\Sigma)$.
%
%Given a developing pair $(\dev,\hol)$ of the convex projective structure corresponding to $(g,\ve{b})\in\mathcal{W}(\Sigma)$
%
% $(g,\ve{b})$  and that corresponding to $(g,e^{\theta\ima}\ve{b})$

%\subsection{Classification of cusps}
%This

%\subsection{Projective structures on infinite volume hyperbolic surfaces}

\appendix
\section{Estimates for solutions to the equation $\Delta u=4e^u-4e^{-2u}$}
In this appendix, we study bounded non-negative smooth  solutions to the PDE
\begin{equation}\label{eqn_eqn}
\Delta u=4e^u-4e^{-2u}
\end{equation}
%subject to the bounds
% \begin{equation}\label{eqn_eqnbound}
% 0\leq u\leq \frac{1}{2}.
% \end{equation}

 We consider those $u$ defined first on a disk, then on a half-plane, and finally on a certain flat surface $X$ obtained by gluing three half-planes. $X$ appears as a part of the local model for $(\Sigma,\ve{b})$ constructed in \S \ref{sec_local}. In each setting, $\Lap{}=4\,\paz\pabz$ is the usual Laplacian with respect to the natural coordinate $z$ on each half-plane. 

Many results here are adaptations of estimates in \S 5.4 of \cite{dumas-wolf} despite of the different normalisations. However, in \emph{loc. cit.} only those $u$ bounded by an explicit positive constant (in our normalisation of the equation, the bound is $\frac{1}{2}$) are studied because the bound is need when .... 

Put $D_r:=\{\zeta\in\mathbb{C}\mid |\zeta|<r\}$.
\begin{lemma}\label{lemma_loglambda}
For any $\lambda>1$ and $r>0$, if $u\in C^2(D_r)\cap C^0(\overline{D}_r)$ satisfies Eq.(\ref{eqn_eqn}) in $D_r$ and has the bounds
$$
0\leq u\leq \log\left(\frac{\lambda^3-1}{4\lambda}r^2+\lambda\right)
$$ 
then 
$$
u(0)\leq \log \lambda.
$$
\end{lemma}
\begin{proof}
Put $\mu=\frac{\lambda^3-1}{4\lambda}$.  A straightforward computation shows that  
$$
U(z):=\log(\lambda+\mu |z|^2) 
$$
is a supersolution to Eq.(\ref{eqn_eqn}), \ie
\begin{equation}\label{eqn_supsol}
\Delta U\leq e^U-e^{-2U}.
\end{equation}

We now use the maximum principle to prove $u\leq U$ on $D_r$, which implies the required result.  By assumption, we have $u\leq U$ on $\pa D_r$. Assume by contradiction that $u>U$ somewhere in $D_r$, then $u-U$ takes positive maximum at some point $z_0\in D_r$. By (\ref{eqn_supsol}) and the strict monotonicity of the function $x\mapsto e^x-e^{-2x}$, at the point $z_0$ we have 
$$
0\geq \Delta (u-U)=e^u-e^{-2u}-\Delta U\geq \big(e^u-e^{-2u}\big)-\big(e^U-e^{-2U}\big)>0,
$$
a contradiction.
\end{proof}

\begin{lemma}\label{lemma_a0}
There exists a constant $C$ such that for any $r\geq 1$ and  any function $u\in C^2(D_r)\cap C^0(\overline{D}_r)$ satisfying (\ref{eqn_eqn}) in $D_r$, we have 
$$
u(0)\leq Cr^\frac{1}{2}e^{-2\sqrt{3}r}.
$$
\end{lemma}
\begin{proof}
A key property of the right-hand side of Eq.(\ref{eqn_eqn}) is that, viewed as a function in $u\geq 0$, its second-order Taylor expansion has non-negative remaining terms:
$$
4e^u-4e^{-2u}-(12u-6u^2)>0 \mbox{ for any }u\geq 0.
$$
As a result, for any non-negative function $v$ on $D$ satisfying
\begin{equation}\label{eqn_v}
\Lap{}v\leq 12v-6v^2 \mbox{ in }D,\ v= \frac{1}{2}  \mbox{ on }\pa D,
\end{equation} 
the maximum principle implies $u\leq v$ on the whole $\overline{D}$: otherwise, $u-v$ is non-positive on $\pa D$ but positive somewhere in $D$, hence takes positive maximum at a point $\zeta_0\in D$. But at $\zeta_0$ we have
$$
\Lap{}(u-v)\geq 4e^{u}-4e^{-2u}-(12v-6v^2)> 4e^{v}-4e^{-2v}-(12v-6v^2)\geq 0,
$$
contradicting the maximality.

It is therefore sufficient to find, for every $r\geq 1$, a function $v$ as above which fulfills the required bound
\begin{equation}\label{eqn_v0}
v(0)\leq C r^\frac{1}{2}e^{-2\sqrt{3}r}.
\end{equation}

To this end, consider the modified Bessel function of the first kind
\begin{equation}\label{eqn_bessel}
I_0(x)=\frac{1}{\pi}\int_0^\pi e^{x\cos\theta}\dif\theta=\frac{e^{x}}{\sqrt{2\pi x}}(1+O(x^{-1})).
\end{equation}
The positive function $h\in C^2(\mathbb{C})$ defined by
$$
h(\zeta)=\frac{I_0(2\sqrt{3}\,|\zeta|)}{I_0(2\sqrt{3} \,r)}
$$
satisfies 
$$
\Delta h=12 h
$$
and we have $h=1$ on $\pa D$ and $h< 1$ on $D$. We take $v$ to be
$$
v=h-\frac{1}{2}h^2.
$$
A computation shows that
$$
\Delta v-(12v-6v^2)=-|\nabla h|^2+6h^3\left(\frac{h}{4}-1\right)\leq 0
$$
on $D$. Thus $v$ satisfies the requirement (\ref{eqn_v}), whereas the bound (\ref{eqn_v0})
follows from the asymptotic expansion of $I_0$ (the second equality in (\ref{eqn_bessel})). 
\end{proof}

The next lemma is proved with the same method.
\begin{lemma}\label{lemma_a1}
Let $H\subset\mathbb{C}$ be a half-plane (\ie a region whose boundary is a straight line) and let $u\in C^2(H)\cap C^0(\overline{H})$ be a function satisfying (\ref{eqn_eqn}) 
%...and (\ref{eqn_eqnbound}) 
in $H$. 
\begin{enumerate}
%\item\label{item_a11}
%There exists a constant $C$ such that
%$$
%u(\zeta)\leq C\, \dist(\zeta,\pa H)^\frac{1}{2}\, e^{-2\sqrt{3}\,\dist(\zeta,\pa H)}
%$$
%for any $z\in H$ with $\dist(z,\pa{H})\geq 1$.
%Here $\dist(\zeta,\pa{H})$ is the distance from $\zeta$ to $\pa H$, measured by the metric $|\dzeta|^2$.
\item\label{item_a12}
If $u$ is invariant under a translation $\zeta\mapsto\zeta+a$ which preserves $H$, then 
$$
u(\zeta)\leq C e^{-2\sqrt{3}\,\dist(\zeta,\pa H)} 
$$
for any $\zeta\in H$ and some constant $C$ independent of $u$.
\item\label{item_a13}
If $u|_{\pa H}$ is integrable and $\lim_{|\zeta|\rightarrow+\infty}u(\zeta)=0$, then there exists a constant $C''$ only depending on the upper bound of $u$ and the $L^1$ norm of $u|_{\pa H}$, such that
$$
u(\zeta)\leq C' \dist(\zeta,\pa H)^{-\frac{1}{2}}e^{-2\sqrt{3}\,\dist(\zeta,\pa H)}
$$
for any $\zeta \in H$ and some constant $C'$ only depending on the integral of $u$ on $\pa H$.
\end{enumerate}
\end{lemma}

\begin{proof}
We can assume without loss of generality that $H$ is the right half-plane $\H=\{\zeta\in\mathbb{C}\mid \re(\zeta)>0\}$, so that $\dist(\zeta,\pa \H)=\re(\zeta)$.  

\vspace{5pt}

(\ref{item_a12})
Set $h(\zeta)=e^{-2\sqrt{3}\re(\zeta)}$ and $v=h-\frac{1}{2}h^2$. Note that $v=\frac{1}{2}\geq u$ on $\pa \H$. The same computation as at the end of the previous proof yields
$$
\Delta\, v\leq 12 v-6v^2.
$$
Let us prove $u\leq v$ on $\H$ by applying the maximum principle similarly as in the previous proof. Assume by contradiction that this is not the case, then $u-v>\epsilon>0$ somewhere in $\H$.  Applying Lemma \ref{lemma_a0} to disks of the form 
$$
\{\zeta\in\mathbb{C}\mid |\zeta-\zeta_0|<|\zeta_0|\},
$$
we see that $u(\zeta_0)-v(\zeta_0)<\epsilon$ whenever $\re(\zeta_0)\geq M$ for a big enough $M>0$. On the other hand, by hypothesis, $u-v$ is invariant under a translation of the form $\zeta\mapsto \zeta+b\ima $ ($b>0$). As a result, $u-v$ takes positive maximum at some point $\zeta_0$ in the rectangle 
$
\{\zeta\in\mathbb{C}\mid 0\leq \re(\zeta)\leq M, \ 0\leq \im (\zeta)\leq b \}
$ 
and this contracts the fact that at $\zeta_0$ we have
$$ 
\Lap{}(u-v)\geq 4e^{u}-4e^{-2u}-(12v-6v^2)> 4e^{v}-4e^{-2v}-(12v-6v^2)\geq 0,
$$

\vspace{5pt}

(\ref{item_a13})
Let 
$$
K_1(x)=\int_0^{+\infty}e^{-x\ch(t)}\ch(t)\dif t=\sqrt{\frac{\pi}{2x}}\,e^{-x}(1+O(x^{-1}))
$$
be the modified Bessel function of the second kind. 
As shown in the proof of Lemma 5.8 in \cite{dumas-wolf}, the function $h\in C^2(\H)\cap C^0(\CH)$ defined by 
$$
h(\zeta)=2\int_{\pa\H}\frac{2\sqrt{3}\re(\zeta)}{\pi|\zeta-\xi|}K_1(2\sqrt{3}\,|\zeta-\xi|)u(\xi)\dif\xi,
$$
is a solution to the Dirichlet problem 
$$
\Lap{}h=12h,\quad h|_{\pa\H}=2\,u|_{\pa\H}.
$$ 
Put  $v=h-\frac{1}{2}h^2$ as before. On the boundary $\pa\H$ we have 
$v=2u(1-u)\geq u$ because of the hypothesis $0\leq u\leq \frac{1}{2}$. A similar application of the maximum principle as before shows that $u\leq v$ in $\H$ as well (the fact that $\lim_{|\zeta|\rightarrow+\infty}u-v=0$ is crucial here). Now the required inequality follows from the above expression of $h$ and the asymptotics of $K_1$.
\end{proof}

Given $\xi', \xi''\in\pa\H$ with $\im(\xi'')<0<\im(\eta')$, we let $X$ denote the surface with boundary obtained by gluing the half-planes $\overline{\H'}=e^{2\pi\ima/3}\overline\H$ and $\overline{\H''}=e^{4\pi\ima/3}\overline\H$  to $\CH$ via translations, as illustrated by Figure \ref{figure_s}, such that $\pa\H'$ and $\pa \H''$ meet $\pa\H$ at $\xi'$ and $\xi''$, respectively. 

We view $\CH$, $\overline{\H'}$ and $\overline{\H''}$ as subsets of $X$.
There is a natural projection $X\rightarrow\mathbb{C}$ mapping $\CH\subset X$ identically to $\CH\subset\mathbb{C}$. A disk/half-plane in $X$ is by definition a subset of $X$ whose projection to $\mathbb{C}$ is a disk/half-plane.

\begin{figure}[h]
\centering\includegraphics[width=4.7in]{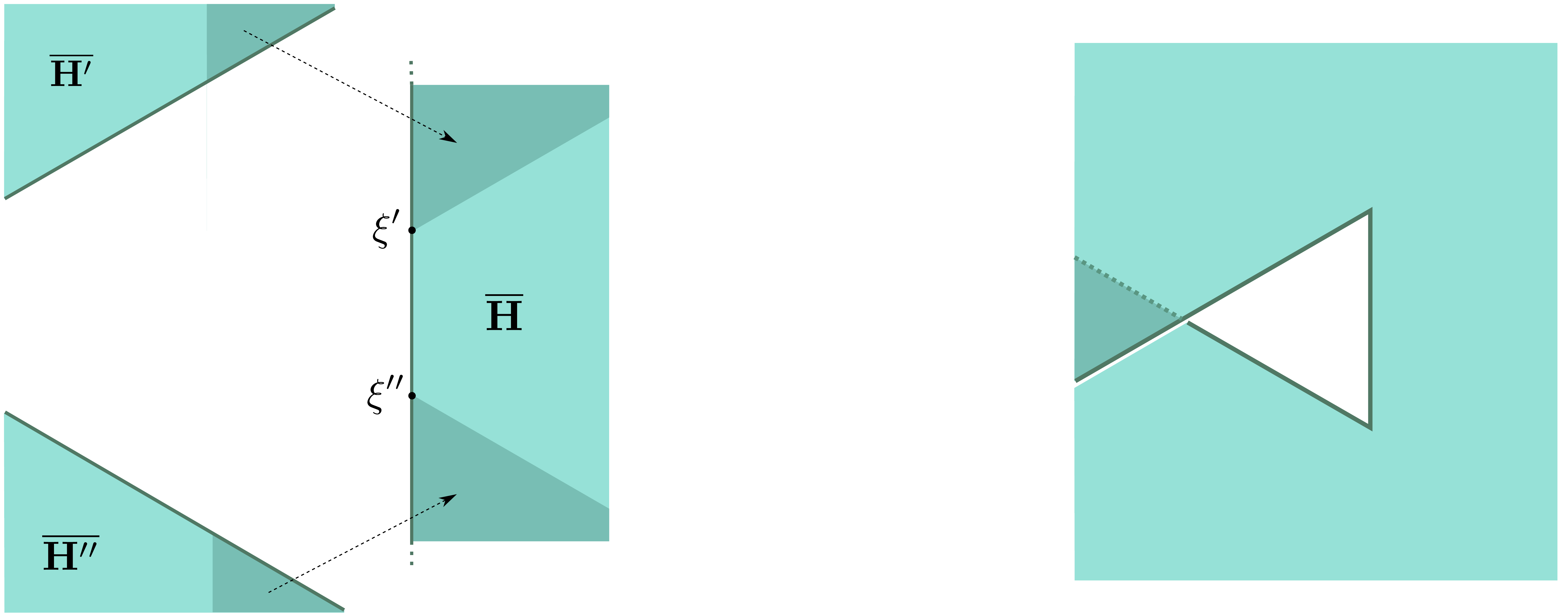}
\caption{The surface surface $X$}
\label{figure_s}
\end{figure}

\begin{lemma}\label{lemma_a2} 
There exists a constant $C>0$ such that if $u$ is a $C^2$ function defined in the interior of $X$ satisfying (\ref{eqn_eqn}) 
%...and (\ref{eqn_eqnbound})
and $u$ extends continuously to $\pa X$,  then
$$
u(\zeta)\leq C|\zeta|^{-\frac{1}{2}}e^{-2\sqrt{3}|\zeta|}
$$
for any $\zeta\in \H\subset X$.
\end{lemma}

\begin{proof}
Let $X'\subset X$ denote the union of $\CH$ with the two half-planes in $X$ which project to $\{\zeta\in\mathbb{C}\mid\im(\zeta)\geq \im(\xi')\}$ and $\{\zeta\in\mathbb{C}\mid\im(\zeta)\leq \im(\xi'')\}$, respectively. The projection $X\rightarrow\mathbb{C}$ is injective on $X'$, hence we represent a point in $X'$ by  the coordinate $\zeta$ of its projection.

We first show that $u(\zeta)$ decays exponentially as $\zeta$ tends to infinity in $X'$.

Let $D_\zeta$ denote the biggest disk in $X$ centered at $\zeta\in X'$ and let $r(\zeta)$ be the radius of $D_\zeta$. Clearly, there is a constant $c>0$ only depending on $\xi'$ and $\xi''$ such that
$$
r(\zeta)\geq \frac{1}{2}(|\zeta|-c).
$$
Since the function $x\mapsto x^\frac{1}{2}e^{-2\sqrt{3}\,x}$ is monotonically decreasing when $x\geq\frac{1}{4\sqrt{3}}$, applying  Lemma \ref{lemma_a0} to $D_\zeta$, we get
$$
u(\zeta)\leq C r(\zeta)^\frac{1}{2}e^{-2\sqrt{3}\,r(\zeta)}\leq  C \left(\frac{|\zeta|-c}{2}\right)^\frac{1}{2}e^{-\sqrt{3}(|\zeta|-c)}
$$
whenever $|\zeta|\geq 2+c$. Hence $u(\zeta)$ decays exponentially for $\zeta\in X'$ tending to infinity.

To prove the lemma, we consider, for each $\theta\in[-\frac{\pi}{2},\frac{\pi}{2}]$, the biggest half-plane $\H_\theta\subset X$ whose boundary intersects the ray $e^{\ima\theta}\mathbb{R}_{\geq 0}\subset \CH$ orthogonally. Note that the boundary $\pa \H_\theta$ is contained in $X'$. The exponential decay property we just established implies that the restriction of $u$ to $\H_\theta$ satisfies the assumptions of Lemma \ref{lemma_a1}, part (\ref{item_a13}), thus we get
\begin{equation}\label{eqn_a2proof1}
u(\zeta)\leq C'\,\dist(\zeta, \pa\H_\theta)^{-\frac{1}{2}}e^{-2\sqrt{3}\,\dist(\zeta, \pa\H_\theta)}
\end{equation}
for  $\zeta\in \H_\theta$. Now, for any $\zeta\in\CH$ with 
$$
|\zeta|\geq d:=\max_{\theta\in[-\frac{\pi}{2},\frac{\pi}{2}]}{\dist(0,\pa\H_{\theta})},$$
 we have $\zeta\in\CH_{\arg(\zeta)}$ and
\begin{equation}\label{eqn_a2proof2}
 |\zeta|=\dist(\zeta, \pa\H_{\arg(\zeta)})+\dist(0,  \pa\H_{\arg(\zeta)})\leq \dist(\zeta, \pa\H_{\arg(\zeta)})+d.
\end{equation}
Combining (\ref{eqn_a2proof1}) and (\ref{eqn_a2proof2}), we get the required estimate for those $\zeta\in \overline\H$ satisfying $|\zeta|\geq d$. Then the estimate holds for any $\zeta\in \overline\H$ because the ration between $u(\zeta)$ and $|\zeta|^{-\frac{1}{2}}e^{-2\sqrt{3}|\zeta|}$ has a upper bound on $\{\zeta\in \overline\H\mid |\zeta|\leq d\}$.
\end{proof}

See \eg \cite{jost} Corollary 1.2.7 for a proof the following well known result.
\begin{lemma}[\textbf{Gradient estimate for the Poisson equation}]\label{lemma_gradientpoisson}
For any $r>0$, there exists a constant $A$ such that 
any $C^2$ function $u$ defined on a neighborhood of the disk $D_r(z_0)=\{z\in\mathbb{C}\mid |z-z_0|\leq r\}$ satisfies
$$
|\pa_zu(z_0)|\leq A\left(\sup_{\pa D_r(z_0)}u-\inf_{\pa D_r(z_0)}u+ \sup_{D_r(z_0)}|\Lap{}u|\right)
$$
\end{lemma}
By virtue of the lemma,  the estimates on $u$ in Lemma \ref{lemma_a0}$, \ref{lemma_a1}$ and \ref{lemma_a2} immediately yield estimates on the derivative $\pa_\zeta u$.
\begin{corollary}\label{coro_a}
In each of the estimates in Lemma  \ref{lemma_a0}, \ref{lemma_a1} and Lemma \ref{lemma_a2}, one can take the constant to be big enough such that $|\pa_\zeta u|$ satisfies the same estimate.
\end{corollary}
%\begin{proof}
%We only prove for Lemma \ref{lemma_a2}. Adapting the proof to Lemma \ref{lemma_a1} is just a matter of notations.
%
%It is sufficient to show that there is a constant $M$ such that 
%$|\pa_\zeta u(\zeta)|\leq C |\zeta|^{-\frac{1}{2}}e^{-2\sqrt{3}\,|\zeta|}$ whenever $|\zeta|\geq M$. We take 
%
%\end{proof}
This is proved by applying Lemma \ref{lemma_gradientpoisson} to disks of fixed radius centered at each $\zeta$, taking account of the fact that 
$$
0\leq \Delta u=4e^u-4e^{-2u}\leq ku
$$ 
for a constant $k$ because of the hypothesis $0\leq u\leq \frac{1}{2}$.

\bibliographystyle{amsalpha} \bibliography{mero}

\end{document}